\documentclass[11pt]{amsart}
\usepackage[dvipsnames]{xcolor}
\usepackage{amsfonts,amsthm,amsmath,amssymb,etoolbox,comment,graphicx,array,tabu,commath,latexsym,environ,tikz,tikz-cd,stackengine,diagbox,caption}
\usepackage{subcaption}
\usepackage[ twoside,top=1in, bottom=1in, left=1in, right=1in]{geometry}
\usepackage{mathrsfs}
\usepackage{soul}
\usepackage{eso-pic}
\usepackage[toc,page]{appendix}
\usepackage{enumitem}
\usepackage{mathtools}

\usepackage[colorlinks, linkcolor = black, anchorcolor = black, citecolor = black, urlcolor=black]{hyperref}

\newcommand{\fg}{\mathfrak{g}}

\newcommand{\cC}{\mathcal{C}}

\newcommand{\cF}{\mathcal{F}}

\newcommand{\cH}{\mathcal{H}}

\newcommand{\cN}{\mathcal{N}}
\newcommand{\cO}{\mathcal{O}}

\newcommand{\cR}{\mathcal{R}}
\newcommand{\cS}{\mathcal{S}}
\newcommand{\cT}{\mathcal{T}}

\newcommand{\cV}{\mathcal{V}}

\newcommand{\cW}{\mathcal{W}}

\newcommand{\cZ}{\mathcal{Z}}
\usepackage[normalem]{ulem}

\newcommand{\bB}{\mathbf{B}}
\newcommand{\bC}{\mathbf{C}}

\newcommand{\bF}{\mathbf{F}}

\newcommand{\bI}{\mathbf{I}}

\newcommand{\bL}{\mathbf{L}}
\newcommand{\bM}{\mathbf{M}}

\newcommand{\bg}{\mathbf{g}}

\newcommand{\R}{\mathbb R}

\newcommand{\Z}{\mathbb Z}
\newcommand{\N}{\mathbb N}

\newcommand{\Zc}{\mathcal Z}

\newcommand{\Fc}{\mathcal F}

\newcommand{\Fb}{\mathbf F}
\newcommand{\Mb}{\mathbf M}
\newcommand{\id}{\mathrm {id}}

\newcommand{\length}{\mathrm{length}}
\newcommand{\area}{\operatorname{area}}

\newcommand{\dist}{\operatorname{dist}}

\newcommand{\rad}{\mathrm{rad}}

\newcommand{\diam}{\mathrm{diam}}

\renewcommand{\tilde}{\widetilde}

\newcommand{\spt}{\operatorname{spt}}

\newcommand{\sing}{\operatorname{sing}}
\newcommand{\x}{\times}

\newcommand{\ins}{\mathrm{in}}
\newcommand{\out}{\mathrm{out}}

\newcommand{\injrad}{\mathrm{injrad}}

\newcommand{\GS}{\mathcal{S}}

\newcommand{\bd}{\mathbf{d}}

\newtheorem{claimIntro}{Claim}[subsection]

\usepackage[style=alphabetic, backend=biber, sorting=nyt, url=false ]{biblatex}
\title[Min-max theory and 
minimal surfaces with prescribed genus]{Min-max theory and 
minimal surfaces with\\ prescribed genus}

\author{Adrian Chun-Pong Chu, Yangyang Li, Zhihan Wang} 
\address{Cornell University, Ithaca, NY 14853, USA}
\email{cc2938@cornell.edu}
\address{University of Notre Dame,
Notre Dame, IN 46556, USA}
\email{yangyangli@nd.edu}
\address{Cornell University, Ithaca, NY 14853, USA}
\email{zw782@cornell.edu}

\date{\today}

\numberwithin{equation}{section}
\newtheorem{thm}{Theorem}[section]
\newtheorem*{thm*}{Theorem}
\newtheorem{cor}[thm]{Corollary}

\newtheorem{prop}[thm]{Proposition}
\newtheorem{lem}[thm]{Lemma}

\theoremstyle{definition}
\newtheorem{defn}[thm]{Definition}

\newtheorem{exmp}[thm]{Example}
\newtheorem{rmk}[thm]{Remark}

\newtheorem{oqn}[thm]{Question}
\newtheorem{cond}{Condition}

\newtheorem*{question*}{Question}

\begin{document}

\begin{abstract}     
We establish a general min-max type theorem that produces minimal surfaces with prescribed genus in 3-manifolds with positive Ricci curvature. An important intermediate step is to show that, in a generic metric with positive Ricci curvature, any family of surfaces with possibly   finitely many singularities can be deformed into a certain topologically optimal family.

Results in this paper  will be crucial  to our  program on the construction of multiple minimal surfaces with prescribed genus in 3-spheres via topological methods \cite{chu2025arbitraryGenus,chuLiWang2025fourGenus2}.
\end{abstract}
\maketitle 

\section{Introduction}
In recent years there is a rapid development on the construction of minimal surfaces from min-max theory. The Almgren-Pitts min-max theory has led to the proof of Yau's conjecture on the existence of infinitely many minimal surfaces in every 3-manifold \cite{MN17,IMN18, ChodoshMantoulidis2020AC,Zho20,Son23} (see also the recent work  \cite{chuStern2025doubling} for an approach using gluing construction for  generic metrics). Another line of research is to construct  minimal surfaces of prescribed genus. It was conjectured by Yau (1982) that every 3-sphere has at least 4 embedded minimal spheres, and by White that every 3-sphere has least 5 embedded minimal tori. Wang--Zhou \cite{WangZhou23FourMinimalSpheres} verified Yau's conjecture for generic metrics and metrics with positive Ricci curvature. In \cite{chuLi2024fiveTori}, the first two authors verified White's conjecture for metrics with positive Ricci curvature. Regarding more works on the construction of minimal submanifolds of controlled topological types under general metrics, see also   \cite{LS47, Str84, GJ86,Grayson89, Whi91, Zho16, HK19, bettiolPiccione2023bifurcationsCliffordTorus, HK23, Ko23a, Ko23b,bettiolPiccione2024nonplanarMinSpheres,LiWang2024NineTori,liWangYao2025minimalLensSpace}.
The above conjectures lead to the following question: 
\begin{oqn}
    In any given Riemannian 3-manifold $M$, how can one construct  embedded  minimal surfaces of a prescribed genus?
\end{oqn}
In this paper, we reduce this problem to the purely topological problem of {\it constructing   suitable ``non-trivial" families of surfaces in $M$.}
Results in this paper are  instrumental to our program on constructing minimal surfaces with prescribed genus in 3-spheres -- namely, the existence of minimal surfaces of arbitrary genus \cite{chu2025arbitraryGenus}, and the existence of 4  minimal surfaces of genus two  in Ricci-positive 3-spheres \cite{chuLiWang2025fourGenus2}\footnote{The current paper and \cite{chuLiWang2025fourGenus2} together supersede  \cite{chuLiWang2025genus2PartI}, which will not be submitted for publication.}.

Here is the setting. In a closed Riemannian $3$-manifold $M$, we let $\cS(M)$ be the set of all surfaces smoothly embedded in $M$ that {\it possibly have finitely many singularities}. For some technical reasons, we will focus on surfaces that are orientable, finite-area, and separate $M$ into two regions. Note that the notion of genus extends naturally to these singular surfaces, so we may let $ \cS_{\leq g}(M)\subset \cS (M)$ be the set of all singular surfaces of genus $\leq g$; readers may refer to \S \ref{sect:prelim} for the precise definitions. Our main theorem below essentially states that if there is a family $\Phi$ of singular surfaces with genus $\leq g$ that detects some non-trivial relative structure of the pair $(\cS_{\leq g}(M),\cS_{\leq g-1}(M))$, then $\Phi$ can be used to produce a minimal surface of genus $g$; see Figure \ref{fig:Sg} for an illustration.

\vbox{
\begin{thm}\label{thm:weakTopoMinMax}
    Let $M$ be a closed orientable Riemannian $3$-manifold with positive Ricci curvature, and $g$ be a positive integer. Suppose there exists a Simon--Smith family
    $$\Phi:X\to \cS_{\leq g}(M)$$ that cannot be deformed via pinch-off processes to become a map into $\cS_{\leq g-1}(M)$. 
    Then $M$ admits an orientable, embedded  minimal surface of genus $g$ with area at most $\displaystyle\max_{x\in X}\area(\Phi(x))$.
\end{thm}}

\begin{figure}
    \centering
    \includegraphics[width=1.8in]{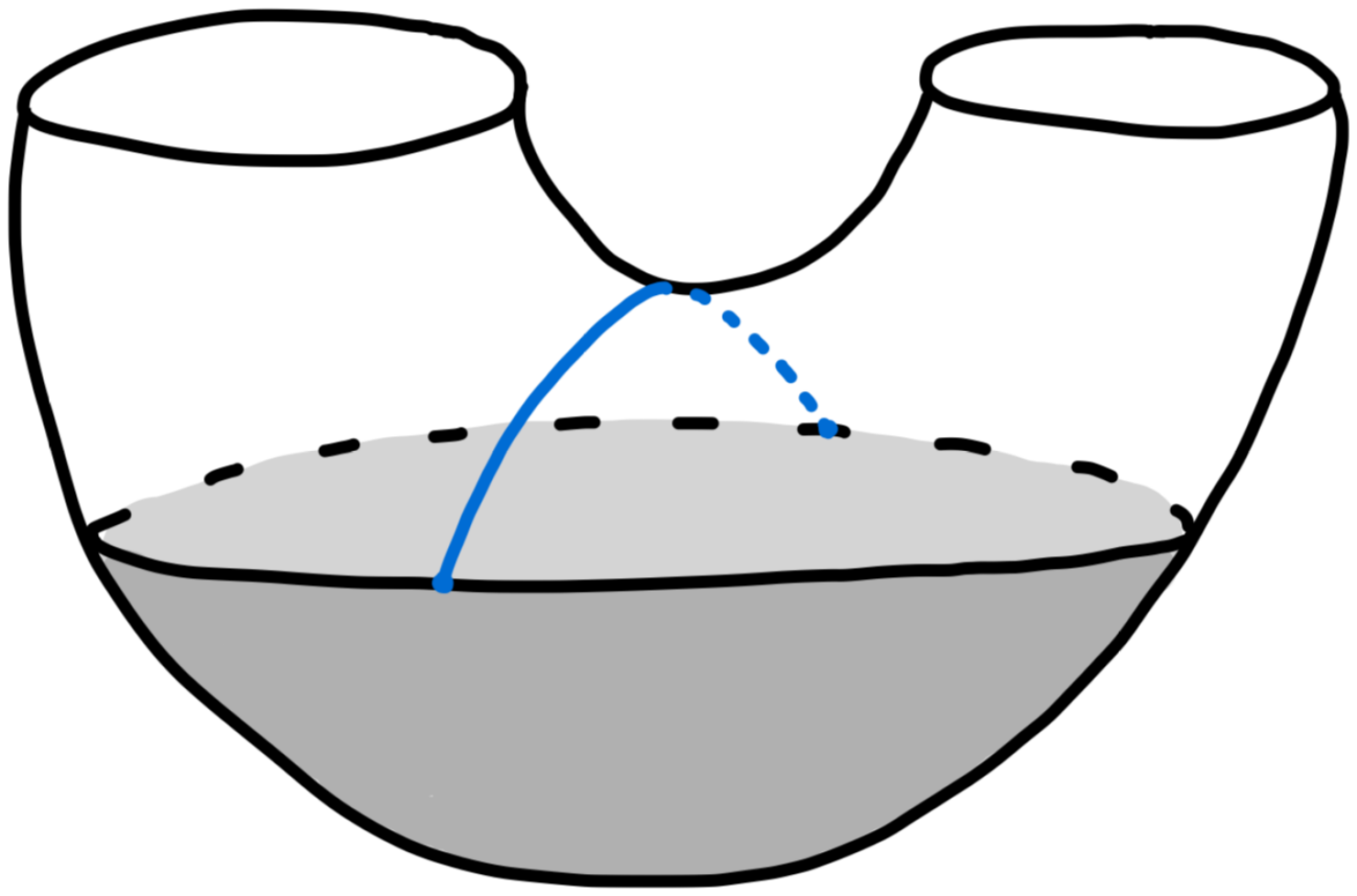} 
    \caption{This schematic diagram shows the set $\cS_{\leq g}(M)$ of all surfaces with genus at most $g$, possibly with singularities. The gray cap below represents $\cS_{\leq g-1}(M)$, while the blue path represents the image of a map $\Phi:X\to\cS_{\leq g}(M)$, detecting some non-trivial relative structure of the pair $(\cS_{\leq g}(M),\cS_{\leq g-1}(M))$.}
    \label{fig:Sg}
\end{figure}

The definitions of ``Simon--Smith family" and ``deformation via pinch-off processes" are given in Definition \ref{def:Simon_Smith_family} and Definition \ref{defn:deformViaPinchoff}, respectively. Note that we will not equip $\cS_{\leq g}(M)$ with a topology. Instead, the notion of the Simon--Smith family includes mild conditions on the map $\Phi$ to ensure the applicability of the Simon--Smith min-max theory. A pinch-off process refers to a deformation process for elements in $\cS(M)$ that consists of any combination of the following three procedures: (1) isotopy, (2) neck-pinch surgery, and (3) shrinking connected components to points (see Figure \ref{fig:pinchOff2}).  In particular, pinch-off processes do not increase genus. 

\begin{figure}
    \centering
    \makebox[\textwidth][c]{\includegraphics[width=5in]{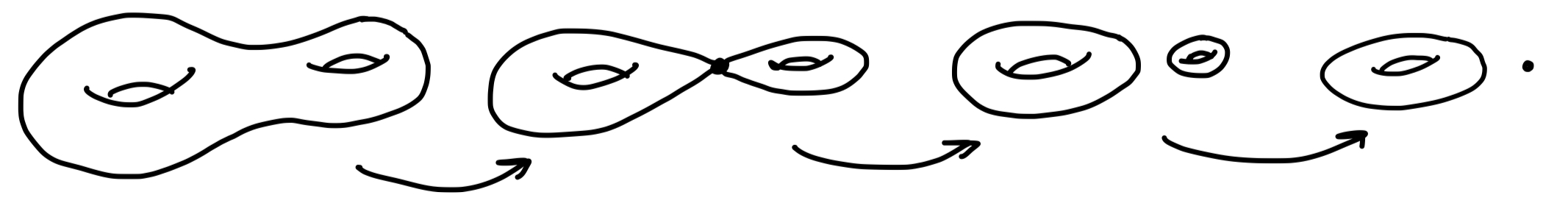}}
    \caption{This is an example of pinch-off process. It has one neck-pinch surgery, and one connected component shrunk to a point.}
    \label{fig:pinchOff2}
\end{figure}

\begin{rmk}
    Regarding the assumption of positive Ricci curvature, a primary purpose is to guarantee the Frankel property that every embedded minimal surface is connected. Without this, heuristically, it is possible that while the pair $(\cS_{\leq g}(M),\cS_{\leq g-1}(M))$ has some non-trivial relative structure, the ``genus $g$" minimal surface produced is however just the union of multiple minimal surfaces of lower genus whose total genus sums to $g$. Furthermore, positive Ricci curvature enables the application of the multiplicity one theorem by Wang--Zhou \cite{wangZhou2022higherMulti}, which will play a crucial role in our argument.
\end{rmk}

Theorem \ref{thm:weakTopoMinMax} is a direct  consequence of the following theorem, which says given any Simon-Smith family $\Phi$, we can find an associated ``topologically optimal family" $\Phi'$, obtained from $\Phi$ by  deformation  via pinch-off processes, such that $\Phi'$ detects some minimal surfaces $\Gamma_1,\dots,\Gamma_l$ of genus $\leq g$ and satisfies the conclusions \eqref{item:optimal} and \eqref{item:ti} below.
 
\begin{thm}[Existence of a topologically optimal family]\label{thm:repetitiveMinMax}
    Let $(M,\mathbf g)$ be a closed orientable Riemannian $3$-manifold with positive Ricci curvature, and $\Phi:X \to \mathcal{S}(M)$ be a Simon--Smith family of genus $\leq g$ such that  the min-max width $L:=\bL(\Lambda(\Phi))>0$. Assume Condition \ref{cond:A} for $(M, \bg)$, $g$ and $L$ (which is a generic condition on $\bg$ by Remark \ref{rmk:generic}).
    Then for any $r > 0$, there exist:
    \begin{itemize}
        \item a deformation via pinch-off processes, $H:[0,1]\x X\to\cS_{\leq g}(M)$, with $H(0,\cdot)=\Phi$, 
        \item  distinct, orientable, multiplicity one, embedded  minimal surfaces $\Gamma_1,\Gamma_2, \dots,\Gamma_l$ of genus $\leq g$ and area  $\leq L$,
        \item subcomplexes $D_0,D_1, \dots,D_l$ of $X$ (after refinement) that cover $X$,
    \end{itemize}
such that: 
    \begin{enumerate}
        \item \label{item:optimal} For  the family $\Phi':= H(1,\cdot)$,  each member of $\Phi'|_{D_i}$  has genus at most that of $\Gamma_i$ for each $i=1,\dots,l$, while each member of $\Phi'|_{D_0}$ has genus $0$.
        \item \label{item:ti} For each $i=1,\dots,l$, there is a smooth function $f_i:D_i\to[0,1]$ such that for every $x\in D_i$, $H(f_i(x),x)$ has the same  genus  as $\Gamma_i$ and is $r$-close to $\Gamma_i$ in the $\bF$-metric. 
    \end{enumerate}
\end{thm}

Below is the definition of Condition \ref{cond:A}. 
Let $\cO_{g,\leq L}(M,\bg)$ (resp. $\cN_{g,\leq L}(M,\bg)$) be the set of all orientable (resp. non-orientable) embedded minimal surfaces in $(M,\bg)$ with genus $g$ and area $\leq L$. Analogously, we define the sets $\cO_{\leq g,\leq L}(M,\bg)$ and $\cN_{\leq g,\leq L}(M,\bg)$ in which the genus of the minimal surfaces concerned is $\leq g$. 

\begin{cond}\label{cond:A}
For a closed Riemannian 3-manifold $(M,\bg)$, a positive integer $g$, and a real number $L>0$, the following hold:

    \begin{itemize}
    \item The set $\cO_{\leq g,\leq L}(M,\bg)$ is finite.
    \item Every element in $\cO_{\leq g-1,\leq L}(M,\bg)$ is non-degenerate.
    \item For  every  $\Gamma\in \cO_{g,\leq L}(M,\bg)$, numbers in the set 
        \begin{align*}
            &\{\area(\Gamma'):\Gamma'\in\cO_{\leq g-1,\leq L}(M,\bg)\cup\cN_{\leq g+1,\leq L}(M,\bg)
            \}
            \cup\{\area(\Gamma)\}
        \end{align*}
        are linearly independent over $\Z$. 
    \end{itemize}
\end{cond}

\begin{rmk}\label{rmk:generic}  Theorem \ref{thm:repetitiveMinMax} is applicable for a generic set\footnote{We consider genericity in the sense of Baire category under $C^\infty$-topology.} of Ricci-positive metrics. More precisely, for a generic   $\bg$, Condition \ref{cond:A}  holds for  every $g$ and $L>0$. Indeed,
  by White \cite{Whi91, white2017bumpy} and Marques-Neves \cite[Proposition 8.5]{MN21}, there is a generic set of metrics $\bg$ such that every $\bg$-minimal surface is strongly\footnote{Strong non-degeneracy means, after passing to the double cover,  the minimal surface concerned is still non-degenerate.} non-degenerate, and every finite collection of $\bg$-minimal surfaces $\{\Gamma_j\}$ has their area linearly independent over $\Z$. Moreover, for every $g$ and $L$, $\cO_{\leq g, \leq L}(M, \bg)\cup \cN_{\leq g, \leq L}(M, \bg)$ must be finite, since otherwise by \cite[Lemma 6.2]{chuLi2024fiveTori}, a pairwise distinct infinite sequence $\Sigma_j$ in this family will converge to some smooth minimal surface $\Sigma_\infty$ with multiplicity, which is either two-sided and degenerate, or one-sided and has a degenerate double cover, contradicting to the choice of $\bg$. 
\end{rmk}
 
The above theorem builds upon numerous important  foundational results in Simon--Smith min-max theory~\cite{Smith82, CD03, DP10, Ket19, Zho20, MN21,WangZhou23FourMinimalSpheres}. Theorem \ref{thm:repetitiveMinMax} is crucial to the  work \cite{chu2025arbitraryGenus} by the first author on the existence of minimal surfaces of arbitrary genus, and also our enumerative min-max theory  \cite{chuLiWang2025fourGenus2}, which constructs {\it multiple} minimal surfaces of prescribed genus.

\subsection{Main ideas}
This paper consists of three main components: 
\begin{enumerate}
    \item {\it Optimal family and repetitive min-max.} Given a Simon-Smith family $\Phi:X\to\cS_{\leq g}(M)$, we repetitively run min-max  to obtain  a ``topologically optimal family", establishing Theorem \ref{thm:repetitiveMinMax}. From this Theorem \ref{thm:weakTopoMinMax} follows directly. 
    \item {\it Genus reduction.} A crucial intermediate step is that, when a minimal surface $\Sigma$ is detected in any stage of min-max, on the ``cap of high area" in the pulled-tight family near $\Sigma$, we need to deform  the whole cap {\it to decrease the genus of each of its member to become the same as $\Sigma$} (Theorem \ref{thm:pinchOffMinMax} and Proposition \ref{Prop_Technical Deformation}). To achieve this, for each member $S$ in the cap, we need to choose some small balls (in $M$) and destroy all topology of $S$ within via neck-pinches and shrinking some components into points, in a  way continuous  across the whole cap. For this process, it is essential to consider surfaces with singularities  (see e.g. \cite[Example 2.6]{chuLi2024fiveTori}).
    \item {\it Selecting homologically linearly independent short loops.} To choose these balls, we need to locate on each $S$ in total $\fg(S)-\fg(\Sigma)$ homologically linearly independent short loops (Proposition \ref{prop:short_loops_II} and \ref{lem:small_balls_II}), where  $\fg(\cdot)$ denotes the genus, and  then pick the desired balls to enclose these loops.
\end{enumerate}

Although the repetitive min-max strategy was already employed by the first two authors in \cite{chuLi2024fiveTori} to produce 5 minimal tori in 3-spheres, in our current setting steps (2) and (3) are particularly delicate    as the   surfaces have {\it arbitrary genus}.  The collapse of genus  in step (2) needs to be done {\it continuously across the cap}, but in general  it is impossible to choose balls in step (3) continuously. In addition, it is necessary to bound the area change of all members on the cap during the neck-pinches and shrinking of  components, so that in the repetitive min-max process, the min-max width at each stage strictly decreases. These greatly complicate the selection of short loops and pinching. 
Let us explain in more detail.

\subsubsection{Optimal family and repetitive min-max}
To derive Theorem \ref{thm:weakTopoMinMax} using Theorem \ref{thm:repetitiveMinMax}, 
suppose by contradiction that no embedded genus $g$ minimal surface has area at most $L_0:=\max_x\area(\Phi(x))$. This implies, by Choi-Schoen compactness \cite{ChoiSchoen1985compactness}, that for some $\delta_0>0$, no orientable embedded genus $g$ minimal surface has area at most $L_0+\delta_0$. Thus, using Marques-Neves' result \cite[Proposition 8.5]{MN21} and Choi-Schoen compactness \cite{ChoiSchoen1985compactness}, it is possible to perturb the metric $\mathbf g$ to ${\mathbf g}'$ such that $(M,{\mathbf g}')$ satisfies Condition \ref{cond:A} 
 and have no genus $g$ minimal surface.

 Hence, applying Theorem \ref{thm:repetitiveMinMax} to $\Phi$ in $(M,\bg')$, we obtain an   optimal family $\Phi'$. In particular,   all the minimal surfaces $\Gamma_1,\dots,\Gamma_l$ given by  Theorem \ref{thm:repetitiveMinMax} must have genus $\leq g-1$. Hence, by item \eqref{item:optimal} in Theorem \ref{thm:repetitiveMinMax}, the family $\Phi'$  actually maps into $\cS_{\leq g-1}(M)$. So we have deformed $\Phi$ to become a map into $\cS_{\leq g-1}(M)$. This contradicts the assumption on $\Phi$, finishing the proof of  Theorem \ref{thm:weakTopoMinMax}. 

As for Theorem \ref{thm:repetitiveMinMax}, the optimal family $\Phi'$  is constructed using a repeated application of Simon--Smith min-max theory.  Starting from $\Phi$ and running min-max theory, we obtain a family that detects some minimal surface $\Sigma$. We would remove the portion (i.e. cap) of this family that is in the neighborhood of $\Sigma$, but a critical step is to first deform the whole cap to have the same genus as $\Sigma$, before removing it: This  ensures the {\it boundary} of the new family has genus $\fg(\Sigma)$,  which is essential for proving Theorem \ref{thm:repetitiveMinMax} \eqref{item:optimal}. We then run min-max again and repeat the process, which must terminate in finitely many steps. Then, as in \cite[\S 3.3]{chuLi2024fiveTori} we assemble the pieces created in a suitable way to obtain the optimal family $\Phi'$ and the deformation $H$. We refer readers to    \S 1.2.3 therein for the sketch, and shall not repeat it here.

\subsubsection{Genus reduction}    We are given a ``cap of high area" $\Phi|_Y$   (where $Y\subset X$ is some subcomplex in $I^m$) consisting of  singular surfaces of genus $\leq g$ that are close to some fixed smooth minimal surface $\Sigma$, where $g\geq\fg(\Sigma)$. For simplicity, let us only  consider families $\Phi=\Phi|_Y$ of  {\it smooth surfaces}, all with genus exactly $g$. 
Essentially we need to prove the following; see Theorem \ref{Prop_Technical Deformation} for the precise statement.

\begin{claimIntro}[Deforming a cap]\label{claim:deformCap}

Fix   any $\epsilon_1>0$. There exists $\delta>0$ such that, for any  family $\Phi:Y\to\cS_{g}(M)$  that is  $\delta$-close to $\Sigma$ (in $\bF$-distance), we can construct a deformation $H:[0,1]\x Y\to\cS_{\leq g}(M)$ such that:
\begin{enumerate}
    \item\label{item:deformCap_sameGenus} $H(0,\cdot )=\Phi$ and each member of $H(1,\cdot)$ has the same genus  as $\Sigma$.
    \item For each $t$ and $x$, $\bF(\Phi(x),H(t,x))<\epsilon_1.$
    \item For each $x$, the deformation $t\mapsto H(t,x)$ is a pinch-off process, and thus genus-decreasing.
\end{enumerate}
\end{claimIntro} 

Let us sketch the proof of  Claim \ref{claim:deformCap}.  We will find for each $y\in Y$ in total  $ q\leq g-\fg(\Sigma)$ balls in $M$ and then destroy the topology of $\Phi(y)$ in these balls:

\begin{claimIntro} [Selecting small balls for a cap]\label{claimIntro:shortLoopCap}
Fix any sufficiently small $\epsilon_2>0$. There exists $\delta>0$ such that, for any family $\Phi:Y\to \cS_{g}(M)$ that is $\delta$-close to $\Sigma$,  we can find a finite cover $\{O_x\}_x$ by open balls for $Y$,  where each $O_x$ has center $x$, and find for each such (finitely many) $x$ in total $q\leq g-\fg(\Sigma)$  balls\footnote{In practice $q$ should depend on $x$, but here for simplicity we assume it is the same for all $x$.} $\hat B^x_1,\dots,\hat B^x_q\subset M$ with the following property. For any $y\in Y$, if $y\in O_{x_1}\cap\dots\cap O_{x_k}$, then we have:
\begin{enumerate}
    \item The {\it boundary spheres} of the balls
    \begin{equation}\label{eq:hatBall}
        \hat B^{x_1}_1,\dots,\hat B^{x_1}_q,\dots, \hat B^{x_k}_1,\dots,\hat B^{x_k}_q
    \end{equation}
    are all disjoint, and these balls have radius in the range $(\epsilon_2,C\epsilon_2)$ for some\footnote{The constant $C$ depends only on $q$ and $m$.} $C$.
    \item For each ball $B$ in the list \eqref{eq:hatBall}, $\partial B$ intersects $\Phi(y)$ transversely.
    \item\label{item:subcollection} For an arbitrary subcollection $\{x_{i_1},\dots,x_{i_p}\}$ of $x$'s,  $\Phi(y) \setminus \cup_{\alpha=1}^p\cup_{\beta=1}^q  \hat B^{x_\alpha}_\beta$ has genus $\fg(\Sigma)$.
    \item $\area(\Phi(y)\cap B)<100\rad(B)^2.$
    \item $\length(\Phi(y)\cap \partial B)<100\rad(B).$

\end{enumerate}
\end{claimIntro}

We first explain how to derive Claim \ref{claim:deformCap} from  Claim \ref{claimIntro:shortLoopCap}. Fix $\epsilon_1>0$. Pick $\epsilon_2>0$ small, to be determined, and apply Claim \ref{claimIntro:shortLoopCap}. For each $x$ indexing the covering $\{O_x\}_x$, let $\eta_x:Y\to[0,1]$ be a smooth cut-off supported on $O_x$ such that $\eta_x\equiv 1$ on a slightly smaller ball $\hat O_x\subset O_x.$ We can assume the finite collection $\{\hat O_x\}$ still covers $Y.$ We decrease genus in two steps:
\begin{enumerate}
    \item If $y\in Y$ lies in $\hat O_x$, then for each ball $B$ in the list \eqref{eq:hatBall} that intersects $\Phi(y)$, we use $\partial B$ to ``cut" $\Phi(y)$. Namely, for each loop  component $\gamma$ of $\partial B\cap \Phi(x)$, we perform a neck-pinch surgery for $\Phi(y)$ along $\gamma$. Doing this\footnote{The loops $\gamma$'s may be nested, so we should perform surgery starting from the inner-most ones.} for all such $\gamma$, we can make $\Phi(y)$ disjoint from $\partial B$. Note, since by Claim \ref{claimIntro:shortLoopCap} (1) the boundaries of the balls in \eqref{eq:hatBall} are disjoint, such surgery processes would not interfere with each other.

    If $y\notin O_x$, we need not use  $\partial\hat B^{x}_1,\dots,\partial\hat B^{x}_q$ to cut $\Phi(y)$.
    As for $y\in O_x\backslash \hat O_x$, we use the cut-off function $\eta_x$ to interpolate in between, to ensure continuity across $Y.$
    \item Then, for any $y\in \hat O_x$, we collapse all components of $\Phi(y)$ within each ball $\hat B^{x}_1,\dots,\hat B^{x}_q$ to its center. This eliminates any topology of $\Phi(y)$ in those balls. Note the balls in \eqref{eq:hatBall} may be nested, and $y$ may belong to multiple $\hat O_x$: We  collapse them one-by-one  anyway, following an arbitrarily chosen order of the balls $\{\hat B^x_\beta\}_{x,\beta}$. At the end, due to Claim \ref{claimIntro:shortLoopCap} \eqref{item:subcollection}, $\Phi(y)$ becomes genus $\fg(\Sigma)$.
    
    Again, if $y\notin O_x$ we do not do anything, and if $y\in O_x\backslash \hat O_x$ we use $\eta_x$ to interpolate. These two steps   guarantees Claim \ref{claim:deformCap} \eqref{item:deformCap_sameGenus}.
\end{enumerate}
Clearly,  these deformations are via pinch-off processes only, as required by Claim \ref{claim:deformCap}  (3).
Now, due to item (1), (4), and (5)  in  Claim \ref{claimIntro:shortLoopCap}, if $\epsilon_2$ is small enough  depending on $\epsilon_1$, we can bound the change in  $\bF$-distance from the original family $\Phi$ by $\epsilon_1$ throughout the surgeries and collapsing, thereby ensuring Claim \ref{claim:deformCap} (2).    Hence, we obtain Claim \ref{claim:deformCap}. 

Now it remains to prove Claim \ref{claimIntro:shortLoopCap}. We  first need the following statement that selects small balls  for just {\it one arbitrary surface} $S$ close to $\Sigma$. Still, for simplicity, we assume $S$ is smooth. For any ball $B\subset M$, denote by $kB$ be the scaled ball with radius $k\;\rad(B)$ but center unchanged.

\begin{claimIntro}[Selecting small balls  for one surface]\label{claimIntro:shortLoopOneSurface}
For any $N\in\N$  and $\epsilon_3>0$, there exists $\delta>0$ such that, for any smooth surface $S\subset M$ with genus $g$ that is $\delta$-close to $\Sigma$,  there exist $q\leq g-\fg(\Sigma)$ balls $B_1,\dots,B_q\subset M$ such that:
\begin{enumerate}
    \item\label{item:shortLoopOneSurfaceDisjoint}  $\epsilon_3 <\rad(B_i)<C(N,g)\epsilon_3$ and the scaled balls $NB_i$ are disjoint.
    \item For any $i$ and any $k=1,\dots, N$, $\partial (kB_i)$ intersects $S$ transversely.
    \item \label{item:shortLoopOneSurfaceGenus} For   any choices of  $k_i=1,\dots,N$,   $S\backslash \cup_i k_iB_i$ has genus $\fg(\Sigma).$
\end{enumerate} 
\end{claimIntro}

We for a moment assume Claim \ref{claimIntro:shortLoopOneSurface} to prove Claim \ref{claimIntro:shortLoopCap}. Fix $\epsilon_2$, and some sufficiently large $N$ determined later. Let $\epsilon_3:=\epsilon_2$   and   apply  Claim \ref{claimIntro:shortLoopOneSurface}. Then putting $S=\Phi(x)$ for any $x\in Y$,  we obtain some balls $B^x_1,\dots,B^x_{q}\subset M.$ We can pick a sufficiently small ball $O_x\subset Y$ centered at $x$ such that the properties in  Claim \ref{claimIntro:shortLoopOneSurface}  hold not only for $\Phi(x)$ but also $\Phi(y)$ for every $y\in O_x$. By the fact that $Y$ is a cubical subcomplex of $I^m$, we can take a finite subcover of $\{O_x\}_{x\in Y}$ such that each member $O_x$ intersects at most $C(m)$ other members. For simplicity we still denote by $\{O_x\}_x$ this finite subcover. 

By a combinatorial argument (Lemma~\ref{lem:combinatorial_I}),  for a sufficiently large $N$ (depending only on $q$ and $m$), we can find some $k^x_i \in \{1,\dots,N\}$ for each $x$ and $i$ with the following property. Let $\hat B^x_{i}$ be the scaled ball $k^x_i  B^x_i$. For any $y\in Y$, if $y\in O_{x_1}\cap\dots\cap O_{x_k}$, then  the boundary spheres of the balls
\begin{equation}\label{eq:ballsBxi}
        \hat B^{x_1}_1,\dots,\hat B^{x_1}_q,\dots,\hat B^{x_k}_1,\dots,\hat B^{x_k}_q
    \end{equation}
    are all disjoint. Thus, they satisfy Claim \ref{claimIntro:shortLoopCap} (1). By  Claim \ref{claimIntro:shortLoopOneSurface} (3), the surface of interest in  Claim \ref{claimIntro:shortLoopCap} (3) has genus at most $\fg(\Sigma)$. But by closeness to $\Sigma$, the genus is exactly $\fg(\Sigma).$ Now, the area bound in Claim \ref{claimIntro:shortLoopCap} (4) follows from the $\delta$-closeness to $\Sigma$ and the  bound  $\epsilon_3 < \rad(B_i) < C(N,g)\epsilon_3$, and  item (5)  from a coarea formula argument. This proves  Claim \ref{claimIntro:shortLoopCap}.

As for the proof of Claim \ref{claimIntro:shortLoopOneSurface}, it depends on the following claim (Proposition \ref{prop:short_loops_II} and \ref{lem:small_balls_II}), which selects homologically independent short loops for any surface close to $\Sigma.$

\begin{claimIntro}[Selecting independent short loops for one surface]\label{claimIntro:linIndepLoop} Fix any $\epsilon_4>0$. There exists $\delta$ such that  the following holds. For any smooth embedded   surface $S$ of genus $g$ that is $\delta$-close to $\Sigma$,  
        there exist $g - \fg(\Sigma)$ disjoint simple loops $\{c_j\}^{g - \fg(\Sigma)}_{j = 1}$ on $S$ such that: 
        \begin{enumerate}[label = \normalfont(\arabic*)]
            \item $\length(c_j)<\epsilon_4$.
            \item The classes $[c_j]$ are  linearly independent in $H_1(S; \Z_2)$.
        \end{enumerate} 
    \end{claimIntro}
We use this claim to prove Claim \ref{claimIntro:shortLoopOneSurface} as follows. Fix $\epsilon_3$. We apply Claim \ref{claimIntro:linIndepLoop} with $\epsilon_4:=\epsilon_3$. Then for any $S$ $\delta$-close to $\Sigma$, we  obtain the loops $c_j$. For each $i$, pick a ball $\tilde B_i\subset M$ that contains $c_j$, with radius $\epsilon_4$.  Via a combinatorial argument (Lemma \ref{lem:disjoint_balls}), if $\epsilon_4$ is small enough (depending on $N,g-\fg(\Sigma)$), we can choose a fewer collection  of {\it larger balls}\footnote{The center of each $B_i$ will be the center of some $\tilde B_j$.}, $\{B_i\}^q_{i=1}$ where $q\leq g-\fg(\Sigma)$ and  $\rad(B_i)\leq (3N)^{g-\fg(\Sigma)}\epsilon_4$, such that $\cup_iB_i$ still contains $\cup_j\tilde B_j$ but now the scaled balls $NB_i$ are all disjoint, so that Claim \ref{claimIntro:shortLoopOneSurface} \eqref{item:shortLoopOneSurfaceDisjoint} holds. Also, $\cup_iB_i$  contains all the loops $c_j$, so in its complement the surface $S$ must have genus $\leq \fg(\Sigma)$, and by the closeness to $\Sigma$ the genus is actually exactly $\fg(\Sigma)$,  as required by  Claim \ref{claimIntro:shortLoopOneSurface}  \eqref{item:shortLoopOneSurfaceGenus}.   So now it remains to prove Claim \ref{claimIntro:linIndepLoop}.

\subsubsection{Selecting homologically linearly independent short loops}\label{sect:selectShortLoops}\;\\ 
 
\noindent {\bf Step 0.}  
We begin  with a claim that selects homologically independent loops for surfaces of small area (Proposition \ref{prop:short_loops_I} and \ref{lem:small_balls_I}). Note this statement concerns  {\it intrinsic surfaces}.
\begin{claimIntro}\label{claimIntro:smallArea} For any $\epsilon_5>0$ and $g\in\N$, there exists $\delta$ such that for all  closed Riemannian surface $\Gamma$ of genus $\leq g$ and area $<\delta$, there exist $\fg(\Gamma)$ disjoint embedded loops on $\Gamma$, with length $<\epsilon_5$, that are linearly independent in $H_1(\Gamma;\Z_2).$
\end{claimIntro}
The proof of this claim uses a systolic inequality of Loewner (see Theorem \ref{thm:systolic_ineq}), which says there exists some universal constant $C$ such that in any surface $\Gamma$, the minimal length of a homologially non-trivial loop is at most $C\area(\Gamma)$. To prove Claim \ref{claimIntro:smallArea},   we first use Loewner's theorem to locate a non-trivial loop $c_1$, and then perform surgery to $\Gamma$ along $c_1$, in which {\it the two discs glued back have  arbitrarily small area}: This is possible as the surgery is intrinsic. Then, we   use Loewner's theorem again to locate a second loop $c_2.$ We may need to modify $c_2$ to avoid the previous surgery discs (see Figure \ref{fig:construct_tilde_c}), so that $c_2$ can also be viewed as embedded in the original surface $\Gamma.$ Repeating, we obtain $\fg(\Gamma)$   short loops at the end, as required by Claim \ref{claimIntro:smallArea}. In the proof,  we can ensure if $\area(\Gamma)$ is less than some $\delta(\epsilon_5,g)$ then all those loops have length $<\epsilon_5.$

\medskip

\noindent {\bf Step 1.}  Returning to Claim \ref{claimIntro:linIndepLoop}, we have a surface $S$ of genus $g$  close to $\Sigma$, with lower genus. Choose $r_0$ such that  $\Sigma$ has a tabular neighborhood $N_{2r_0}(\Sigma)\subset M$ of radius $2r_0$. Fix $\epsilon_4$. By the coarea formula, we can find some $\delta$ such that for all surfaces $S$ $\delta$-close to $\Sigma$, there exists some $r\in (r_0,2r_0)$ such that, regarding   the tabular neighborhood $N_r(\Sigma)$, 
\begin{itemize}
    \item $\partial N_r(\Sigma)$ intersects $S$ transversely,
    \item $\area(S\backslash N_r(\Sigma))<\epsilon_4$ and $ \length(S\cap\partial N_r(\Sigma))<\epsilon_4.$
\end{itemize}
Then we perform surgery (inside $M$) along the loops in $S\cap\partial N_r(\Sigma)$ such that the resulting surface avoids $\partial N_r(\Sigma).$ Denote by $S_1$ the components inside $N_r(\Sigma).$

\medskip

\noindent {\bf Step 2.} In this step, we will find a suitable triangulation on $\Sigma$ to guide us in cutting $S_1$ into many small pieces, to which we will apply Claim \ref{claimIntro:smallArea} later in Step 3, in order to locate short loops.

We need to pick a triangulation $\cT$ on $\Sigma$ such that, 
for any 2-simplex $\Delta$ of $\cT$, we have  (see Figure \ref{fig:Gamma_Delta}):
\begin{itemize}
    \item Denote by $\Delta\x [-r,r]\subset M$  the prism of height $2r$ for which the middle slice $\Delta\x\{0\}$ coincides with $\Delta\subset\Sigma$, and denote by $\partial\Delta\x [-r,r]$ its three side faces. Then $\partial\Delta\x [-r,r]$ intersects $S_1$ transversely. 
    \item $(\partial\Delta\x [-r,r])\cap S_1$  consists of one {\it long loop} $c_\Delta$ and some {\it short loops}.
    \item The long loop $c_\Delta$ consists of 3 smooth segments, lying respectively on the 3 side faces of $\partial\Delta\x [-r,r]$, and $|\length(c_\Delta)-\length(\partial\Delta)|<\epsilon_4$.
    \item Each short loop lies on the {\it interior} of the side faces of $\partial\Delta\x [-r,r]$, and their total length is $<\epsilon_4.$ (In particular, for each vertex $v$ of $\Delta$, the vertical line $v\x[-r,r]$ cannot intersect any short loop.)
\end{itemize}
Furthermore,    {\it we wish that this triangulation $\cT$ depends not on $S$}. Clearly, this is impossible: If $\cT$ were  chosen first, we can easily find $S$ such that $S_1$ intersects some $\partial\Delta\x[-r,r]$ non-transversely. Therefore, in practice, we will  first choose some triangulation\footnote{ This $\cT_0$ does not need to depend on $\epsilon_4,\delta,r_0$.} $\cT_0$ so that for any $S$ concerned, we can perturb $\cT_0$ to some $\cT$ where the above conditions hold; moreover, {\it the amount of perturbation can be controlled uniformly independently of $S.$} . This perturbation is very delicate and occupies a great portion of Proposition \ref{prop:short_loops_II}.

Now, we perform surgery (in $M$) to $S_1$ along all {\it short loops}.  The resulting  surface $S_2$ will intersect any $\partial\Delta\x[-r,r]$ only along the  long loop $c_\Delta$.

\medskip

\noindent{\bf Step 3.} 
For each 2-simple $\Delta$,  we consider the following {\it intrinsic} closed surface $S_\Delta$: We cut $S_2$  along  $c_\Delta$ and then glue back a  topological disc to $S_2\cap (\Delta \x [-r,r])$ along $c_\Delta$, intrinsically.

Now note, back when we chose $\cT_0$ and $\cT$, we could require the area of $\Delta\subset\Sigma$ to be sufficiently small, and thus so is each intrinsic surface $S_\Delta$ (as $S$ is $\delta$-close to $\Sigma$), such that   Claim \ref{claimIntro:smallArea}, applied to $S_\Delta$, would imply the existence of $\fg(S_\Delta)$ disjoint, embedded, homologically non-trivial, linearly independent loops of length $<\epsilon_4$ in $S_\Delta$. A priori, these loops may intersect the surgery discs at the end of Step 2, but by modifying them actually we can  prevent this. 

\medskip 
\noindent{\bf Step 4.} 
Finally, denote by $\cC$ the collection of all the loops in Step 1, the {\it short loops} in Step 2, and the loops in Step 3. Note all these loops are disjoint, embedded, and have length $<\epsilon_4$. Moreover, they  can be viewed as lying in the original surface $S$. Then, it is easy to see that, if we perform surgery {\it intrinsically} to the surface $S$ along all members of $\cC$, the resulting surface has the same genus as $\Sigma$. Hence, there   exists a subcollection of $\fg(S)-\fg(\Sigma)$ members in $\cC$ that are homologically independent. This finishes the proof of Claim \ref{claimIntro:linIndepLoop}.

\subsection{Organization}
We first introduce the necessary terminologies in \S \ref{sect:prelim}. Then in \S \ref{sect:minMaxThm} we prove  Theorem \ref{thm:weakTopoMinMax} assuming Theorem \ref{thm:repetitiveMinMax}, and  in \S \ref{sect:ProofRepetitiveMinMax} we prove   Theorem \ref{thm:repetitiveMinMax} using a repetitive min-max procedure. An essential step in the construction of optimal family is a genus reduction procedure via pinch-off processes.  In \S \ref{sect:short_loops}, we locate short loops in members of the family for pinching, and finally in \S\ref{sect:interpolation} we carry out the genus reduction.

\subsection{Use of AI} 




We rewrote and simplified the proof of Proposition \ref{prop:short_loops_II} with the help of ChatGPT 5.5. In particular, in Step 1.2, ChatGPT  suggested some ideas, albeit in an argument containing some gaps, regarding the use of cut-off functions.


\subsection*{Acknowledgment} We would like to thank Andr\'e Neves for his constant support. 
This material is based upon work supported by the NSF Grant  DMS-1928930, while the first author was in residence at the SLMath. The first and the second author were partially supported by the AMS-Simons travel grant.

\section{Preliminaries}\label{sect:prelim}
\subsection{Notations}

In the sequel, all simplicial complexes and cubical complexes are assumed to be finite; all chains, homology groups, and cohomology groups are assumed to have $\Z_2$-coefficients, unless specified otherwise. In general, we use $M$ to denote a closed Riemannian manifold.
 
\begin{itemize}
    \item $\bI_n(M;\Z_2)$: the set of integral $n$-dimensional currents in $M$ with $\Z_2$-coefficients.
    \item $\Zc_{n}(M;\Z_2)\subset \bI_n(M;\Z_2)$: the subset that consists of elements $T$ such that $T=\partial Q$ for some $Q\in\bI_{n+1}(M;\Z_2)$ (such $T$ are also called {\it flat $k$-cycles}).
    \item $\Zc_{n}( M;\nu;\Z_2)$ with $\nu=\cF,\bF,\bM$: the set $\Zc_{n}( M;\Z_2)$ equipped with the three topologies given corresponding respectively to the  {\em flat norm $\Fc$},  the {\em $\Fb$-metric}, and the  {\em mass norm $\Mb$} (see \cite{Pit81} and the survey paper \cite{MN20}). For the flat norm, there are two definitions that, by the isoperimetric inequality, would induce the same topology:
    $$\cF(T):=\inf\{\bM(P)+\bM(Q):T=P+\partial Q\},\quad \textrm{ and } \quad \cF(T):=\inf\{\bM(Q):T=\partial Q\}.$$
    In this paper, we will use the second definition.
    \item $\mathcal{V}_n(M)$: the closure, in the varifold weak topology, of the space of $n$-dimensional rectifiable varifolds in $M$.
    \item $\cC(M)$: the space of Caccioppoli sets in $M$, equipped with the metric induced by the Hausdorff measure of the symmetric difference.
    \item $\partial^* \Omega$: the reduced boundary of $\Omega \in \cC(M)$.
    \item $\nu_\Omega$: the inward pointing normal of $\Omega \in \cC(M)$.
    \item $\norm{V}$: the Radon measure induced on $M$ by  $V\in \mathcal{V}_n(M)$.
    \item $|T|$: the varifold in $\cV_n(M)$ induced by a current $T\in \Zc_{n}(M;\Z_2)$, or a countable $n$-rectifiable set $T$. In the same spirit, given a map $\Phi$ into $\Zc_n(M;\Z_2)$, the associated map into $ \mathcal{V}_n(M)$ is denoted by $|\Phi|$.
    \item $\spt(\cdot)$: the support of a current or a measure.
    \item $[W]$: the $\Z_2$-current induced by $W$, if $W$ is a countably $2$-rectifiable set with $\cH^2(W) < \infty$. In the same spirit, given a map $f$ that outputs countably $2$-rectifiable sets, the associated map into the space $\cZ_n(M; \Z_2)$ is denoted by $[f]$.
    \item $[\mathcal{W}] := \{[\Sigma_i]\} \subset \cZ_n(M;\Z_2)$ for a set $\mathcal{W}$ of varifolds $\{V_i\} \subset \cV_n(M)$, each associated with a countably $n$-rectifiable set $\Sigma_i$.
    \item $\bB^\nu_\varepsilon(\cdot)$:  the open $\varepsilon$-neighborhood of an element or a subset of the space $\Zc_n(M;\nu;\Z_2)$.
    \item $\bB^{\Fb}_\varepsilon(\cdot)$: the open $\varepsilon$-neighborhood  of an element or  a subset of $\cV_n(M)$ under the $\Fb$-metric.
    \item $\Gamma^\infty(M)$: the set of smooth Riemannian metrics on $M$.
    \item $B_r(p)$: the open $r$-neighborhood of a point $p$; $B_r(\{p_1, \dots,p_n\}):=\bigcup_iB_r(p_i)$.
    \item $\fg(S)$: the genus of a surface $S$. 
    \item $\ins(S)$: the inside, open region of a  surface $S$;  $\out(S)$: the outside  region of a  surface $S$.
\end{itemize}

For an $m$-dimensional cube $I^m = \mathbb{R}^m \cap \{x : 0 \leq x_i \leq 1, i = 1,2, \dots, m\}$, we can give it {\em cubical complex structures} as follows.
\begin{itemize}
    \item  $I(1,j)$: the cubical complex on $I := [0,1]$ whose $1$-cells and $0$-cells are respectively 
    $$[0,1/3^j],[1/3^j,2/3^j],\dots,[1-1/3^j,1]\;\;\textrm{ and }\;\; [0],[1/3^j],[2/3^j],\dots,[1].$$
    \item $I(m,j)$: the cubical complex structure
    \[
        I(m,j)=I(1,j)\otimes\dots\otimes I(1,j)\;\;(m\textrm{ times})
    \]
    on $I^m$. $\alpha = \alpha_1 \otimes \cdots \otimes \alpha_m$ is a {\em $q$-cell} of $I(m, j)$ if and only if each $\alpha_i$ is a cell in $I(1, j)$ and there are exactly $q$ $1$-cells. A cell $\beta$ is a {\em face} of a cell $\alpha$ if and only if $\beta \subset \alpha$ as sets.
\end{itemize}
We call $X \subset I(m, j)$ a {\em cubical subcomplex of $I(m, j)$} if every face of a cell in $X$ is also a cell in $X$. For convenience, we also call $X$ a {\em cubical complex} without referring to the ambient cube. 

We denote by $|X|$ the {\em underlying space} of $X$.  For the sake of simplicity, we will also consider a complex and its underlying space as identical unless there is ambiguity. 
Given a cubical subcomplex $X$ of some $I(m, j)$, for $j' > j$, one can {\em refine} $X$ to a cubical subcomplex
\[
    X(j') := \{\sigma \in I(m, j'): \sigma \cap |X| \neq \emptyset\}
\]
of $I(m, j')$. For the sake of convenience, we will denote the refined cubical subcomplex by $X$ unless there is ambiguity.

\subsection{Surface with singularities}
\label{sect:basicDef}
Let $M$ be a closed orientable 3-manifold, equipped with a Riemannian metric (though all notions in \S \ref{sect:basicDef}, except those involving the $\bF$-metric, will not depend on the metric of $M$). In this paper, we will consider surfaces with finite area and finitely many singularities, which possibly include some isolated points too:

\begin{defn}[Punctate surface]\label{def:punctate_surf} Let $M$ be a closed orientable 3-manifold.
    A closed subset $S \subset M$ is a {\it punctate surface} in $M$ provided that: 
    \begin{enumerate}
        \item The 2-dimensional Hausdorff measure $\cH^2(S)$ is finite.
        \item\label{item:punctateOrientable} There exists a (possibly empty) finite set $P \subset S$ such that $S \setminus P$ is a smooth, orientable, embedded surface.
        \item\label{item:punctateSeparate} (Separating) Let $S_{\mathrm{iso}}$ be the set of isolated points of $S$. Then the complement of $S\backslash S_{\mathrm{iso}}$ is a disjoint union of two open regions of $M$, each having $S\backslash S_{\mathrm{iso}}$ as its topological boundary.
    \end{enumerate} 
    We denote by $\cS(M)$ the set of all punctate surfaces in $M$.
\end{defn}
Note that the condition of $S\backslash P$ being orientable in item (2) is redundant, as this follows from item (3) and the assumption that $M$ is orientable. Moreover,
it is clear that for every $S\in\cS(M)$, there exists a smallest possible set, denoted by $S_{\sing}$, satisfying item (2) above, which consists of the non-smooth points of $S$; see Figure \ref{fig:a punctate surface} for an example. In this paper, we often consider the case where $P$  contains smooth points of $S$, e.g., in Definition \ref{def:Simon_Smith_family}.

\begin{defn}[Punctate set]
    Let $S\in\cS(M)$. For \textit{any} finite set $P$ such that $S \setminus P$ is a smooth embedded surface, we call $P$  a {\it punctate set} for $S$.
\end{defn}

\begin{figure}[!ht]
    \centering
    \includegraphics[width=0.5\linewidth]{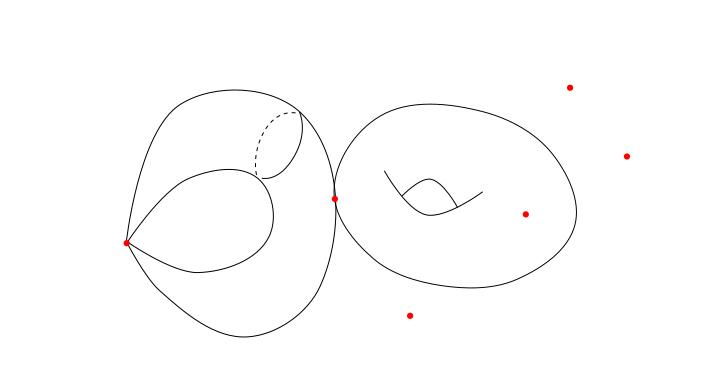}
    \caption{This picture shows an element $S\in\cS(M)$: It is a closed set, which contains a (black) smooth surface part and also the red points. The smallest possible punctate set for $S$ is given by the red points. Note $S$ has genus $1$.}
    \label{fig:a punctate surface}
\end{figure}
 
    As shown in \cite[\S 2]{chuLi2024fiveTori}, every element  $S\in \cS(M)$ can be associated uniquely to an integral 2-current $[S]\in \bI_2(M;\Z_2)$ with $\partial [S]=0$. In other words, $[S]\in\cZ_2(M;\Z_2)$.
Hence, we can borrow the $\bF$-metric  on $\cZ_2(M;\Z_2)$ and impose it on $\cS(M)$.
\begin{defn}\label{def:bF_metric_GS} 
    We define $\bF:\cS(M)\x\cS(M)\to\R$ by
    $$\bF(S_1,S_2):=\bF([S_1],[S_2])=\cF([S_1],[S_2])+\bF(|S_1|,|S_2|).$$
\end{defn}

One can check that $\bF$ is a pseudometric on $\cS(M)$ (two members of $\cS(M)$ differing just by finitely many points would have zero $\bF$-distance).

Now, for any $S \in \GS(M)$, let $S_{\sing}$ be the set of non-smooth points of $S$.
Then by Sard's theorem, there exists a set  $E \subset (0, \infty)$ of full measure such that for every $r \in E$, $S \setminus B_r(S_{\sing} )$ is a compact smooth surface with boundary, where $B_{r}(P)$ denotes the union of the open $r$-balls centered at the points in $P$. Moreover, by~\cite[Chapter~1~\S 2.1~Lemma~1.5]{CM12}, for any $r_1 > r_2 > 0$ in $E$, we have the   genus inequality 
    $\mathfrak{g}(S \setminus B_{r_1}(S_{\sing})) \leq \mathfrak{g}(S\setminus B_{r_2}(S_{\sing}  ))\,.$
Hence, the limit
\[
    \lim_{r\in E,r\to 0} \mathfrak{g}(S\setminus B_{r}(S_{\sing} )) \in \N \cup \{\infty\}
\]
always exists.
 
\begin{defn}[Genus of a punctate surface]\label{def:genus} 
    Let $S\in\cS(M)$, and let $E$ be the set of $r>0$ such that $S \setminus B_r(S_{\sing} )$ is a smooth surface with boundary. Then we define the {\it genus} of $S$ by
    \[
        \fg(S) := \lim_{r\in E,r\to 0} \fg(S \setminus B_{r}(S_{\sing}  ))\,.
    \]
    Moreover, the set of $S\in\cS(M)$ with genus $g$ is denoted $\cS_{g}(M)$, or simply $\cS_{g}$; the set of $S\in\cS(M)$ with genus $\leq g$ is denoted $\cS_{\leq g}(M)$, or simply $\cS_{\leq g}$.
\end{defn}

Figure \ref{fig:a punctate surface} is an example of a genus one punctate surface. Moreover, it is easy to check that given any $S\in\cS(M)$, for any punctate set $P$, if we let $E$ be the set of $r>0$ such that $S \setminus B_r(P)$ is a smooth surface with  boundary, then $\lim_{r\in E,r\to 0} \fg(S \setminus B_{r}(P))=\fg(S)$. 

We finish this section with a lemma relating genus and the $\bF$-pseudometric, which is proved in Appendix \ref{sect:genusIneq}.
\begin{lem}\label{lem:genusLarger}
    Let $(M,\bg)$ be a closed Riemannian $3$-manifold, and $\Sigma$ be an orientable smooth embedded surface in $M$.  Then there exists $r>0$ such that for any $S\in\cS(M)$ with $\bF(S,\Sigma)<r$, we have $\fg(S)\geq \fg(\Sigma)$.
\end{lem}

\subsection{Simon--Smith family}

\begin{defn}[Simon--Smith family]\label{def:Simon_Smith_family}
    Let $X$ be a cubical or simplicial complex. 
    A map $\Phi: X \to \GS(M)$ is called a {\em Simon--Smith family} with parameter space $X$ if the following hold.
    \begin{enumerate}[label=\normalfont(\arabic*)]
        \item \label{item:Hausdorff_cts} The  map $x\mapsto\cH^2(\Phi(x))$ is continuous. 
        \item \label{item:closedFamily} (Closedness) The family $\Phi$ is ``closed" in the sense that 
        $\{(x,p)\in X\x M: p\in \Phi(x)\}$
        is a closed subset of $X \x M$. 
        \item \label{item:SingPointsUpperBound} ($C^\infty$-continuity away from singularities) For each $x\in X$, we can choose a finite set $P_\Phi(x)\subset \Phi(x)$ such that:
        \begin{itemize}
            \item $\Phi(x)\backslash P_\Phi(x)$ is a smooth embedded orientable surface, i.e.  $P_\Phi(x)$ is  a  punctate set for $\Phi(x)$.
            \item For any $x_0\in X$  and within any open set $U\subset\subset M \backslash P_\Phi(x_0)$, $\Phi(x)\to \Phi(x_0)$ smoothly whenever $x\to x_0$.
            \item $\displaystyle N(P_\Phi) := \sup_{x\in X} |P_\Phi(x)|<\infty.$
        \end{itemize}
    \end{enumerate}
\end{defn}

\begin{exmp}
    The set $P_\Phi(x)$ might contain more points than the non-smooth points of $\Phi(x)$. Consider a family $\Phi:[0,1]\to \cS(M)$, where $\Phi(t)$ is a smooth torus for $t< 1$,  $\Phi(1)$ is a {\it smooth} sphere, and as $t\to 1$, the handle of the torus becomes smaller and eventually vanishes at some smooth point $p$ on the sphere $\Phi(1)$. In this case, $P_\Phi(1)$ contains the point $p$. 
\end{exmp}

\begin{rmk} 
    In previous literature, Simon--Smith families are often allowed to contain certain subsets of measure zero as members. In this paper, to simplify the notations, we choose not to do that. But we note that all results can be modified in obvious ways to allow Simon--Smith families to contain closed subsets of measure zero. See \cite{chuLi2024fiveTori} for such a modification.
\end{rmk}

\begin{defn}
    A Simon--Smith family $\Phi$ is of  {\it genus $\leq g$} if each $\Phi(x)$ has genus $\leq g$. 
\end{defn} 
 
As shown in \cite[\S 2]{chuLi2024fiveTori}, every Simon--Smith family is $\bF$-continuous.

\begin{defn}\label{def:homotopyClass}
    Two Simon--Smith families $\Phi,\Phi':X\to\cS(M)$ are said to be {\em Simon--Smith homotopic via isotopies} if there exists a continuous map 
    $\varphi: [0, 1] \times X \to \operatorname{Diff}^\infty(M)$ such that:
    \begin{enumerate}
        \item $\varphi(0, x) = \operatorname{Id}$ for all $x \in X$,
        \item $\varphi(1, x)(\Phi(x)) = \Phi'(x)$ for all $x \in X$,
    \end{enumerate}
    where $\textrm{Diff}^\infty(M)$ denotes the group of diffeomorphisms on $M$ equipped with the $C^\infty$ topology. Moreover, the set of all families homotopic to a Simon--Smith family $\Phi$ is called the {\em Simon--Smith homotopy class associated to $\Phi$}, denoted by $\Lambda (\Phi)$.
\end{defn}

\begin{rmk}\label{rmk:homotopy}
    \begin{itemize} 
        \item []
        \item For any choice of sets $P_\Phi(x)$ for $\Phi$ and for any  $\varphi$, the sets $P_{\Phi'}(x) := \varphi(1, x) \circ P_\Phi(x)$ would satisfy the conditions in Definition \ref{def:Simon_Smith_family} for the family $\Phi':=\varphi(1,\cdot)\circ\Phi \in \Lambda (\Phi)$. Consequently, we may assume that $\sup_x |P_\Phi(x)| < \infty$ is a constant within $\Lambda (\Phi)$.
        \item If $\Phi$ is of genus $\leq g$, then so is every $\Phi' \in \Lambda (\Phi)$.
    \end{itemize}
\end{rmk}

\subsection{Min-max theorems} 
\label{sect:minMax} 
    Given a homotopy class $\Lambda$ of Simon--Smith families, its {\em width} is defined as
    \[
        \mathbf{L}(\Lambda):=\inf_{\Phi\in \Lambda}\sup_{x\in X}\cH^2(\Phi(x))\,.
    \] 
    A sequence $\{\Phi_i\}$ in $\Lambda$ is said to be {\em minimizing} if 
   $\sup_{x\in X}\cH^2(\Phi_i(x))\to\mathbf{L}(\Lambda)\,.$
    If $\{\Phi_i\}$ is a minimizing sequence in $\Lambda$ and $\{x_i\} \subset X$ is such that $ \cH^2(\Phi_i(x_i))\to\mathbf{L}(\Lambda)\,,$
    then $\{\Phi_i(x_i)\}$ is called a {\em min-max sequence}. 
    For a minimizing sequence $\{\Phi_i\}$, we define its \textit{critical set} $\bC(\{\Phi_i\})$ to be the set of all subsequential varifold-limit of its min-max sequences:
    \[
        \bC(\{\Phi_i\}):=\{V=\lim_j|\Phi_{i_j}(x_j)|: x_j\in X, \| V\|(M)=\bL(\Lambda)\}\,.
    \]
    And $\{\Phi_i\}$ is called \textit{pulled-tight} if every varifold in $\bC(\{\Phi_i\})$ is stationary. 

\begin{defn}\label{def:minimal_surf_varifolds}
    Let $(M, {\mathbf g})$ be a closed orientable  Riemannian $3$-manifold and  $L > 0$.
    \begin{enumerate}
        \item Denote by $\cW_L(M, {\mathbf g})$ the set of all       varifolds $W \in \cV_2(M)$, with $\|W\|(M) = L$, of the form 
        \begin{equation}\label{eq:formOfW}
            W=m_1|\Gamma_1|+\cdots+m_l|\Gamma_l|\,,
        \end{equation}
        where  $m_1, \dots,m_l$ are   positive integers, and $\Gamma_1, \dots,\Gamma_l$ are disjoint smooth, connected, embedded minimal surfaces in $(M, {\mathbf g})$, with $m_i$ even if $\Gamma_i$ is non-orientable (see \cite[Remark 1.4]{Ket19}).
        \item Denote by $\cW_{L, \leq g}(M, {\mathbf g})$ the set of all varifolds $W \in \cW_L(M, {\mathbf g})$ (of the form (\ref{eq:formOfW})) satisfying
        \begin{equation}\label{eq:genus_bound}
            \sum_{j\in I_O}m_j\fg(\Gamma_j)+ \frac{1}{2}\sum_{j\in I_N}m_j(\fg(\Gamma_j)-1)\leq g\,,
        \end{equation}
        where $\Gamma_j$ is orientable if $j \in I_O$ and non-orientable if $j \in I_N$. 
    \end{enumerate}
\end{defn}

For simplicity, when the ambient manifold $(M, {\mathbf g})$ is clear, we may omit it in the notations, and use the notations $\cW_{L}$ and $\cW_{L, \leq g}$, respectively.

\begin{thm}[Simon--Smith min-max theorem]\label{thm:minMax}
    Let $(M, \bg)$ be a closed orientable Riemannian $3$-manifold. 
    Suppose a Simon--Smith family $\Phi$ of genus $\leq g$ has $L := \mathbf{L}(\Lambda(\Phi)) > 0$. Then for any $r > 0$, there exists a minimizing sequence $\{\Phi_{i}\}$ in $\Lambda(\Phi)$ such that:
    \begin{enumerate}[label=\normalfont(\arabic*)]
        \item\label{item:minMaxPulltight} The sequence $\{\Phi_i\}$ is pulled-tight and 
        $\bC(\{\Phi_{i}\})\cap \cW_{L, \leq g} \neq \emptyset\,.
        $
        \item\label{item:minMaxMultiplicityone} There exists some
        $$W=m_1|\Gamma_1|+\cdots+m_l|\Gamma_l|\in \bC(\{\Phi_{i}\})\cap \cW_{L, \leq g},$$
        under the notation in Definition \ref{def:minimal_surf_varifolds} (1),
        such that:
        \begin{enumerate}[label=\normalfont(\alph*)]
            \item If $\Gamma_j$ is unstable and two-sided, then $m_j = 1$.
            \item If $\Gamma_j$ is one-sided, then its connected double cover is stable and $m_j$ is even.
        \end{enumerate}
        \item \label{item:minMaxW} Furthermore, there exists $\eta > 0$ such that for all sufficiently large $i$,
        \[
            \cH^2(\Phi_{i}(x)) \geq L - \eta \implies |\Phi_i(x)| \in \bB^{\Fb}_{r}(\cW_{L, \leq  g})\,.
        \]
    \end{enumerate}
\end{thm}

The above theorem builds upon the foundational results ~\cite{Smith82, CD03, DP10, Ket19, MN21,WangZhou23FourMinimalSpheres}, and was proven  in  \cite{chuLi2024fiveTori}: See Theorem 2.16 therein (we can guarantee that   $m_i$ is even if $\Gamma_i$ is one-sided \cite[Remark 1.4]{Ket19}). Note that since $M$ is orientable, a surface is two-sided if and only if it is orientable. 

The following theorem, proven in \cite[\S 3]{chuLi2024fiveTori}, is also useful. Roughly speaking, if all min-max minimal surfaces are of multiplicity one, then there exists some family where in the ``high-area cap"  we can guarantee closeness to the minimal surfaces in the $\bF$-norm for currents.

\begin{thm}[$\bF$-closeness]\label{thm:currentsCloseInBoldF}
    Let $(M, {\bg})$ be a closed orientable  Riemannian $3$-manifold,  and $\Phi:X\to\cS(M)$ be a Simon--Smith family of genus $\leq g$ such that $L := \mathbf{L}(\Lambda(\Phi)) > 0$. Suppose that $\cW_{L, \leq g}$ consists of only finitely many varifolds, each induced by a multiplicity one embedded minimal surface. Then for any $r>0$, there exists $\eta>0$ and $\Phi' \in \Lambda(\Phi)$ such that 
    \[
      \cH^2(\Phi'(x))\geq L -\eta \implies [\Phi'(x)] \in \bB^\bF_r([\cW_{L, \leq g}]) \,.
    \]
\end{thm} 

\subsection{Pinch-off process and genus reduction}\label{sect:pinchOff}

In this section, we state an important intermediate   result, Theorem \ref{thm:pinchOffMinMax}. Namely, given a family $\Phi$ that produces a \textit{multiplicity one, non-degenerate} min-max minimal surface $\Gamma$, we can construct a ``better" family $\Phi'\in\Lambda(\Phi)$ such that {\it on the high-area cap, the genus of each element is reduced to the same as $\Gamma$}. Thus, we  gain topological control on the high-area cap, compared to Theorem \ref{thm:currentsCloseInBoldF}. This is achieved by performing {\it pinch-off processes}, continuously across the whole cap of punctate surfaces of high area in $\Phi''$. See Figure \ref{fig:pinchOff2} for an illustration.

\begin{defn}\label{defn:pinch_off}
    In a closed 3-manifold $M$, a  collection  $\{\Gamma(t)\}_{t\in[a,b]}$ of elements in $\cS(M)$ is called a {\em    pinch-off process} if the following holds. Regarding the parameter space $[a, b]$   as a time interval, there exist finitely many spacetime points $(t_1,p_1), \dots,(t_n,p_n)$ in $(a,b]\x M$, with each  $p_i\in \Gamma(t_i)$, such that  each $(t,p)\in[a,b]\x M$ is of one of the following types.

    If $(t,p)$ is not equal to any $(t_i,p_i)$, then:
    \begin{enumerate}
        \item  There exists a closed neighborhood $J\subset [a,b]$ around $t$ and a ball $U\subset M$ around $p$ such that $\{\Gamma(t)\}_{t\in J}$ deforms by an isotopy within $U$.
    \end{enumerate}
    
    If $(t,p)$ is equal to some $(t_i,p_i)$, then there exists a closed neighborhood $J\subset [a,b]$ around $t$ and a ball $U\subset M$ around  $p$ such that the family $\{\Gamma(t')\cap U\}_{t'\in J}$ of surfaces is of one of the following types (Note it is possible that $t=b$.):
    \begin{enumerate}
    \setcounter{enumi}{1}
        \item {\it Surgery }:
        The initial surface deforms by an isotopy for $t'<t$, but then at time $t$  becomes, up to diffeomorphism, a double cone with $p$ as the cone point (as shown in the second picture in Figure \ref{fig:pinchOff2}). Afterwards, it either remains a double cone or splits into two smooth discs: In either case it is allowed to deform by another isotopy.
        \item {\it Shrinking into points}: 
        For each $t'\in J$, {\it $\Gamma(t')$ does not intersect $\partial U$.} Moreover, for the family
        $\{ \Gamma(t')\cap U\}_{t' \in J }$, the initial surface deforms by  isotopy in $U$ when $t'<t$, but then becomes just the point $p$ at time $t$, and this point moves smoothly after time $t.$
    \end{enumerate}
\end{defn}

Note, all pinch-off processes are Simon--Smith families. In addition, in (3), since $\Gamma(t')$ does not intersect $\partial U$, the part we are shrinking must be a union of several connected components. The reason for allowing the double cone to remain after time $t$ in case (2) will be addressed in Remark
\ref{rmk:leftIntact}. Now, the lemma below follows directly from the definition of the pinch-off process. 

\begin{lem}\label{prop:pinchOffSOB} 
    Along a pinch-off process, genus is non-increasing.
\end{lem}

\begin{defn}\label{defn:deformViaPinchoff}
    Let $\Phi,\Phi':X\to\cS(M)$ be Simon--Smith families. Suppose we have a Simon--Smith family $H:[0,1]\x X\to \cS (M)$ such that:
    \begin{itemize}
        \item  $H(0, \cdot)=\Phi$ and $H(1, \cdot)=\Phi'$.
        \item For each $x\in X$, $t\mapsto H(t,x)$ is a pinch-off process.
    \end{itemize}
    Then   $H$ is called a {\it deformation via pinch-off processes} from $\Phi$ to $\Phi'$.
\end{defn}
    
We now state a crucial genus reduction result.

\begin{thm}[Genus reduction]\label{thm:pinchOffMinMax}
    In a closed orientable Riemannian $3$-manifold $(M, {\mathbf g})$, let $\Phi:X\to\GS(M)$ be a Simon--Smith family  of genus $\leq g$ with $X$ a cubical complex. Assume that $L := \bL(\Lambda  (\Phi)) > 0$, and $\cW_{L, \leq g}=\{|\Gamma|\}$ where $\Gamma$ is a nondegenerate embedded closed minimal surface. Then for any $r > 0$, there exists a Simon--Smith family $\Phi': X \to \mathcal{S}_{\leq g}$  with the following properties:
    \begin{enumerate}[label=\normalfont(\arabic*)]
        \item There exists an $\eta>0$ such that for all $x \in X$ with $\cH^2(\Phi'(x))\geq L-\eta$, we have 
        $\Phi'(x)\in \cS_{\fg(\Gamma)}\cap \bB^\bF_{r}(\Gamma).$
        \item\label{item:mapping_cylinder} $\Phi$ can be deformed via pinch-off processes to $\Phi'$.
    \end{enumerate}
\end{thm}

We postpone the proof of Theorem \ref{thm:pinchOffMinMax} to Section~\ref{sect:interpolation}, which details the pinch-off constructions, and Appendix~\ref{sect:proof_pinchOffMinMax}. As a special case, when $L = 0$, the whole family can be deformed via pinch-off processes to surfaces of genus $0$:

\begin{prop}[Genus reduction to zero]\label{prop:zeroWidth}
    In a closed orientable Riemannian $3$-manifold $(M, {\mathbf g})$,  suppose that $\Phi:X\to\GS(M)$ is a Simon--Smith family of genus $\leq g$ with $X$ a cubical complex and $\bL(\Lambda  (\Phi)) = 0$. Then $\Phi$ can be deformed via pinch-off processes into some  Simon--Smith family $\Phi' : X \to \GS_{0}$ of genus 0.
\end{prop}

The above proposition is a direct corollary of Proposition~\ref{prop:Technical Deformation_II}; see \S\ref{sect:proof_prop:zeroWidth}. 

Finally, we finish this section with a proposition useful in proving the enumerative min-max theorem in \cite{chuLiWang2025fourGenus2}, which is also proved in \S \ref{sect:interpolation}

\begin{prop}\label{prop:pinchOffg0g1}
    Given a closed orientable Riemannian $3$-manifold $(M, {\mathbf g})$,  and integers $0\leq g_0<g_1$, suppose that $\Phi:X\to\GS(M)$ is a Simon--Smith family of genus $\leq g_1$ with $X$ a cubical complex. Let $K$ be an arbitrary compact subset in $X$, and assume $\Phi|_K$ maps into $\cS_{\leq g_0}$.
    Then  there exists a subcomplex $Z\subset X$ (after refinement) whose interior contains $K$ such that  $\Phi|_Z$ can be deformed via pinch-off processes to become some Simon--Smith family of genus $\leq g_0$.
\end{prop}

\section{Proof of Theorem \ref{thm:weakTopoMinMax} via Theorem \ref{thm:repetitiveMinMax}}\label{sect:minMaxThm}

Assuming Theorem \ref{thm:repetitiveMinMax}, let us prove Theorem \ref{thm:weakTopoMinMax} by contradiction: Suppose that no embedded genus $g$ minimal surface has area at most $\max\cH^2\circ\Phi$. This implies, by Choi-Schoen compactness \cite{ChoiSchoen1985compactness}, that for some $\delta_0>0$, no orientable embedded genus $g$ minimal surface has area at most $\max\cH^2\circ\Phi+\delta_0$. Thus, using \cite[Proposition 8.5]{MN21} and Choi-Schoen compactness \cite{ChoiSchoen1985compactness}, we can perturb the metric $\mathbf g$ to ${\mathbf g}'$ such that $(M,{\mathbf g}')$ has the following properties: 
\begin{enumerate}
    \item\label{item:propg'unstable} It has positive Ricci curvature. 
    \item All embedded minimal surfaces are non-degenerate.
    \item\label{item:propg'LinIndep} The areas of the embedded minimal surfaces   are linearly independent over $\mathbb Z$.
    \item\label{item:propg'Nonexist} There does not exist an orientable embedded minimal surface of genus $g$  and   area at most $\max\cH^2_{\mathbf 
 g'}\circ\Phi+\delta_0/2$. 
\end{enumerate}

It is easy to see that now $(M,\bg')$ satisfy the assumptions in Theorem \ref{thm:repetitiveMinMax}.  We claim that applying Theorem \ref{thm:repetitiveMinMax} (to be proven in \S \ref{sect:ProofRepetitiveMinMax}) to $(M,\bg')$, $g$ and $\Phi$, with $r=1$, leads to a contradiction and thus prove Theorem \ref{thm:weakTopoMinMax}. Indeed, using the defining properties (\ref{item:propg'unstable}) - (\ref{item:propg'LinIndep}) of $(M,\mathbf g')$, the assumptions of  Theorem \ref{thm:repetitiveMinMax} are satisfied.  So we can apply Theorem \ref{thm:repetitiveMinMax} and obtain a Simon--Smith family $H:[0,1]\x X\to\cS_{\leq g}$ that satisfies the properties in Theorem \ref{thm:repetitiveMinMax}. However, recall that by property (\ref{item:propg'Nonexist}) of $(M,\bg')$, there does not exist an orientable embedded minimal surface of genus $g$ that has area at most $\max\cH^2_{\mathbf 
 g'}\circ\Phi$. As a result, all those minimal surfaces $\Gamma_1, \dots,\Gamma_l$ produced in  Theorem \ref{thm:repetitiveMinMax} must have genus $\leq g-1$. By item~\eqref{item:ti} of the theorem, $\Phi'|_{D_i}$ maps into  $ \cS_{\leq g-1}$ for each $i=1, \dots,l$. Moreover, $\Phi'|_{D_0}$ has genus $0$, so we know actually the whole family $\Phi'$ maps into $\cS_{\leq g-1} $. This contradicts the assumption of Theorem \ref{thm:weakTopoMinMax} that $\Phi$ cannot be deformed via pinch-off processes to become a map into $\cS_{\leq g-1} $. Therefore, we have proven Theorem \ref{thm:weakTopoMinMax}.

\section{Proof of Theorem \ref{thm:repetitiveMinMax}: Repetitive min-max}\label{sect:ProofRepetitiveMinMax}
    The goal of this section is to prove Theorem  \ref{thm:repetitiveMinMax}.
    
    First, we will repetitively run min-max using the given family $\Phi$, and obtain a number of families. Then, we assemble these families to obtain the desired ``optimal family" $\Phi'$, whose domain is a union of parts $D_0, D_1, \dots, D_l$. The min-max construction together with our genus reduction theorem, Theorem \ref{thm:pinchOffMinMax}, leads to the conclusions of Theorem \ref{thm:repetitiveMinMax}.

\subsection{Repetitive min-max: $k$-th stage}
In this section, we will repetitively run min-max. But before that, let us fix some constants. 
\begin{itemize}
    \item First, we choose a $r'>0$ such that $r'<r$, where $r$ is given in  Theorem  \ref{thm:repetitiveMinMax}.
    \item Let   $d_0$ be the minimum value of $\bF(\Sigma,\Sigma')$ where $\Sigma,\Sigma'$ are any two distinct orientable embedded minimal surfaces of genus $\leq g$ in $(M,\bg)$: Recall there are only finitely many orientable embedded minimal surfaces of genus $\leq g$ by assumption. By choosing a smaller value for $r'$, let us assume $r'<d_0$.
    \item By further shrinking $r'$,  by Lemma \ref{lem:genusLarger}, we can assume that for each of the finitely many orientable embedded minimal surfaces $\Sigma$ of genus $\leq g$, if $S\in\cS(M)$ is such that $\bF(S,\Sigma)<r'$, then $\fg(S)\geq \fg(\Sigma)$.
\end{itemize}

\begin{rmk}
    Throughout \S \ref{sect:ProofRepetitiveMinMax}, a superscript $^k$ or subscript $_k$ would indicate that the object concerned is constructed from running the min-max process for the $k$-th time.
\end{rmk}

We will {\it recursively} describe the process. Let $k=1,2, \dots$. Right before  the $k$-th stage of min-max, we would have on our hands a subcomplex $X^{k-1}$ of $X$ (up to refinement), and a Simon--Smith family
$\Phi^{k-1}:X^{k-1}\to\cS_{\leq g}.$
(Before the first time we run min-max, i.e., $k=1$, we of course have $X^0:=X$ and $\Phi^0:=\Phi$.)
 
For each $k \in \mathbb{N}^+$, we apply the min-max theorem, Theorem~\ref{thm:minMax}, to $\Phi^{k-1}$ and $r'>0$. If the width $\bL(\Lambda(\Phi^{k-1}))$ is positive, then we obtain a varifold $V^k$ in $\cW^k:=\cW_{\bL(\Lambda(\Phi^{k-1})),\leq g}$ (defined in Definition \ref{def:minimal_surf_varifolds}) satisfying Theorem~\ref{thm:minMax} (2); Otherwise, this $k$-th stage of min-max would actually be the last stage, which will be discussed in the subsequent section. So let us assume the width is positive at this stage.

Since $(M, \bg)$ has positive Ricci curvature, the Frankel property holds. Consequently, every varifold in $\mathcal{W}^k$ is associated with a connected surface. Also, we know that $\spt(V^k)$ must be orientable, for otherwise it would be one-sided (as $M$ is orientable) and its double cover would be stable (by Theorem  \ref{thm:minMax} (2) (b)), contradicting the fact that $M$ has positive Ricci curvature. As a result, by Theorem~\ref{thm:minMax} (2), $V^k$ is of multiplicity one. Furthermore, we know the following:

\begin{lem}\label{lem:W_one_elmt} 
   We retain the assumptions in  Theorem  \ref{thm:repetitiveMinMax}. Let $\Psi:Y\to\cS_{\leq g}(M)$ be a Simon--Smith  family whose width $L':=\bL(\Lambda(\Psi))$ satisfies $0<L'<L$. Assume that   set $\cW _{L',\leq g}(M,\bg)$ contains some element $|\Gamma|$ of multiplicity one, where $\Gamma$ is orientable.
    \begin{enumerate}[label=\normalfont(\arabic*)]
        \item If $\fg(\Gamma)=g$, then every member in $\mathcal{W}_{L',\leq g}(M,\bg) $ is induced by  some minimal surface of genus $g$ and  multiplicity one,
        \item otherwise (so that $\fg(\Gamma)<g$), 
        $\mathcal{W}_{L',\leq g}(M,\bg)   = \{|\Gamma|\}$.
    \end{enumerate}
    In particular, $\mathcal{W}_{L,\leq g} (M,\bg)$ consists of finitely many varifolds, each associated with a multiplicity one orientable minimal surface.
\end{lem} 
\begin{proof}
In case (1), using the third bullet point of Condition \ref{cond:A} regarding  linear independence of area,   every $V\in \cW^G_{L',\leq g}(M,\bg)$ must be of multiplicity one and necessarily $\spt V $ has the same genus as $\Sigma$.  The second case can be proven similarly. The finiteness of $\cW^G_{L',\leq g}(M,\bg)$ follows from the assumption in  Theorem \ref{thm:repetitiveMinMax} that there are only finitely many orientable embedded minimal surfaces of genus $\leq g$.
\end{proof}

Using this lemma, there are two cases regarding $V^k\in\cW^k:=\cW_{\bL(\Lambda(\Phi^{k-1})),\leq g}$:
    \begin{itemize}
    \item[Case 1.]   $\cW^{k}$ consists of finitely many elements, each  associated with a multiplicity one, orientable, embedded minimal surface of genus $g$; Denote them as $T^k_1, \dots,T^k_{n_k} \in \cS_g$;
        
    \item[Case 2.] $\cW^{k}=\{V^{k}\}$, and $V^{k}$ is associated with a multiplicity one, orientable, embedded minimal surface of genus $\leq g-1$, which must be non-degenerate as assumed in Theorem \ref{thm:repetitiveMinMax}. Let $T^k\in \cS_{\leq g-1}$ be the associated element,
        
    \end{itemize}
Let us analyze them separately.

\subsubsection{Case 1}  Applying
 Theorem \ref{thm:currentsCloseInBoldF} to $\Phi^{k-1}:X^{k-1}\to\cS(M)$ and $r'>0$, 
    we can obtain some $\tilde\Phi^k\in\Lambda(\Phi^{k-1})$ and  $\delta_k>0$ such that 
    $$\cH^2(\tilde \Phi^k(x))\geq \bL(\Lambda(\Phi^{k-1}))-\delta_k\implies\tilde
        \Phi^k(x)\in\bB^\bF_{r'}(\cW^k).$$
    Hence, we can choose some subcomplex $\hat E^k$ of $X^{k-1}$ (after refining $X^{k-1}$), which we call a ``cap",
 such that:
 \begin{itemize}
     \item $\cH^2(\tilde \Phi^k(x))\geq \bL(\Lambda(\Phi^{k-1}))-\delta_k/2 \implies x\in \hat E^k,$
     \item $x\in \hat E^k \implies \tilde
        \Phi^k(x)\in\bB^\bF_{r'}(\cW^k).$
 \end{itemize} 
    In addition, for future purpose, let us choose a slightly larger subcomplex $E^k \supset \hat E^k$ in $X^{k-1}$ (after refinement) such that $ E^k$ also satisfies these two properties, and that there exists a smooth cut-off function $\eta_k:X^{}\to[0,1]$ such that $\eta_k=1$ on $\overline{X\backslash E^k}$ while  $\eta_k=0$ on $ \hat E^k$.
         
    Since $r'<d_0$, $E^k$ is a {\it disjoint} union of subcomplexes $E^k_1, \dots,E^k_{n_k}$ such that, for each $j=1, \dots,n_k$, there exists $T^k_j \in \mathcal{W}^k$ satisfying
   $\tilde \Phi^k(E^{k}_j)\subset \bB^\bF_{r'}(T^k_j)\subset \bB^\bF_{r}(T^k_j)\,.$
    In fact, since $r'$ is less than the constant given by Lemma \ref{lem:genusLarger} (applied to any $T^k_j$), we know that every member of $\tilde \Phi^k|_{E^k}$ is of genus $g$. Hence, we have 
        $\tilde \Phi^k(E^{k}_j)\subset  \cS_g\cap \bB^\bF_{r}(T^k_j).$
    In addition, recalling by definition $\tilde\Phi^k$ is in the same Simon--Smith homotopy class as $\Phi^{k-1}$, we let $H^k:[0,1]\x X^{k-1}\to \cS_{\leq g}$  be  such that $H^k(0,\cdot)=\Phi^{k-1}$,  $H^k(1,\cdot)=\tilde\Phi^k$, and $t\mapsto H^k(t,y)$ is a smooth deformation for every $y$. 
  
\subsubsection{Case 2}
    In this case, we let $T^k\in\cS_{\leq g-1}$ be the element associated to the varifold $V^k$, and  apply Theorem~\ref{thm:pinchOffMinMax} to $\Phi^{k-1}:X^{k-1}\to\cS(M)$ and  $r$. We obtain: 
    \begin{itemize}
        \item a Simon--Smith family of genus $\leq g$, 
       $\tilde\Phi^k:X^{k-1}\to\cS_{\leq g}\,,$
        \item a Simon--Smith family of genus $\leq g$, $
            H^k: [0,1]\x  X^{k-1} \to \GS_{\leq g},$
        \item a constant $\delta_k>0$,
    \end{itemize}
    satisfying the following:
    \begin{itemize}
    \item $ H^k(0,\cdot)=\Phi^{k-1}$, $H^1(1,\cdot)=\tilde\Phi^k$,
        \item If $\cH^2(\tilde\Phi^k(x))\geq \bL(\Lambda(\Phi^{k-1}))-\delta_k$, then
        $\tilde\Phi^k(x)\in  \cS_{\fg(T^k)}\cap\bB^{\bF}_{r}(T^k)$.
        \item $t\mapsto H^k(t,x)$ is a pinch-off process for each $x\in X^{k-1}$.
    \end{itemize}
    Clearly, we can choose some subcomplex $\hat E^k\subset X^{k-1}$, which we call a ``cap",  
     such that 
     \begin{itemize}
         \item $\cH^2(\tilde \Phi^k(x))\geq \bL(\Lambda(\Phi^{k-1}))-\delta_k/2 \implies x\in \hat E^k,$
         \item $x\in  \hat E^k \implies \tilde
        \Phi^k(x)\in  \cS_{\fg(T^k)}\cap \bB^\bF_{r}(T^k).$
     \end{itemize}
    As in Case 1, for future use, we choose a slightly larger subcomplex $ E^k\supset \hat E^k$ in $X^{k-1}$ (after refinement) such that $E^k$ also satisfies these two properties, and that there exists a smooth cut-off function $\eta_k:X^{}\to[0,1]$ such that $\eta_k=1$ on $\overline{X\setminus E^k}$ while  $\eta_k=0$ on $ \hat E^k$.

\subsubsection{Family for running the $(k+1)$-th min-max}

Now, we have described the objects we  obtain in the $k$-th stage of min-max, in both Case 1 and 2. To enter the next stage, we need to construct a new family. Namely, we define
$X^k:=\overline{X^{k-1}\backslash \hat E^k}, $ and $\Phi^k:=\tilde\Phi^k|_{X^k}.$ 
Note, by how we defined $\hat E^k$, it is guaranteed that the new width $\bL(\Lambda(\Phi^{k}))$ is strictly less than $\bL(\Lambda(\Phi^{k-1}))$.
 Combined with the assumption that $(M,\bg)$ has only finitely many orientable embedded minimal surfaces of genus $\leq g$, our repetitive min-max process must terminate in finitely many steps. Specifically, for some $K \geq 2$, when we apply the min-max process for the $K$-th time using the family $\Phi^{K-1}$, the width $\bL(\Lambda(\Phi^{K-1}))$ is zero. 

\subsection{The last stage}
    We now have that the width $\bL(\Lambda(\Phi^{K-1}))$ is zero. Thus, using Theorem \ref{prop:zeroWidth}, we have:
\begin{itemize}
    \item  a Simon--Smith  family of genus 0, $\Phi^K: X^{K-1} \to  \GS_{0}$,
    \item a Simon-Simon family of genus $\leq g$, $H^K:[0,1]\x  X^{K-1}\to\cS_{\leq g}$,
\end{itemize}
such that:
\begin{itemize} 
                \item $H^K(0,\cdot)=\Phi^{K-1}$ and $H^K(1,\cdot) = \Phi^K$.
                \item For all $x \in X^{K-1}$, $t \mapsto H^K(t, x)$ is a  pinch-off process. 
        \end{itemize}

\subsection{An optimal family $\Phi'$}\label{subsect:new_family_Xi}
    In this subsection, we will use the families defined above to reconstruct a new Simon--Smith family $\Phi'$ of genus $\leq g$, with parameter space $X$, and a deformation via pinch-off processes, $H$, from $\Phi$ to $\Phi'$. Readers may refer to Figure \ref{fig:Xi} for a schematic of the families constructed.
    
    \begin{figure}[!ht]
        \centering
        \includegraphics[width=0.6\textwidth]{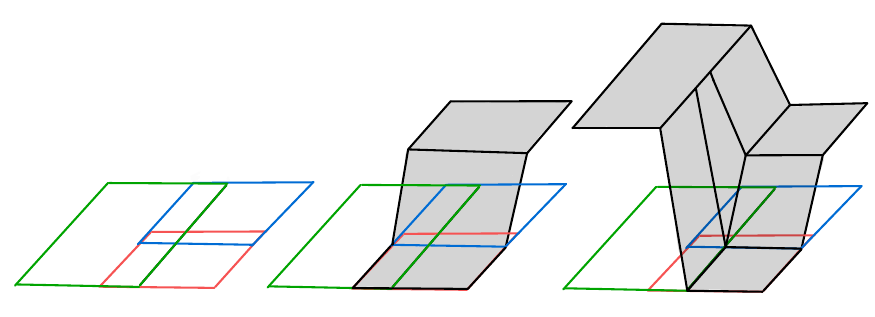}
        \caption{In the first picture, the red region is $X^{K-1}$, the blue is $E^{K-1}$, and the green is $E^{K-2}$. In the second picture, the gray surface (with suitable smoothing) denotes $\Xi^{K-1}$. In the last picture, the gray surface denotes $\Xi^{K-2}$.}
        \label{fig:Xi}
    \end{figure} 

\begin{itemize}
    \item First, we take the last family, $\Phi^K$, which has domain $X^{K-1}=\overline{\tilde X^{K-2}\backslash \hat E^{K-1}}$. We ``lift" $\Phi^K$ on  ``bridge region" $\overline{E^{K-1}\backslash\hat E^{K-1}}$, using the deformation $H^K$ and the cut-off $\eta_{K-1}$. More precisely, we consider the family 
\begin{equation*} 
         H^K(\eta_{K-1}(x),x), \quad x\in  X^{K-1}.
    \end{equation*} Note, for $x$ at the ``boundary" $X^{K-1}\cap\hat E^{K-1}$, the expression above becomes $H^K(0,x)=\Phi^{K-1}(x)$. 
    \item As a result, this  family  can be glued with the family $\tilde\Phi^{K-1}|_{\hat E^{K-1}}$ along $X^{K-1}\cap\hat E^{K-1}$. Then, the resulting family, denoted $\Xi^{K-1}$, has domain $X^{K-2}$. Now, it is easy to see that there exists a deformation via pinch-off processes from $\tilde \Phi^{K-1}$ to $\Xi^{K-1}$. Recall we have also the deformation via pinch-off processes $H^{K-1}$ from $\Phi^{K-2}$ to $\tilde\Phi^{K-1}$. Thus, let us concatenate these two processes, to obtain a deformation via pinch-off processes 
$$G^{K-1}:[0,1]\x  X^{K-2}\to \cS_{\leq g}$$such that $G^{K-1}(0,\cdot)=\Phi^{K-2}$ and $G^{K-1}(1,\cdot)=\Xi^{K-1}$.
    \item Now, we ``lift" $\Xi^{K-1}$ on the ``bridge region" $\overline{E^{K-2}\backslash\hat E^{K-2}}$, using the deformation $G^{K-1}$ and the cut-off $\eta_{K-2}$. More precisely, we consider the family 
\begin{equation*} 
        G^{K-1}(\eta_{K-2}(x),x), \quad x\in  X^{K-2}.
    \end{equation*} Note, for $x$ at the ``boundary" $X^{K-2}\cap\hat E^{K-2}$, the expression above becomes $G^{K-1}(0,x)=\Phi^{K-2}(x)$. 
    \item As a result, this  family  can be glued with the family $\tilde\Phi^{K-2}|_{\hat E^{K-2}}$ along $X^{K-2}\cap\hat E^{K-2}$. The resulting family, denoted $\Xi^{K-2}$ has domain $X^{K-3}$. Now, as before, we know there exists  a deformation via pinch-off processes 
$$G^{K-2}:[0,1]\x  X^{K-3}\to \cS_{\leq g}$$such that $G^{K-2}(0,\cdot)=\Phi^{K-3}$ and $G^{K-1}=\Xi^{K-2}$.
\end{itemize}

We repeat this process. At the final step, we would be gluing some family of the form $G^{2}(\eta_{1}(x),x)$, $x\in  X^{1}$
    to the family $\tilde\Phi^1|_{\hat E^1}$, and obtain a family $\Xi^1$ with domain $X^1\cup \hat E^1=X$. Now, we define the desired ``optimal family" $\Phi'$ to be $\Xi^1$. By construction,  it is easy to see that there exists some deformation via pinch-off processes
    $G^1:[0,1]\x X\to\cS_{\leq g}$ 
such that $G^1(0,\cdot)=\Phi$ and $G^1(1,\cdot)=\Phi'$. 

In the following, we will denote $G^1$ by $H$ instead, to align with the notation in Theorem \ref{thm:repetitiveMinMax}.

\subsection{Completion of the proof}
 
For simplicity, we now collect all the minimal surfaces $T^k$ and $T^k_j$ together, and relabel them $\Gamma_1,\Gamma_2,\dots,\Gamma_l,$
but in an order that {\it reverses} the order they were obtained in the min-max process (i.e. if $i\leq j$, then $\Gamma_i$ was obtained in a later stage of min-max than $\Gamma_j$). Note they all must have area $\leq \bL(\Lambda(\Phi))$. At the same time, we relabel their corresponding associated caps $E^k$ and $E^k_j$ as
$\tilde{D}_1,\tilde D_2,\dots,\tilde D_l.$
We also set $\tilde{D}_0 := X^{K-1}\,.$
Finally, for each $i = 0, 1, 2, \cdots, l$, we define
\[
    D_i := \overline{\tilde{D}_i \setminus \bigcup^l_{j = i + 1} \tilde D_j}\,.
\]
In other words, if $\tilde{D}_i$ was obtained at the $k$-th stage of min-max, then we remove from it all the ``bridge parts" coming from the $k$-th stage or earlier   (see Figure \ref{fig:Xi}). Moreover, if some $D_i$ is empty, we remove $D_i$ from the collection $\{D_j\}$ and the corresponding $\Gamma_i$ from the set $\{\Gamma_j\}$, and then relabel the remaining sets accordingly.

Let us finish the proof of Theorem \ref{thm:repetitiveMinMax}. The fact that $H(0,\cdot)=\Phi$ was already explained at the end of the previous section. 
We fix an $i=1, \dots,l$, and suppose the cap $\tilde{D}_i$ was defined in the $k$-th stage of min-max. Then, by how we defined the caps $E^k$ and $E^k_j$, we have $\tilde\Phi^k|_{\tilde{D}_i}\in\cS_{\fg(\Gamma_i)}\cap\bB^\bF_r(\Gamma_i)$, where $\Gamma_i$ is the minimal surface associated with the cap $\tilde{D}_i$. Note that $D_i$ is constructed by removing the ``bridge parts" above $\tilde{D}_i$, and thus, the family $\Phi|_{D_i}$ can be deformed via pinch-off processes to become $\tilde\Phi^k|_{D_i}$, which in turn can be deformed via pinch-off processes to become $\Phi'|_{D_i}$, and both deformations appear as subfamilies in $H$. Then, item (\ref{item:ti})  follows immediately. Moreover, this also imply that $\Phi'|_{D_i}$ must be of genus $\leq \fg(\Gamma_i)$, which is the first part of item (\ref{item:optimal}). Finally, $\Phi'|_{D_0}$ is exactly $\Phi^K|_{D_0}$, so $\Phi'|_{D_0}$ is of genus $0$. Thus, item (\ref{item:optimal}) is true.
This finishes the proof of Theorem \ref{thm:repetitiveMinMax}.



\begin{rmk}\label{rmk:leftIntact}
    In the above construction, we concatenated multiple pinch-off processes together. Namely, the ending surface of one process is the beginning of another. If the ending surface of the first process has a cone point, this cone point would be preserved in the second process. This explains why in Definition \ref{defn:pinch_off} (2) we need to allow the  cone to be left intact.
\end{rmk}

\section{Existence of homologically independent simple short loops} \label{sect:short_loops}
    Let $(M, \bg)$ be an orientable closed Riemannian $3$-manifold. We let $\Gamma$ and $\Sigma$ denote smooth embedded orientable closed surfaces in $M$, and $S$ denote a punctate surface in $\GS(M)$. The main goal of these sections is to locate homologically independent simple short loops in a smooth surface and small balls covering nontrivial topology in a punctate surface. 
    
    Recall that the \emph{systole} of a surface $\Gamma$, denoted by $\operatorname{sys} \pi_1(\Gamma)$, is the length of the least length of a noncontractible loop in $\Gamma$. In~\cite{Gro92}, Gromov introduced the notion of \emph{$\Z_2$ homological systole} of $\Gamma$, denoted by $\operatorname{sys} H_1(\Gamma; \Z_2)$, defined as the least length of a loop in $\Gamma$, which represents a non-trivial class in $H_1(\Gamma; \Z_2)$. The remarkable \emph{systolic inequality} by Loewner and Gromov states that if the area of $\Gamma$ is small, then the $\Z_2$ homological systole is also small; see Theorem~\ref{thm:systolic_ineq}. By the standard minimization process in the calculus of variations, the systole can be realized by a simple loop ---specifically, simple closed geodesic. 

    In general, if the genus $\fg(\Gamma) > 1$ and the area $\mathcal{H}^2(\Gamma)$ is small, one expects to find $\fg(\Gamma)$ homologically independent simple short loops in $\Gamma$. This is closely related to \emph{Bers' constant} for hyperbolic surfaces \cite{Bers_74_degenerating_surf,Bers_85_inequality_surf}. A natural idea is to perform a neck-pinch surgery along the detected loop in $M$ realizing the systole and then employ an inductive argument; see Proposition~\ref{prop:short_loops_I} in \S\ref{subsect:short_loops_small_area} 

    In \S\ref{subsect:short_loops_large_area}, we show that in $(M, \bg)$, for $g_1 \geq g_0 \in \N$, a smooth embedded orientable closed surface $\Gamma$ of genus $g_1$ must contain $g_1 - g_0$ disjoint simple short loops which are homologically independent, provided that it is close to a smooth embedded orientable closed surface $\Sigma$ of genus $g_0$ with respect to the $\bF$-metric. 
    
    If $\fg(\Gamma) > \fg(\Sigma)$, and $\Gamma$ lies in a tubular neighborhood $N_\eta(\Sigma)$ of $\Sigma$, then the map
    $f: \Gamma \hookrightarrow N_\eta(\Sigma) \to \Sigma$
    induces a non-injective map
        $f_*: \pi_1(\Gamma) \to \pi_1(\Sigma)\,.$
    It follows from Gabai's resolution of the so-called \emph{Simple Loop Conjecture}~\cite[Theorem~2.1]{Gabai_85_simple_loop} that there exists a non-contractible simple closed curve $\alpha \subset \Gamma$ such that $f|_\alpha$ is homotopically trivial. In particular, together with Dehn's lemma, $\alpha$ bounds a compressing disk in $M$. However, this topological result does not give us any control on the length of the loop, and our result can be regarded as a quantitative and homological generalization of this simple loop conjecture.

    Heuristically, if $\Gamma$ and $\Sigma$ are sufficiently close with respect to the $\bF$-metric, a large portion $\Gamma'$ of $\Gamma$ is homeomorphic to $\Sigma$ removing finitely many disks. Hence, $\Gamma \setminus \Gamma'$ has small area, and one may appeal to Proposition~\ref{prop:short_loops_I} to detect multiple simple short loops which lie outside $\Gamma'$, provided that its boundary also has small length. The technical details will be presented in the proof of Proposition~\ref{prop:short_loops_II}, where we also need to deal with the case where some thin parts, such as \emph{filigrees}, of $\Gamma$ are not inside a tubular neighborhood of $\Sigma$.

    Next, in \S\ref{subsect:small_balls_punctate_surf}, we extend this result to a genus $g_1$ punctate surface. However, as illustrated in Appendix~\ref{sect:PS_long_loops}, one cannot expect these loops to have small lengths; instead, they are contained in $3$-balls of small radius, i.e., small balls. 

\subsection{Simple short loops in surfaces with small area}\label{subsect:short_loops_small_area} 

    If the surface $\Gamma$ is an embedded surface, then in general, the simple short loop realizing the systole might be knotted in $M$, and in this case, there is no compressing disk to perform the surgery in the ambient manifold. Therefore, in this subsection, $\Gamma$ is always regarded as an \emph{intrinsic surface} so that we can carry out the inductive argument. 

    \begin{prop}[Existence of homologically independent simple short loops I]\label{prop:short_loops_I} 
        For any $g \in \mathbb{N}$ and $\varepsilon > 0$, there exists $\delta = \delta(g, \varepsilon) \in (0, \varepsilon)$ with the following property.
        For any smooth orientable closed surface $\Gamma$ containing a finite set of disjoint closed disks $\{\overline{D_i}\}^n_{i = 1}$ in $\Gamma$ with piecewise smooth boundary satisfying 
        \[
            \mathfrak{g}(\Gamma) \leq g\,, \quad \mathcal{H}^2(\Gamma) < \delta\,, \quad \sum^n_{i = 1} \mathcal{H}^1(\partial D_i) < \delta\,,
        \]
        there exist $\mathfrak{g}(\Gamma)$ disjoint simple loops $\{c_j\}^{\mathfrak{g}(\Gamma)}_{j = 1}$ in $\Gamma \setminus \bigcup^n_{i = 1} \overline{D_i}$ such that:
        \begin{enumerate}[label=\normalfont(\arabic*)]
            \item\label{item:short_loops_I_short} Each $c_j$ has length $\mathcal{H}^1(c_j) \leq \varepsilon$.
            \item\label{item:short_loops_I_indep} $\{[c_j]\}^{\mathfrak{g}(\Gamma)}_{j = 1}$ is linearly independent in $H_1(\Gamma; \Z_2)$.
        \end{enumerate}
        In particular, $\Gamma \setminus \bigcup^{\mathfrak{g}(\Gamma)}_{j = 1} c_j$ is a genus $0$ surface with boundary.
    \end{prop} 
    \begin{rmk} 
        In this proposition, as well as Proposition~\ref{prop:short_loops_II}, we specify a finite set of disks that serves as \emph{forbidden regions} for the loops. This technical assumption allows us to apply the propositions to punctate surfaces by removing neighborhoods of their punctate sets.
    \end{rmk} 
    \begin{rmk} 
        By Corollary~\ref{cor:Pitts_bF_flat_quant}, in $(M, \bg)$, small $\mathcal{H}^2(\Gamma)$ is equivalent to the assumption that $\bF(\Gamma, \emptyset)$ is small. Hence, if we regard $\emptyset$ as a degenerate genus $0$ surface, then the condition $\mathcal{H}^2(\Gamma) < \delta$ can be replaced by $\bF(\Gamma, \emptyset) < \delta$ and the proposition is a special case of Proposition~\ref{prop:short_loops_II} with $\Sigma = \emptyset$.
    \end{rmk} 

    Note that without the forbidden regions condition, this follows immediately from a quantitative estimate on the lengths of pants decomposition \cite[Proposition~6.3]{Balacheff_Parlier_Sabourau_12_short_loop_decomp}. In the following, we prove it by applying the systolic inequality inductively and a technical result that allows loops to bypass the forbidden regions. 
    
    \begin{thm}[Loewner torus inequality~{\cite[Section~3]{Pu52}}, Gromov's homological systolic inequality~{\cite[Section~2.C]{Gro92}}]\label{thm:systolic_ineq} 
        Given a smooth orientable closed surface $\Gamma$ of positive genus, there exists a universal constant $C > 0$ such that
        \[
            \operatorname{sys} H_1(\Gamma; \Z_2) := \inf \left\{\mathcal{H}^1(c) : [c] \neq 0 \in H_1(\Gamma; \Z_2)\right\} \leq C \sqrt{\mathcal{H}^2(\Gamma)}\,.
        \]
        In particular, there exists a simple closed geodesic $c \in \Gamma$ such that $[c] \neq 0 \in H_1(\Gamma; \Z_2)$ and $\mathcal{H}^1(c) \leq C \sqrt{\mathcal{H}^2(\Gamma)}$.
    \end{thm}
    \begin{rmk}
        In \cite{Gro92}, Gromov proved an optimal estimate that there exists a universal constant $\tilde C > 0$ such that for any $\Gamma$ of genus $g \geq 2$, $\operatorname{sys} H_1(\Gamma; \Z_2) \leq \tilde C {\log g} \sqrt{\mathcal{H}^2(\Gamma)/g}\,.$
        This optimal result is not needed in our paper.
    \end{rmk}

    \begin{lem}[Bypassing the forbidden regions]\label{lem:bypass_forbid}
        Let $\alpha$ be a positive real number. Suppose that in a smooth orientable closed surface $\Gamma$, there exists a simple loop $c \subset \Gamma$ and a finite set of disjoint closed disks $\{\overline{D_i}\}^n_{i = 1}$ in $\Gamma$ with piecewise smooth boundary satisfying
            $\sum^n_{i = 1} \mathcal{H}^1(\partial D_i) < \alpha\,.$
        Then there exists a finite union of disjoint simple loops $\tilde c$ in $\Gamma \setminus \bigcup^n_{i = 1} \overline{D_i}$ such that
        $[\tilde c] = [c]$ in $ H_1(\Gamma; \Z_2)$ and $ \mathcal{H}^1(\tilde c) < \mathcal{H}^1(c) + \alpha\,.$
    \end{lem}
    \begin{proof}
        Since each disk has piecewise smooth boundary, by Sard's theorem, for each disk $D_i$, we can choose a larger disk $D'_i \supset \overline{D_i}$ such that $\{\overline{D'_i}\}^n_{i = 1}$ are also pairwise disjoint, $c$ intersects $\partial D'_i$ transversally, and 
           $ \sum^n_{i = 1} \mathcal{H}^1(\partial D'_i) < \alpha\,.$

        Since each $D'_i$ is contractible, $c$ separates $D'_i$ into two sets $E'_{i,1}$ and $E'_{i, 2}$ such that
            $\partial E'_{i, 1} \cap D'_i = \partial E'_{i, 2} \cap D'_i = c \cap D'_i\,.$
        We define (see figure~\ref{fig:construct_tilde_c})
        \[
            \tilde c = (c \setminus \bigcup^n_{i = 1} (\partial E'_{i, 1} \cap D'_i) ) \cup \bigcup^n_{i = 1} (\partial D'_i \cap \partial E'_{i, 1})\,.
        \]

        \begin{figure}
            \centering
            \includegraphics[width=0.5\textwidth]{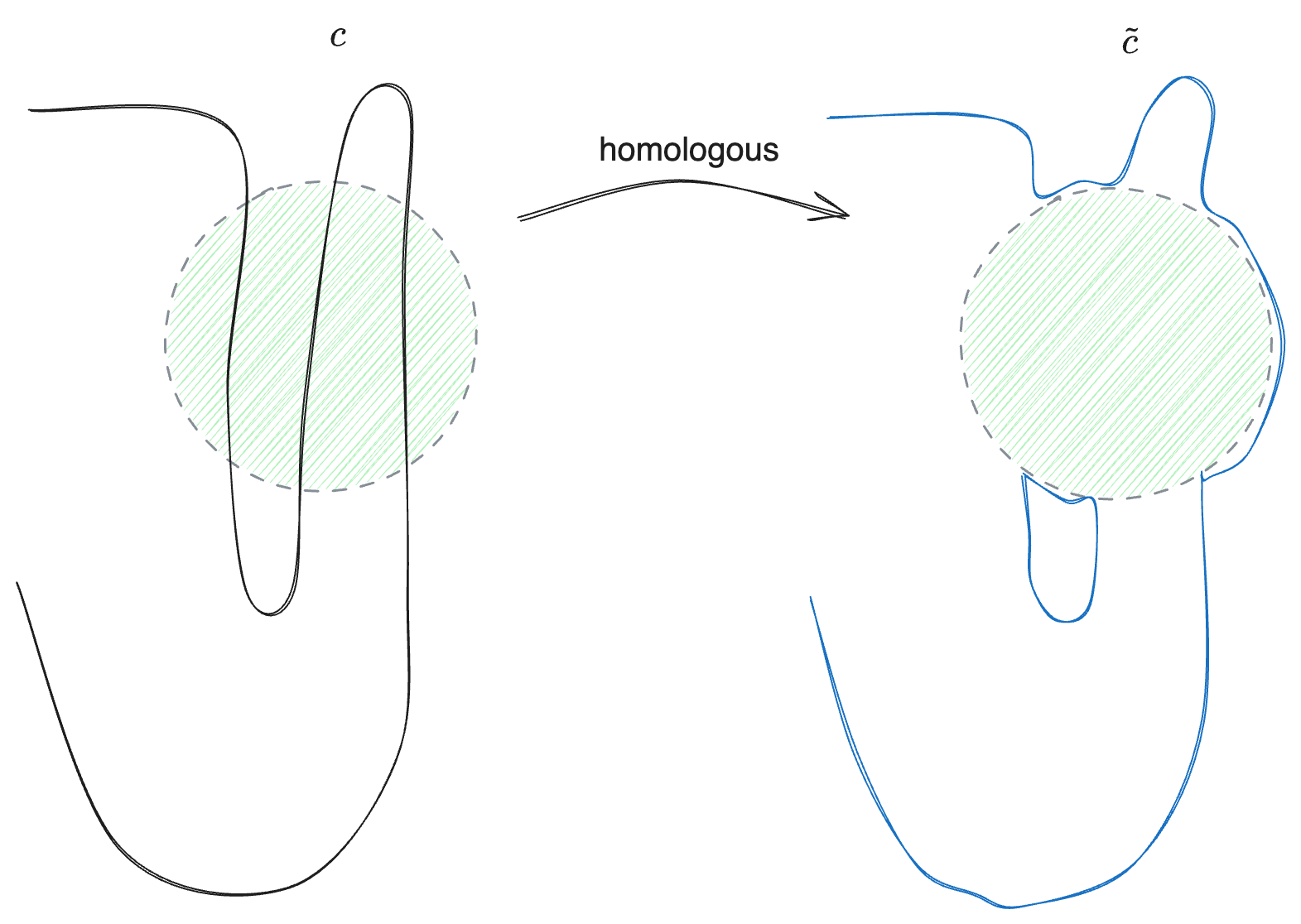}
            \caption{Constructing $\tilde c$ from $c$}
            \label{fig:construct_tilde_c}
        \end{figure}

        Clearly, by the transversality, $\tilde c$ is a union of disjoint simple loops in $\Gamma \setminus \bigcup^n_{i = 1} \overline{D_i}$ satisfying $[\tilde c] = [c] + \partial [\cup_i E'_{i,1}] = [c]$ in   $H_1(\Gamma; \Z_2)$,
        and
        \[
            \mathcal{H}^1(\tilde c) \leq \mathcal{H}^1(c) + \sum^n_{i = 1} \mathcal{H}^1(\partial D'_i \cap \partial E'_{i, 1}) \leq \mathcal{H}^1(c) + \sum^n_{i = 1} \mathcal{H}^1(\partial D'_i) < \mathcal{H}^1(c) + \alpha\,.
        \]
    \end{proof}

    \begin{proof}[Proof of Proposition~\ref{prop:short_loops_I}]
        If $\mathfrak{g}(\Gamma) = 0$, the proposition holds trivially. In the following, we fix $g \in \mathbb{N}^+$ and only consider the case where $\mathfrak{g}(\Gamma) \in \mathbb{N}^+$. 
        
        It suffices to prove that for any $g' \in \{1, 2, \dots, g\}$ and $\varepsilon > 0$, there exists $\delta'(g, g', \varepsilon) \in (0, \varepsilon)$ with the following property: For any smooth orientable closed surface $\Gamma$ containing disjoint closed disks $\{\overline{D_i}\}^n_{i = 1}$ with piecewise smooth boundary, satisfying:
        \begin{enumerate}[label = \normalfont(\roman*)']
            \item\label{item:short_loops_I_genus'} $g' \leq \mathfrak{g}(\Gamma) \leq g$,
            \item\label{item:short_loops_I_area'} $\mathcal{H}^2(\Gamma) < \delta'(g, g', \varepsilon)$, 
            \item\label{item:short_loops_I_bdry'} $\sum^n_{i = 1} \mathcal{H}^1(\partial D_i) < \delta'(g, g', \varepsilon)$,
        \end{enumerate}
        there exist $g'$ disjoint simple loops $\{c_j\}^{g'}_{j = 1}$ in $\Gamma \setminus \bigcup^n_{i = 1} \overline{D_i}$ such that:
        \begin{enumerate}[label=\normalfont(\arabic*)']
            \item\label{item:short_loops_I_short'} Each $c_j$ has length $\mathcal{H}^1(c_j) \leq \varepsilon$.
            \item\label{item:short_loops_I_indep'} $\{[c_j]\}^{g'}_{j = 1}$ is linearly independent in $H_1(\Gamma; \Z_2)$.
        \end{enumerate} 
        
        We prove this existence result inductively on $g'$. For $g' = 1$, we define $\delta'(g, 1, \varepsilon) := \min({\varepsilon}/{2}, {\varepsilon^2}/({4C^2}))\,,$
        where $C$ is the universal constant given in Theorem~\ref{thm:systolic_ineq}. By Theorem~\ref{thm:systolic_ineq}, for any $\Gamma$ and $\{\overline{D_i}\}^n_{i = 1}$ satisfying~\ref{item:short_loops_I_genus'}, ~\ref{item:short_loops_I_area'} and ~\ref{item:short_loops_I_bdry'} there exists a loop $c'_1 \subset \Gamma$ such that
        \[
            \mathcal{H}^1(c'_1) \leq C\sqrt{\mathcal{H}^2(\Gamma)} \leq \frac{\varepsilon}{2}\,, \quad [c'_1] \neq 0\in H_1(\Gamma; \Z_2)\,.
        \]
        By Lemma~\ref{lem:bypass_forbid}, $c'_1$ is homologous to a finite union of disjoint simple loops $\tilde c_1$ lying outside $\{\overline{D_i}\}^n_{i = 1}$, and 
        \[
            \mathcal{H}^1(\tilde c_1) < \delta'(g, 1, \varepsilon) + \mathcal{H}^1(c'_1) < \varepsilon\,, \quad [\tilde c_1] = [c'_1] \neq 0 \in H_1(\Gamma; \Z_2)\,.
        \]
        Hence, there exists at least one loop in $\tilde c_1$, denoted by $c_1$, such that $\mathcal{H}^1(c_1) < \varepsilon\,,$ and $ [c_1] \neq 0 \in H_1(\Gamma; \Z_2)\,.$
        
        Inductively, assume the existence of $\delta'(g, g', \varepsilon)$ for any $\varepsilon$ and $g' = k \in \{1, 2, \dots, g - 1\}$. We define 
        \[
            \delta'(g, k + 1, \varepsilon) := \min({\varepsilon}/{4}, {\varepsilon^2}/({8C^2}), \delta'(g, k, {\varepsilon}/({8k})) \,)\,.
        \]
        By the inductive assumption, for any $\Gamma$ and $\{\overline{D_i}\}^n_{i = 1}$ satisfying ~\ref{item:short_loops_I_genus'}, ~\ref{item:short_loops_I_area'} and ~\ref{item:short_loops_I_bdry'}, there exist $k$ disjoint simple loops $\{c_j\}^k_{j = 1}$ in $\Gamma \setminus \bigcup^n_{i = 1} \overline{D_i}$ satisfying ~\ref{item:short_loops_I_short'} and~\ref{item:short_loops_I_indep'} with $\varepsilon$ replaced by $\frac{\varepsilon}{8k}$. 
        
        We perform a sequence of surgeries on $\Gamma$ along $\{c_j\}^k_{j = 1}$---cut open along these disjoint loops and attach $2k$ disks, denoted by $\{D_i\}^{n + 2k}_{i = n + 1}$, of area small enough such that the resulting smooth orientable closed surface $\Gamma'$ of positive genus has
        \begin{equation}\label{eqn:area_Sigma'}
            \mathcal{H}^2(\Gamma') \leq \mathcal{H}^2(\Gamma) + \frac{\varepsilon^2}{8C^2} < \delta'(g, k + 1, \varepsilon) + \frac{\varepsilon^2}{8C^2} \leq \frac{\varepsilon^2}{4C^2}\,.
        \end{equation}
        Moreover, since each $c_i$ has $\mathcal{H}^1(c_i) \leq \frac{\varepsilon}{8k}$, the $(n + 2k)$ disks $\{D_i\}^{n + 2k}_{i = 1}$ in $\Gamma'$ satisfy
        \begin{equation}\label{eqn:loop_sum_est}
            \sum^{n + 2k}_{i = 1} \mathcal{H}^1(\partial D_i) = \sum^{n}_{i = 1} \mathcal{H}^1(\partial D_i) + 2k \frac{\varepsilon}{8k} < \delta'(g, k + 1, \varepsilon) + \frac{\varepsilon}{4} \leq \frac{\varepsilon}{2}\,.
        \end{equation}
        Applying Theorem~\ref{thm:systolic_ineq} again to $\Gamma'$, together with ~\eqref{eqn:area_Sigma'}, we obtain a loop $c'_{k + 1} \subset \Gamma'$ with length
        \[
            \mathcal{H}^1(c'_{k + 1}) \leq C \sqrt{\mathcal{H}^2(\Gamma')} \leq  {\varepsilon}/{2}\,.
        \]
        By Lemma~\ref{lem:bypass_forbid} with~\eqref{eqn:loop_sum_est}, $c'_{k + 1}$ is homologous to a finite union of loops $\tilde c_{k+1} \subset \Gamma' \setminus \bigcup^{n + 2k}_{i = 1} \overline{D_i}$ with 
        \[
            \mathcal{H}^1(\tilde c_{k + 1}) < \mathcal{H}^1(c'_{k + 1}) + \frac{\varepsilon}{2} \leq \varepsilon\,, \quad [\tilde c_{k + 1}] = [c'_{k + 1}] \neq 0 \in H_1(\Gamma'; \Z_2)\,.
        \] 
        Hence, there exists a simple loop $c_{k + 1}$ in $\tilde c_{k + 1}$ with $[c_{k+1}] \neq 0 \in H_1(\Gamma'; \Z_2)$, which implies that $[c_{k + 1}] $ is linearly independent of $[c_1], [c_2], \dots, [c_k]$ in $H_1(\Gamma; \Z_2)$. Notice that 
        \[
            \Gamma' \setminus \bigcup^{n + 2k}_{i = 1} \overline{D_i} = \Gamma \setminus \bigcup^{n}_{i = 1} \overline{D_i}\,,
        \] 
        so $c_{k + 1}$ is also a simple loop in $\Gamma \setminus \bigcup^{n}_{i = 1} \overline{D_i}$. Therefore, we construct $\{c_j\}^{k + 1}_{j = 1}$ satisfying ~\ref{item:short_loops_I_short'} and~\ref{item:short_loops_I_indep'}. 

        This closes the induction, and the proposition holds with
            $\delta(g, \varepsilon) := \min_{1 \leq g' \leq g}\delta'(g, g', \varepsilon)\,.$
    \end{proof} 

\subsection{Simple short loops in surfaces with large area}\label{subsect:short_loops_large_area} 

    In this subsection, we first locate short loops for smooth surfaces. 

    \begin{prop}[Existence of homologically independent simple short loops II] \label{prop:short_loops_II} 
        In an orientable closed Riemannian $3$-manifold $(M, \bg)$, let $\Sigma$ be an embedded smooth orientable genus $g_0$ closed surface. For any $\varepsilon\in (0,1)$ and any integer $g_1 \geq g_0$, there exists $\delta = \delta(M, \bg, \Sigma, g_1, \varepsilon) \in (0, \varepsilon)$ with the following property. 
        In $(M, \bg)$, for any embedded smooth orientable closed surface $\Gamma$ containing a finite set of disjoint closed disks $\{\overline{D_i}\}^n_{i = 1}$ with piecewise smooth boundary satisfying:
        \begin{enumerate}[label = \normalfont(\roman*)]
            \item\label{item:short_loops_II_genus} $g_0 \leq \mathfrak{g}(\Gamma) \leq g_1$,
            \item\label{item:short_loops_II_F_close} $\mathbf{F}(\Gamma, \Sigma) \leq \delta$,
            \item\label{item:short_loops_II_bdry} $\sum^n_{i = 1} \mathcal{H}^1(\partial D_i) < \delta$,
            \item\label{item:short_loops_II_diam} $\sum^n_{i = 1} \operatorname{diam}_M(\overline{D_i}) < \delta$,
        \end{enumerate}
        there exist $\mathfrak{g}(\Gamma) - g_0$ disjoint simple loops $\{c_j\}^{\mathfrak{g}(\Gamma) - g_0}_{j = 1}$ in $\Gamma \setminus \bigcup^n_{i = 1} \overline{D_i}$ such that: 
        \begin{enumerate}[label = \normalfont(\arabic*)]
            \item\label{item:short_loops_II_short} Each $c_j$ has length $\mathcal{H}^1(c_j) \leq \varepsilon$.
            \item\label{item:short_loops_II_indep} $\{[c_j]\}^{\mathfrak{g}(\Gamma) - g_0}_{j = 1}$ is linearly independent in $H_1(\Gamma; \Z_2)$.
        \end{enumerate} 

        In particular, $\Gamma \setminus \bigcup^g_{j = 1} c_j$ is a genus $g_0$ surface with $2(\mathfrak{g}(\Gamma) - g_0)$ boundary circles. 
    \end{prop} 
    \begin{rmk}\label{rmk:drop_genus_lb_cond} 
        By Lemma~\ref{lem:genusLarger}, for sufficiently small $\delta$ depending on $M, \bg$ and $\Sigma$, if $\mathbf{F}(\Gamma, \Sigma) < \delta$, then $\mathfrak{g}({\Gamma}) \geq \mathfrak{g}({\Sigma})$. Therefore, we can replace the condition ``$g_0 \leq \mathfrak{g}(\Gamma) \leq g_1$'' by ``$\mathfrak{g}(\Gamma) \leq g_1$'' in the proposition.
    \end{rmk} 

    We first prove a useful lemma which enables us to reduce to the case where $\Gamma$ is in a small neighborhood of $\Sigma$. 

    \begin{lem}\label{lem:restr_to_nbhd} 
        In an orientable closed Riemannian $3$-manifold $(M, \bg)$, let $\Sigma$ be an embedded smooth orientable closed surface. For any $\varepsilon \in (0, 1)$, there exists $\delta = \delta(M, \bg, \Sigma, \varepsilon) \in (0, \varepsilon / 2)$ with the following property. 
        In $(M, \bg)$, for any embedded smooth orientable closed surface $\Gamma$ containing a finite set of disjoint closed disks $\{\overline{D_i}\}^n_{i = 1}$ with piecewise smooth boundary satisfying:
        \begin{enumerate}[label = \normalfont(\roman*)]
            \item\label{item:restr_to_nbhd_F_close} $\mathbf{F}(\Gamma, \Sigma) \leq \delta$,
            \item\label{item:restr_to_nbhd_diam} $\sum^n_{i = 1} \operatorname{diam}_M(\overline{D_i}) < \delta$,
        \end{enumerate}
        there exists $\eta \in (0, \varepsilon)$ such that:
        \begin{enumerate}[label = \normalfont(\arabic*)]
            \item\label{item:restr_to_nbhd_normal_exp} $\exp^\perp: \Sigma \times [-\eta, \eta] \to M$ is a diffeomorphism to its image $\exp^\perp(\Sigma\times [-\eta, \eta])$; We denote $N_\eta(\Sigma) := \exp^\perp(\Sigma\times (-\eta, \eta))$.
            \item\label{item:restr_to_nbhd_inter_trans} $\partial N_\eta(\Sigma)$ intersects $\Gamma$ transversally and $\partial N_\eta(\Sigma) \cap \bigcup^n_{i = 1} \overline{D_i} = \emptyset$.
            \item\label{item:restr_to_nbhd_inter_small} $\mathcal{H}^1(\partial N_\eta(\Sigma) \cap \Gamma) < \varepsilon$.
            \item\label{item:restr_to_nbhd_outside_small} $\mathcal{H}^2(\Gamma \setminus N_\eta(\Sigma)) < \varepsilon$.
        \end{enumerate}
        Moreover, by taking $\delta$ even smaller, we can remove a small tubular neighborhood of $\partial N_\eta(\Sigma) \cap \Gamma$ in $\Gamma$, and attach disjoint embedded disks $\{D'_i\}^{n'}_{i = 1}$ in $M$ to obtain two disjoint embedded smooth orientable closed surfaces $\Gamma_1$ and $\Gamma_2$ satisfying:
        \begin{enumerate}[label = \normalfont(\arabic*)]
            \setcounter{enumi}{4}
            \item\label{item:restr_to_nbhd_nbhd} $\Gamma_1 \subset N_\eta(\Sigma)$.
            \item\label{item:restr_to_nbhd_Gamma_1_F_close} $\bF(\Gamma_1, \Sigma) < \varepsilon$.
            \item\label{item:restr_to_nbhd_Gamma_2_area} $\mathcal{H}^2(\Gamma_2) < \varepsilon$.
        \end{enumerate}
        And $\{\overline{D_i}\}^{n_\Gamma}_{i = 1} := \{\overline{D_i}\}^n_{i = 1} \cup \{\overline{D'_i}\}^{n'}_{i = 1}$ is also a finite set of disjoint closed disks with piecewise smooth boundary, which satisfies:
        \begin{enumerate}[label = \normalfont(\arabic*)]
            \setcounter{enumi}{7}
            \item\label{item:restr_to_nbhd_new_bdry} $\sum^{n'}_{i = 1} \mathcal{H}^1(\partial D'_i) < \varepsilon / 2$.
            \item\label{item:restr_to_nbhd_new_diam} $\sum^{n'}_{i = 1} \operatorname{diam}_M(\overline{D'_i}) < \varepsilon / 2$.
        \end{enumerate} 
    \end{lem} 
    \begin{proof} 
        Firstly, we choose $\eta_0 = \eta_0(M, \bg, \Sigma) \in (0, 1)$ such that 
$\exp^\perp: \Sigma \times [-\eta_0, \eta_0] \to M$
        is a diffeomorphism to its image. We fix $\zeta \in (0, \min(\varepsilon, \eta_0))$ which will be determined later. 

        By the definition of $\bF$-metric, we can choose $\delta \in (0, \zeta/10)$ such that for any $\Gamma$ and $\{\overline{D}_j\}^n_{j = 1}$ satisfying~\ref{item:restr_to_nbhd_F_close} and~\ref{item:restr_to_nbhd_diam},
      $\mathcal{H}^2(\Gamma \setminus N_{\zeta / 2}(\Sigma)) < \zeta^2/10.$
        By the coarea formula and Sard's theorem, there exists a measurable set $E \in (\zeta / 2, \zeta)$ with $\mathcal{H}^1(E) > \zeta / 5$ such that for any $r \in E$, $\Gamma$ intersects $\partial N_{r}(\Sigma)$ transversally and
       $\mathcal{H}^1 (\Gamma \cap \partial N_{r}(\Sigma)) < \zeta < \varepsilon\,.$
        By~\ref{item:restr_to_nbhd_diam}, we can choose $\eta \in E$ such that $\partial N_\eta(\Sigma) \cap \bigcup^n_{i = 1} \overline{D_i} = \emptyset$. These confirm~\ref{item:restr_to_nbhd_normal_exp},~\ref{item:restr_to_nbhd_inter_trans},~\ref{item:restr_to_nbhd_inter_small} and~\ref{item:restr_to_nbhd_outside_small}. 
        
        By the Jordan-Schoenflies theorem and the isoperimetric inequality, Lemma~\ref{lem:isop_ineq}, there exist positive constants $\nu_\Sigma$ and $\delta_\Sigma$, depending on $M$, $\bg$ and $\Sigma$, such that for any $r \in (-\eta_0, \eta_0)$ and any simple loop $c$ in $\partial N_r(\Sigma)$ with $\mathcal{H}^1(c) \leq \delta_\Sigma$, the loop $c$ bounds a disk $D \subset \partial N_r(\Sigma)$ with 
        \[
            \mathcal{H}^2(D)\leq \nu_\Sigma (\mathcal{H}^1(c))^2 \leq \nu_\Sigma \delta_\Sigma \mathcal{H}^1(c)\,, \quad \operatorname{diam}_M(\overline{D}) < \nu_\Sigma \mathcal{H}^1(c)\,.
        \] 
        
        Therefore, if we require that
           $ \zeta < \min\left(\frac{\varepsilon}{2 + 4\nu_\Sigma}, \delta_\Sigma\right)\,,$
        then in $(M, \mathbf{g})$, choosing suitable $\delta$ and $\eta$ as before, we can remove a small tubular neighborhood of $\Gamma \cap \partial N_{\eta}(\Sigma) = \bigcup^{n'}_{i = 1} c'_i$ and attach disks $\{D'_i\}^{2n'}_{i = 1}$ near $\partial N_\eta(\Sigma)$ to obtain two disjoint embedded closed surfaces $\Gamma_1 \subset N_\eta(\Sigma)$ and $\Gamma_2$ satisfying:
        $$\mathcal{H}^2(\Gamma_2) < \mathcal{H}^2(\Gamma \setminus N_\eta(\Sigma)) + \sum^{n'}_{i = 1} \mathcal{H}^2(D'_i)< \mathcal{H}^2(\Gamma \setminus N_\eta(\Sigma)) + 2 \nu_\Sigma \delta_\Sigma \sum^{n'}_{i = 1} \mathcal{H}^1(\partial D'_i)\leq (1 + 2\nu_\Sigma \delta_\Sigma)\zeta  \leq \varepsilon\,,$$
        $$
            \sum^{n'}_{i = 1} \mathcal{H}^1(\partial D'_i) \leq \zeta < \varepsilon / 2\,,\quad\textrm{ and }\quad 
       \sum^{n'}_{i = 1} \operatorname{diam}_M(\overline{D'_i}) \leq \nu_\Sigma \sum^{n'}_{i = 1} \mathcal{H}^1(\partial D'_i) + \delta < \nu_\Sigma \zeta + \zeta < \varepsilon / 2\,.$$
        These confirm~\ref{item:restr_to_nbhd_nbhd},~\ref{item:restr_to_nbhd_Gamma_2_area}, ~\ref{item:restr_to_nbhd_new_bdry} and~\ref{item:restr_to_nbhd_new_diam}. 

        Finally, we require that 
            $\zeta < \frac{\delta(M, \bg)}{1 + 2\nu_\Sigma \delta_\Sigma}\,,$
        where $\delta(M, \bg)$ is a constant from Lemma~\ref{lem:isop_ineq}. By the isoperimetric inequality again, $\Gamma_2$ bounds a $3$-dimensional modulo $2$ integral current $A$ such that 
        \[
            \mathcal{H}^3(A) \leq \nu_M (\mathcal{H}^2(\Gamma_2))^{3/2} \leq \nu_M (1 + 2\nu_\Sigma \delta_\Sigma)^{3/2} \zeta^{3/2} \,,
        \]
        where $\nu_M$ is also a constant from Lemma~\ref{lem:isop_ineq}.
        Hence, 
        \[
            \mathcal{F}([\Gamma], [\Gamma_1]) \leq \nu_M (1 + 2\nu_\Sigma \delta_\Sigma)^{3/2} \zeta^{3/2}\,,
        \] and
        \[
            |\mathcal{H}^2(\Gamma) - \mathcal{H}^2(\Gamma_1)| \leq 2 \mathcal{H}^2(\Gamma_2) < 2(1 + 2\nu_\Sigma \delta_\Sigma)\zeta\,.
        \]
        It follows from Corollary~\ref{cor:Pitts_bF_flat_quant} that, if we choose $\zeta$ small enough, choosing suitable $\delta$ and $\eta$ as before, we have 
            $\mathbf{F}(\Gamma, \Gamma_1) < \varepsilon / 2$ and $ \mathbf{F}(\Gamma, \Sigma) \leq \delta < \varepsilon / 2\,.$
        This confirms~\ref{item:restr_to_nbhd_Gamma_1_F_close} of the proposition. 
        Note that the choice of $\zeta$ only depends on $M$, $\bg$, $\Sigma$ and $\varepsilon$ and so does $\delta$. 
    \end{proof} 

    This lemma can be extended to punctate surfaces, which will be used in Appendix~\ref{sect:genusIneq}.
    
    \begin{cor}\label{cor:restr_PS_to_nbhd} 
        In an orientable closed Riemannian $3$-manifold $(M, \bg)$, let $\Sigma$ be an embedded smooth orientable closed surface. For any $\varepsilon \in (0, 1)$, there exists $\delta = \delta(M, \bg, \Sigma, \varepsilon) \in (0, \varepsilon)$ with the following property. 
        In $(M, \bg)$, for any punctate surface $S$ with $\mathbf{F}(S, \Sigma) \leq \delta$,
        there exists $\eta \in (0, \varepsilon)$ such that:
        \begin{enumerate}[label = \normalfont(\arabic*)]
            \item\label{item:restr_PS_to_nbhd_normal_exp} $\exp^\perp: \Sigma \times [-\eta, \eta] \to M$ is a diffeomorphism to its image $\exp^\perp(\Sigma\times [-\eta, \eta])$. We denote $N_\eta(\Sigma) := \exp^\perp(\Sigma\times (-\eta, \eta))$.
            \item\label{item:restr_PS_to_nbhd_inter_trans}$\partial N_\eta(\Sigma)$ intersects $S$ transversally.
            \item\label{item:restr_PS_to_nbhd_inter_small} $\mathcal{H}^1(\partial N_\eta(\Sigma) \cap S) < \varepsilon$.
            \item\label{item:restr_PS_to_nbhd_outside_small} $\mathcal{H}^2(S \setminus N_\eta(\Sigma)) < \varepsilon$.
        \end{enumerate}
        Moreover, by taking $\eta$ even smaller, we can remove a small tubular neighborhood of $\partial N_\eta(\Sigma) \cap S$ in $S$, and attach disjoint embedded disks $\{D'_i\}^{n'}_{i = 1}$ in $M$ to obtain two disjoint embedded punctate orientable closed surfaces $S_1$ and $S_2$ satisfying:
        \begin{enumerate}[label = \normalfont(\arabic*)]
            \setcounter{enumi}{4}
            \item\label{item:restr_PS_to_nbhd_nbhd} $S_1 \subset N_\eta(\Sigma)$.
            \item\label{item:restr_PS_to_nbhd_Gamma_1_F_close} $\bF(S_1, \Sigma) < \varepsilon$.
            \item\label{item:restr_PS_to_nbhd_Gamma_2_area} $\mathcal{H}^2(S_2) < \varepsilon$.
        \end{enumerate}
    \end{cor} 
    \begin{proof}
        For a punctate surface $S$, a punctate set $P$ for $S$ can be viewed as a finite set of degenerating disks. Hence, one can follow the same proof of Lemma~\ref{lem:restr_to_nbhd} to conclude the corollary.
    \end{proof}

    By Lemma \ref{lem:restr_to_nbhd}, the closed surface $\Gamma$ can be decomposed as two components $\Gamma_1$ and $\Gamma_2$. Since $\Gamma_2$ has small area, Proposition~\ref{prop:short_loops_I} can be applied to exhaust all short simple loops it contains. For $\Gamma_1$, which lies in a $\eta$-normal neighborhood of $\Sigma$, we choose a triangulation of $\Sigma$ and apply Proposition~\ref{prop:short_loops_I} again to $\Gamma_1$ in the $\eta$-normal neighborhood of each triangle. With small enough $\eta$, removing all those short simple loops in $\Gamma_1$ as described above yields a genus $g_0$ surface with boundary. The details are presented below.   
    
    \begin{proof}[Proof of Proposition~\ref{prop:short_loops_II}] 
        If $g_1 = g_0$, the proposition holds trivially. In the following, we fix $g_1 > g_0 \in \mathbb{N}$, an embedded orientable genus $g_0$ closed surface $\Sigma \subset M$ and $\varepsilon > 0$ as in the proposition. 

        \medskip
        \paragraph*{\textbf{Step 0. Set up}} 

        Let $\cT$ be a triangulation of $\Sigma$, with $\cT^{(0)}$ the set of all $0$-cells $v$, $\cT^{(1)}$ the set of all $1$-cells $e$, and $\cT^{(2)}$ the set of all $2$-cells $\Delta$. Let $\tau \in (0, \frac{1}{100000\# \cT})$.
        
        For each $1$-cell $e$, we extend it to a closed smooth loop on $\Sigma$ denoted by $\tilde e$. Then we can define a smooth \textit{signed} distance $\rho_{\tilde e}: U_{\tilde e} \to \R$ to $\tilde{e}$ on some small neighborhood $U_{\tilde e} \subset \Sigma$ of $\tilde e$. Since $\tilde e$ is contractible on $\Sigma$, we can extend $\rho_{\tilde e}$ smoothly to the whole $\Sigma$ such that $\inf_{\Sigma \setminus U_{\tilde e}} |\rho_{\tilde e}| > 0$; we still denote this \textit{modified} signed distance function by $\rho_{\tilde e}$. We also choose $U'_{\tilde e} \Subset U_{\tilde e}$ to be a smaller neighborhood.

        Consider the tangent bundle $T\Sigma$  of $\Sigma$. We claim that the restriction $T\Sigma|_{\cT^{(1)}}$ onto the 1-skeleton $\cT^{(1)}$ is a trivial vector bundle. Indeed, by dimension counting, the zero section of this bundle can be perturbed to some everywhere non-zero smooth section $V_1$. Since $\Sigma$ is orientable, we can choose another smooth non-vanishing vector field $V_2$ such that $\{V_1,V_2\}$ is positively oriented, showing that $T\Sigma|_{\cT^{(1)}}$ is a trivial bundle. We can smoothly extend $V_1, V_2$ onto an open neighborhood of $\mathcal{T}^{(1)}$, while maintaining their linear independence, and then smoothly extend $V_1, V_2$ to all of $\Sigma$ without requiring non-vanishingness.
        
        We choose $\xi > 0$ small enough such that $V_1$ and $V_2$ generate a $2$-parameter family of diffeomorphisms $\Phi_y: \Sigma \to \Sigma$, parameterized by $y = (y_1, y_2) \in \mathbb{D}^2_{2\xi}$ via $\Phi_y(x) = \exp_x(y_1 V_1(x) + y_2 V_2(x))$. Moreover, we can ensure:
        \begin{enumerate}
            \item The  map $\mathbb{D}^2_{2\xi}\x\Sigma\to\Sigma$ defined by  $(y, x) \mapsto \Phi_y(x)$  is a submersion at  $(0,x)$ for any $x\in \mathcal{T}^{(1)}$, giving us a 2-dimensional family of ``translations" along the 1-skeleton.
            \item For all $e \in \cT^{(1)}$ and $y \in \mathbb{D}^2_{2\xi}$, $\Phi^{-1}_y(e) \subset U'_{\tilde e}$.
        \end{enumerate}

        With these two vector fields, we can define a family of perturbed triangulations $\{\mathcal{T}_y := \Phi_y(\mathcal{T})\}_{y \in \mathbb{D}^2_{2\xi}}$. We denote the perturbed cells by $v_y := \Phi_y(v)$, $e_y := \Phi_y(e)$, and $\Delta_y := \Phi_y(\Delta)$. By choosing $\xi$ even smaller if necessary, the derivatives of $\Phi_y$ are uniformly bounded with respect to $y$, and the lengths and areas of these cells change by a factor within $(1 - \tau^2/100, 1 + \tau^2/100)$ compared to $\mathcal{T}$. 

        Let $\zeta>0$ be sufficiently small such that the normal exponential map for $\Sigma$ is a $(1+\tau^2/100)$-bilipschitz diffeomorphism onto a tubular neighborhood $N_\zeta(\Sigma)$.  Denote the closest-point projection by $\pi: N_\zeta(\Sigma) \to \Sigma$. For any  cell, we can consider its vertical preimage under $\pi$: specifically, the vertical line segments $\mathrm{I}_{v, y} := \pi^{-1}(v_y)$ over vertices, the rectangular strips $\mathrm{R}_{e, y} := \pi^{-1}(e_y)$ over edges, and the solid prisms $\mathrm{Pri}_{\Delta, y} := \pi^{-1}(\Delta_y)$ over faces, whose boundaries are the cylinders $\mathrm{Cyl}_{\Delta, y} := \pi^{-1}(\partial \Delta_y)$. 

        \medskip
        \paragraph*{\bf Step 1. Existence of a good generic triangulation} 

        We show that there exists $\delta_1 > 0$ with the following property. For any $\Gamma \subset N_\zeta(\Sigma)$ with $\mathbf{F}(\Gamma, \Sigma) < \delta_1$ and $\sum_{i=1}^n \operatorname{diam}_M(\overline{D_i}) < \delta_1$, there exists $y \in \mathbb{D}^2_\xi$ for which the preimage of $\mathcal T_y$  intersects $\Gamma \setminus \bigcup_{i=1}^n \overline{D_i}$ in a ``good" way, namely: 
        \begin{enumerate}[label=\normalfont(\arabic*)]
            \item\label{item:v_transverse} For any $v \in \mathcal{T}^{(0)}$, $\mathrm{I}_{v, y} \pitchfork \Gamma$ consists of exactly one point.
            \item\label{item:Delta_transverse} For any $\Delta \in \mathcal{T}^{(2)}$, $\mathrm{Cyl}_{\Delta, y} \pitchfork \Gamma$ is a finite union of disjoint simple loops $\bigcup^{n_\Delta}_{k=1} c_{\Delta, k}$. 
            \item\label{item:long_loop} Exactly one loop $c_{\Delta, 1}$ (the \emph{long loop}) is homologically nontrivial in $\mathrm{Cyl}_{\Delta, y}$  and satisfies $\frac{\mathcal{H}^1(c_{\Delta, 1})}{\mathcal{H}^1(\partial \Delta_y)} \in (1-\tau, 1+\tau)$.
            \item\label{item:short_loop} All other loops $c_{\Delta, k}$ for $k \geq 2$ (the \emph{short loops}) are trivial in $\mathrm{Cyl}_{\Delta, y}$, do not intersect the vertex lines $\mathrm{I}_{v, y}$, and their total length is bounded by $\tau \mathcal{H}^1(\partial \Delta_y)$.
        \end{enumerate}

        Taking $\delta_1$ sufficiently small guarantees that any $\Gamma$ satisfying $\mathbf{F}(\Gamma, \Sigma) \leq \delta_1$ must be homologous to $\Sigma$ in $N_\zeta(\Sigma)$ over $\Z_2$. Consequently, transverse intersections of $\Gamma$ with the regions $\mathrm{I}_{v, y}$ and $\mathrm{R}_{e, y}$ implies $[I_{v, y} \cap \Gamma] = [I_{v, y} \cap \Sigma] \neq 0 \in H_0(I_{v, y}, \partial I_{v, y}; \Z_2)$ and
        $[\mathrm{R}_{e, y} \cap \Gamma] = [\mathrm{R}_{e, y} \cap \Sigma] \neq 0 \in H_1(\mathrm{R}_{e, y}, \partial_c R_{e, y}; \Z_2)$, where $\partial_c R_{e, y} = \pi^{-1}(\partial e_y)$.
        In particular, when intersecting transversely, we have
        \begin{equation} \label{equ:H^0(I cap Gamma)>0}
            \mathcal{H}^0(\mathrm{I}_{v, y} \cap \Gamma) \ge 1,
        \end{equation}
        and $\mathrm{R}_{e, y} \cap \Gamma$ must contain a path connecting its two vertical boundaries, implying that  
        \begin{equation} \label{equ:H^0(R cap Gamma)>length(e)}
            \cH^1(\mathrm{R}_{e, y} \cap \Gamma) \geq (1 - \tau^2/100)\mathcal{H}^1(e_y).
        \end{equation}
In the following, we use the coarea formula over the parameter space $\mathbb{D}^2_{\xi}$ to obtain more detailed control for the sizes of the intersections $\mathrm{I}_{v,y}\cap\Gamma$ and $\mathrm{R}_{e,y}\cap\Gamma$.

        \

        \noindent\textbf{Step 1.1. Control on $\mathrm{I}_{v,y}\cap\Gamma$.}
        For a given vertex $v \in \mathcal{T}^{(0)}$, the map $y \mapsto \Phi_y(v)$ is a diffeomorphism from $\mathbb{D}^2_{2\xi}$ onto a disk in $\Sigma$ with uniformly bounded derivatives. Let $B_v$  be the image of $\mathbb{D}^2_{\xi}$ under this map. Define the local submersion $\mathfrak{y}_v: \pi^{-1}(B_v) \to \mathbb{D}^2_{\xi}$ by $\Phi_{\mathfrak{y}_v(p)}(v) = \pi(p)$, which is clearly Lipschitz. Applying the coarea formula to $\mathfrak{y}_v$ over $\Gamma$ yields:
        \[
            \int_{\mathbb{D}^2_{\xi}} \mathcal{H}^0(\mathrm{I}_{v,y} \cap \Gamma) \, \mathrm{d}\cH^2(y) = \int_{\Gamma \cap \pi^{-1}(B_v)} J^\Gamma_{\mathfrak{y}_v} \, \mathrm{d}\mathcal{H}^2 \le \int_{\Gamma \cap \pi^{-1}(B_v)} J^M_{\mathfrak{y}_v} \, \mathrm{d}\mathcal{H}^2.
        \]
        Here, we set
        \[
            J^M_{\mathfrak{y}_v} := \sqrt{\det[(\nabla^M \mathfrak{y}^i_v \cdot \nabla^M \mathfrak{y}^j_v)]}\,, \quad J^\Gamma_{\mathfrak{y}_v} := \sqrt{\det[(\nabla^\Gamma \mathfrak{y}^i_v \cdot \nabla^\Gamma \mathfrak{y}^j_v)]}, \quad J^\Sigma_{\mathfrak{y}_v} := \sqrt{\det[(\nabla^\Sigma \mathfrak{y}^i_v \cdot \nabla^\Sigma \mathfrak{y}^j_v)]}\,.
        \]
        By Lemma~\ref{lem:boldF_restriction}, because $\mathbf{F}(\Gamma, \Sigma) < \delta_1$, for sufficiently small $\delta_1$, we have
        \[
            \int_{\Gamma \cap \pi^{-1}(B_v)} J^M_{\mathfrak{y}_v} \, \mathrm{d}\mathcal{H}^2 \leq \left(1 + \frac{\tau^2}{10}\right)\int_{\Sigma \cap \pi^{-1}(B_v)} J^M_{\mathfrak{y}_v} \, \mathrm{d}\mathcal{H}^2\,.
        \]
        Since $\pi|_\Sigma = \operatorname{id}|_\Sigma$, $J^M_{\mathfrak{y}_v} = J^\Sigma_{\mathfrak{y}_v}$, so the right-hand side is
        \[
            \left(1 + \frac{\tau^2}{10}\right)\int_{\Sigma \cap \pi^{-1}(B_v)} J^\Sigma_{\mathfrak{y}_v} \, \mathrm{d}\mathcal{H}^2 = \left(1 + \frac{\tau^2}{10}\right) \mathcal{H}^2(\mathbb{D}^2_{\xi})\,.
        \]
        Let $A_v = \{ y \in \mathbb{D}^2_{\xi} \mid \mathcal{H}^0(\mathrm{I}_{v, y} \cap \Gamma) = 1 \}$. Because by \eqref{equ:H^0(I cap Gamma)>0}, $\mathcal{H}^0(\mathrm{I}_{v,y} \cap \Gamma)$ is an integer $\ge 1$ whenever the intersection is transverse, by Markov's inequality\footnote{For a nonnegative function $f$ defined on a measure space $(X, \mu)$ and any positive number $M$, $\int_X f d\mu \geq \frac{1}{M} \mu\{f \geq M\}$.}, 
        \[
            \mathcal{H}^2(\mathbb{D}^2_{\xi} \setminus A_v) \leq \frac{\tau^2}{10}\mathcal{H}^2(\mathbb{D}^2_{\xi})\,, \text{ and thus } \mathcal{H}^2(A_v) \ge (1 - \frac{\tau^2}{10})\mathcal{H}^2(\mathbb{D}^2_{\xi})\,.
        \]

        \

        \noindent\textbf{Step 1.2. Control on  $\mathrm{R}_{e,y}\cap \Gamma$.}
Recall that for each $e \in \cT^{(1)}$, we extend it to a  smooth loop $\tilde e\subset \Sigma$ and choose a \textit{modified} signed distance function $\rho_{\tilde e}$ as in the beginning of Step 0. For the small tubular neighborhood $N_\zeta(\Sigma)$ of $\Sigma$, we define the smooth function $P_{\tilde e}: N_\zeta(\Sigma) \times \mathbb{D}^2_{2\xi} \to \mathbb{R}$ by
        \[
            P_{\tilde e}(p, y): = \rho_{\tilde e}(\Phi_y^{-1}(\pi(p)))\,.
        \]
        
        For a fixed $y$, $\mathrm{R}_{e, y}$ is exactly the portion of the zero level set of $P_{\tilde e, y}(\cdot) := P_{\tilde e}(\cdot, y)$ where $\Phi_y^{-1}(\pi(\cdot)) \in e$. Let $\mathbf{1}_e$ be the indicator function of the segment $e$ defined on ${\tilde e}$ and let $\pi_{\tilde e}$ be the closest point projection to $\tilde e$, which is well-defined in $U_{\tilde e}$.
        
        Fix a small neighborhood $U_{e}$ of $e$ such that for all $p \in \pi^{-1}(U_{e})$, the map $y \in \mathbb{D}^2_{2\xi} \mapsto \Phi_y^{-1}(\pi(p)) \in U_{\tilde e}$ is a local diffeomorphism; Indeed, this follows from the fact that the vector fields $V_1, V_2$ are linearly independent near $e$. Because $\nabla \rho_{\tilde e} \neq 0$ in $U_{\tilde e}$, for all $p \in U_e$, the map $y \mapsto P^p_{{\tilde e}}(y) := P_{\tilde e}(p, y)$ is a submersion, and its gradient $\nabla^{\mathbb{D}^2} P^p_{{\tilde e}}(y)$ is non-zero everywhere. By the compactness, $\inf_{y \in \mathbb{D}^2_{3\xi/2}} |\nabla^{\mathbb{D}^2} P^p_{{\tilde e}}(y)| > 0$.

        Using 
        a sequence of smooth cut-off functions $\{\varphi_{e, \sigma}\}_{\sigma}$ defined on $\tilde e$ approximating $\mathbf{1}_e$ with $\varphi_{e,\sigma} \geq \mathbf{1}_e$, we can express the intersection length using the coarea formula on $\Gamma$, for all $y \in \mathbb{D}^2_{2\xi}$ with $R_{e,y} \pitchfork \Gamma$:
        \begin{align*}
            \mathcal{H}^1(\mathrm{R}_{e, y} \cap \Gamma) &= \int_{\{p \in \Gamma \cap \pi^{-1}(U_e) \mid P_{\tilde e}(p, y) = 0\}} \mathbf{1}_e(\pi_{\tilde e} \circ \Phi_y^{-1}(\pi(p)))\, \mathrm{d}\cH^1(p)\\
            &= \lim_{\sigma \to 0}\int_{\{p \in \Gamma \cap \pi^{-1}(U_e) \mid P_{\tilde e}(p, y) = 0\}} \varphi_{e, \sigma}(\pi_{\tilde e} \circ \Phi_y^{-1}(\pi(p)))\, \mathrm{d}\cH^1(p)
        \end{align*}

        Let $\psi_{\mathbb{D}^2_{\xi}}$ be a cut-off functions compactly supported in $\mathbb{D}^2_{3\xi/2}$ with    $\psi_{\mathbb{D}^2_\xi} \geq \mathbf{1}_{\mathbb{D}^2_\xi}$ to be determined later. Applying Fubini's theorem, we compute:
        \begin{align*}
            & \quad \int_{\mathbb{D}^2_{2\xi}} \psi_{\mathbb{D}^2_{\xi}}(y) \mathcal{H}^1(\mathrm{R}_{e, y} \cap \Gamma) \, \mathrm{d}\cH^2(y)\\
            &= \int_{\mathbb{D}^2_{2\xi}} \psi_{\mathbb{D}^2_{\xi}}(y) \lim_{\sigma \to 0} \left( \int_{\{p \in \Gamma \cap \pi^{-1}(U_e) \mid P_{\tilde e}(p, y)=0\}}  \varphi_{e, \sigma}(\pi_{\tilde e} \circ \Phi_y^{-1}(\pi(p)))\, \mathrm{d}\mathcal{H}^1(p)\right) \, \mathrm{d}\mathcal{H}^2(y)\\
            &= \lim_{\sigma \to 0} \int_{\mathbb{D}^2_{2\xi}} \psi_{\mathbb{D}^2_{\xi}}(y) \left( \int_{\{p \in \Gamma \cap \pi^{-1}(U_e) \mid P_{\tilde e}(p, y)=0\}}  \varphi_{e, \sigma}(\pi_{\tilde e} \circ \Phi_y^{-1}(\pi(p)))\, \mathrm{d}\mathcal{H}^1(p)\right) \, \mathrm{d}\mathcal{H}^2(y)\\ 
            &= \lim_{\sigma \to 0} \int_{\{(p, y) \in (\Gamma \cap \pi^{-1}(U_e)) \times \mathbb{D}^2_{2\xi} \mid P_{\tilde e}(p, y)=0\}} \psi_{\mathbb{D}^2_{\xi}}(y) \varphi_{e, \sigma}(\pi_{\tilde e} \circ \Phi_y^{-1}(\pi(p)))\frac{|\nabla^\Gamma P_{\tilde e, y}(p)|}{\sqrt{|\nabla^\Gamma P_{\tilde e, y}(p)|^2 + |\nabla^{\mathbb{D}^2} P^p_{{\tilde e}}(y)|^2}} \, \mathrm{d}\mathcal{H}^3(p, y) \\ 
            &= \lim_{\sigma \to 0} \int_{\Gamma \cap \pi^{-1}(U_e)} \underbrace{\left( \int_{\{y \in \mathbb{D}^2_{2\xi} \mid  P_{\tilde e}(p, y)=0\}} \psi_{\mathbb{D}^2_{\xi}}(y) \varphi_{e, \sigma}(\pi_{\tilde e} \circ \Phi_y^{-1}(\pi(p))) \frac{|\nabla^\Gamma P_{{\tilde e},y}(p)|}{|\nabla^{\mathbb{D}^2} P^p_{{\tilde e}}(y)|} \, \mathrm{d}\mathcal{H}^1(y) \right)}_{=: w_{e, \sigma}^\Gamma(p)} \mathrm{d}\mathcal{H}^2(p)\\ 
            &= \int_{\Gamma \cap \pi^{-1}(U_e)} \underbrace{\left(\int_{\{y \in \mathbb{D}^2_{2\xi} \mid P_{\tilde e}(p, y) = 0\}} \psi_{\mathbb{D}^2_{\xi}}(y) \mathbf{1}_e(\pi_{\tilde e} \circ \Phi_y^{-1}(\pi(p))) \frac{|\nabla^\Gamma P_{{\tilde e},y}(p)|}{|\nabla^{\mathbb{D}^2} P^p_{{\tilde e}}(y)|} \, \mathrm{d}\mathcal{H}^1(y) \right)}_{=: w_e^\Gamma(p)} \mathrm{d}\mathcal{H}^2(p)\,. 
        \end{align*}
        For sufficiently small $\sigma$, $w^\Gamma_{e, \sigma}(p)$, and thus $w^\Gamma_{e}$, are compactly supported in $\Gamma \cap \pi^{-1}(U_e)$, so we can extend the definitions of $w^\Gamma_e$ and $w^\Gamma_{e \sigma}$ onto the whole $\Gamma$ by defining them to be zero everywhere else. In particular, the last two lines above can be rewritten as:
        \begin{align*}
            \lim_{\sigma \to 0} \int_{\Gamma} w_{e, \sigma}^\Gamma(p)\mathrm{d}\mathcal{H}^2(p) = \int_{\Gamma} w_e^\Gamma(p)\mathrm{d}\mathcal{H}^2(p)\,.
        \end{align*}
        Similarly, we define $w^M_e$, $w^M_{e, \sigma}$, $w^\Sigma_e$, and $w^\Sigma_{e, \sigma}$ by replacing $\Gamma$ with $M$ and $\Sigma$, respectively.

        Since $\nabla^{\mathbb{D}^2} P^p_{{\tilde e}}(y)$ is everywhere non-zero, by the implicit function theorem, the set $\{y \in \mathbb{D}^2_{2\xi} \mid  P^p_{\tilde e}(y)=0\}$ is a smooth curve, and varies locally smoothly with respect to $p$. In particular, for all small $\delta$, the function $w^M_{e, \sigma}(p)$ is $C^1$ in $N_\zeta(\Sigma)$, where the projection $\pi$ is well-defined.

        Now, we can choose $\psi_{\mathbb{D}^2_{\xi}}$ to be sufficiently close to $\mathbf{1}_e$  such that 
        \[
            \int_{\mathbb{D}^2_{2\xi}} \psi_{\mathbb{D}^2_\xi}(y) \mathcal{H}^1(\mathrm{R}_{e, y} \cap \Sigma) \, \mathrm{d}\cH^2(y) \leq \left(1 + \frac{\tau^2}{100} \right)\int_{\mathbb{D}^2_{\xi}}  \mathcal{H}^1(\mathrm{R}_{e, y} \cap \Sigma) \, \mathrm{d}\cH^2(y)\,.
        \]
        Indeed, this follows from the fact that $R_{e, y} \cap \Sigma = e_y$ and its length varies in $y$ with a controlled factor as described in Step 0.

        In addition, for every $p \in \Sigma$ and $y \in \mathbb{D}^2_{2\xi}$, $|\nabla^M P_{{\tilde e},y}(p)| = |\nabla^\Sigma P_{{\tilde e},y}(p)|$, and thus
        \[
            \lim_{\sigma \to 0} \int_\Sigma w^M_{e, \sigma}(p) \mathrm{d} \cH^2(p) = \int_{\mathbb{D}^2_{2\xi}} \psi_{\mathbb{D}^2_\xi}(y) \mathcal{H}^1(\mathrm{R}_{e, y} \cap \Sigma) \, \mathrm{d}\cH^2(y)\,.
        \]
        Therefore, we can choose $\sigma_0$ sufficiently small, which only depends on $\Sigma$, such that
        \[
            \int_\Sigma w^M_{e, \sigma_0}(p)\, \mathrm{d}\cH^2(p) \leq \left(1 + \frac{\tau^2}{100}\right)^2 \int_{\mathbb{D}^2_{\xi}} \mathcal{H}^1(R_{e,y} \cap \Sigma) \, \mathrm{d}\cH^2(y)\,.
        \]

        Using the   facts  $|\nabla^\Gamma P_{{\tilde e},y}(p)| \leq |\nabla^M P_{{\tilde e},y}(p)|$ on $\Gamma$,  $|\nabla^\Sigma P_{{\tilde e},y}(p)| = |\nabla^M P_{{\tilde e},y}(p)|$ on $\Sigma$,  $\mathbf{F}(\Gamma, \Sigma) < \delta_1$,  and  $\Gamma \subset N_\zeta(\Sigma)$, we can ensure (by choosing $\delta_1$ sufficiently small) that
        \begin{align*}
            \int_{\mathbb{D}^2_\xi} \mathcal{H}^1(R_{e, y} \cap \Gamma) d\cH^2(y) &\leq \int_{\Gamma} w^\Gamma_e(p) \, \mathrm{d}\mathcal{H}^2(p) \leq \int_{\Gamma} w^M_e(p) \, \mathrm{d}\mathcal{H}^2(p) \leq \int_{\Gamma} w^M_{e, \sigma_0}(p) \, \mathrm{d}\mathcal{H}^2(p) \\
            &\leq \left(1 + \frac{\tau^2}{100}\right) \int_{\Sigma} w^M_{e, \sigma_0}(p) \, \mathrm{d}\mathcal{H}^2(p)
            \leq \left(1 + \frac{\tau^2}{10}\right) \int_{\mathbb{D}^2_{\xi}} \mathcal{H}^1(R_{e,y} \cap \Sigma) \, \mathrm{d}\cH^2(y)\,.
        \end{align*}
        
        Let $B_e = \{ y \in \mathbb{D}^2_{\xi} \mid \mathcal{H}^1(\mathrm{R}_{e, y} \cap \Gamma) \le (1 + \frac{\tau}{2})\mathcal{H}^1(e_y) \}$. By \eqref{equ:H^0(R cap Gamma)>length(e)}, we have $\cH^1(\mathrm{R}_{e,y}\cap \Gamma)\geq (1 - \tau^2/100)\mathcal{H}^1(e_y)$ whenever they intersect transversally. Then by Markov's inequality we obtain 
        \[
            \mathcal{H}^2(B_e) \ge (1 - \tau)\mathcal{H}^2(\mathbb{D}^2_{\xi})\,.
        \]
        For a generic $y \in B_e \cap \bigcap_v A_v$, the length bound forces $\mathrm{R}_{e, y} \cap \Gamma$ to consist of exactly \emph{one} path connecting the two vertical boundary components of $\mathrm{R}_{e,y}$, plus disjoint closed loops of total length $\leq \tau \mathcal{H}^1(\partial \Delta_y)$.

        \
        
        \noindent \textbf{Step 1.3. Avoiding the disks.}
        Let $C$ be the subset of parameters $y \in \mathbb{D}^2_\xi$ where $(\bigcup_v \mathrm{I}_{v, y}) \cup (\bigcup_e \mathrm{R}_{e, y})$ intersects the initial disks $\bigcup_{i=1}^n \overline{D_i}$. The measure of $C$ can be controlled by the coarea formula and is bounded by a constant times $\sum_{i=1}^n \operatorname{diam}_M(\overline{D_i}) < \delta_1$. Choosing $\delta_1$ sufficiently small ensures $\mathcal{H}^2(C) < \tau \mathcal{H}^2(\mathbb{D}^2_\xi)$.

        Since $\tau \in (0, \frac{1}{100000\# \cT})$ (as chosen in Step 0), the intersection $\mathcal{Y}_{\text{good}} = \bigcap_{v} A_v \cap \bigcap_{e} B_e \setminus C$ has positive measure. Using Sard's theorem, we select a generic  parameter $y  \in \mathcal{Y}_{\text{good}}$ such that in each cylinder $\mathrm{Cyl}_{\Delta, y}$ the paths connecting the vertical edges form a homologically non-trivial long loop $c_{\Delta, 1}$; and  all other loops $c_{\Delta, k}$ ($k \ge 2$) are disjoint from the vertical edges,  homologically trivial within the cylinder, with total   length $\leq \tau \mathcal{H}^1(\partial \Delta_y)$. This confirms \ref{item:long_loop} and \ref{item:short_loop}.

        \medskip
        
        \paragraph*{\textbf{Step 2. Special case}} {\it We prove the proposition with the extra assumption that $\Gamma \subset N_\zeta(\Sigma)$.}

        Let $y \in \mathbb{D}^2_\xi$ be fixed  as given by Step 1. By Lemma~\ref{lem:boldF_restriction}, we can take $\delta_2 \in (0, \min(\delta_0/3, \delta_1))$ sufficiently small such that $\mathbf{F}(\Gamma, \Sigma) < \delta_2$ forces $\mathcal{H}^2(\mathrm{Pri}_{\Delta, y} \cap \Gamma) < \delta_0/3$ for all $\Delta \in \mathcal{T}^{(2)}$.  

        For any $\Delta \in \mathcal{T}^{(2)}$, the surface $\Gamma_\Delta := \Gamma \cap \mathrm{Pri}_{\Delta, y_1}$ has boundary consisting of the long loop $c_{\Delta, 1}$ and short loops $\{c_{\Delta, k}\}_{k=2}^{n_\Delta}$.
        We glue small disks $\{D_{\Delta, k}\}_{k=1}^{n_\Delta}$ to these boundary loops to form a closed surface $\Gamma'_\Delta$ (See Figure \ref{fig:Gamma_Delta}).
        Since the area of $\Gamma_\Delta$ and the length of its boundary loops are small, we can ensure $\mathcal{H}^2(\Gamma'_\Delta) < \delta_0$, and the total boundary length of the disks $D_{\Delta, k}$ and the initial disks $D_i$ inside $\Gamma_\Delta$ is strictly bounded by $\delta_0$.
        Applying Proposition~\ref{prop:short_loops_I} to $\Gamma'_\Delta$ yields $\mathfrak{g}(\Gamma'_\Delta)$ disjoint homologically independent simple loops $\{c'_{\Delta, k}\}$ in $\Gamma_\Delta \setminus \bigcup \overline{D_i}$, each of length $\leq \varepsilon / 2$.
        
        By performing surgeries on $\Gamma'_\Delta$ along these short loops, and removing the disk attached to the long loop $c_{\Delta, 1}$, we reduce $\Gamma_\Delta$ to a collection of spheres and exactly one topological disk bounded by $c_{\Delta, 1}$.  
        Then, gluing these topological disks together across all $\Delta \in \mathcal{T}^{(2)}$ along the long loops $c_{\Delta, 1}$ gives us back a surface homeomorphic to $\Sigma$, which has genus $g_0$: this  follows immediately from computing its Euler characteristic.
        In summary, we see that  performing surgeries to $\Gamma$ along the short loops $\bigcup_\Delta \{c_{\Delta, k}\}_{k \ge 2}$ and the  loops $\bigcup_\Delta \{c'_{\Delta, k}\}$ gives  a genus $g_0$ surface, and  we have obtained  $g = \mathfrak{g}(\Gamma) - g_0$ disjoint, homologically independent simple loops in $\Gamma \setminus \bigcup_{i=1}^n \overline{D_i}$, all of length $\leq \varepsilon$.

        \begin{figure}
            \centering
            \includegraphics[width=0.5\textwidth]{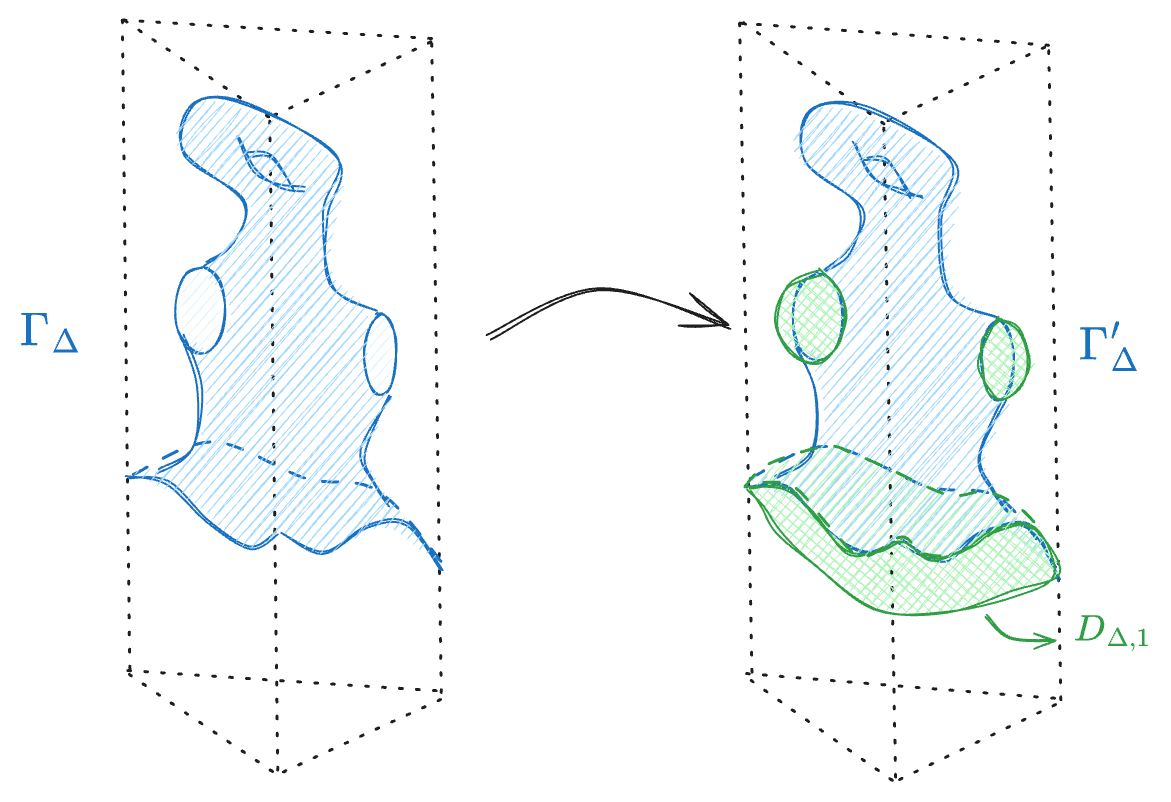}
            \caption{$\Gamma_\Delta$ and $\Gamma'_\Delta$}
            \label{fig:Gamma_Delta}
        \end{figure}

        \medskip
        \paragraph*{\bf Step 3. General case}

        Let $\delta = \delta(M, \bg, \Sigma, g_1, \varepsilon) := \delta(M, \bg, \Sigma, \min(\delta_2, \zeta))$ be chosen according to Lemma~\ref{lem:restr_to_nbhd}. For a general $\Gamma$, Lemma~\ref{lem:restr_to_nbhd} allows us to cut $\Gamma$ along $\partial N_\eta(\Sigma)$ (with boundary length $< \varepsilon$) and attach disks to decompose it into two disjoint orientable closed surfaces: $\Gamma_1 \subset N_\eta(\Sigma)$ with $\mathbf{F}(\Gamma_1, \Sigma) < \delta_2$, and $\Gamma_2$ with small area $\mathcal{H}^2(\Gamma_2) < \delta_2 < \delta_0$. 

        Applying Step 2 to $\Gamma_1$ and Proposition~\ref{prop:short_loops_I}  to $\Gamma_2$, we obtain $\mathfrak{g}(\Gamma_1) - g_0$ independent short loops in $\Gamma_1$ and $\mathfrak{g}(\Gamma_2)$ independent short loops in $\Gamma_2$. In addition, we collect also the loops obtained from  cutting along $\partial N_\eta(\Sigma)$. All these loops are disjoint, lie entirely in $\Gamma \setminus \bigcup_{i=1}^n \overline{D_i}$, and have lengths bounded by $\varepsilon$. Crucially, they induce $\mathfrak{g}(\Gamma) - g_0$ independent homology classes. This  concludes our proof.
    \end{proof}
    
\subsection{Small balls in punctate surfaces}\label{subsect:small_balls_punctate_surf} 

    In the previous two subsections, we've proved how to detect simple short loops in a smooth surface. However, neither Proposition~\ref{prop:short_loops_I} or Proposition~\ref{prop:short_loops_II} can be applied to every surface in our Simon--Smith family (Definition~\ref{def:Simon_Smith_family}), as not every surface is smooth therein. 

    In fact, for any $L \in \mathbb{R}^+$, one can construct a genus $1$ punctate surface $T$ of area $1$ but any nontrivial simple loop in $S$ has length at least $L$; see Appendix~\ref{sect:PS_long_loops} for the sketch of such a construction. In particular, Proposition~\ref{prop:short_loops_I} fails for a punctate surface of genus $1$. Fortunately, one can still find a nontrivial simple loop of $T$ that lies in a geodesic ball of uniformly bounded radius. 

    To generalize this statement to all punctate surfaces of any genus, we first prove the following useful lemma, which asserts that finitely many (not necessarily disjoint) small geodesic balls can be covered by relatively larger and pairwise disjoint geodesic balls in a controlled manner. 

    \begin{lem}\label{lem:disjoint_balls}
        In an orientable closed Riemannian $3$-manifold $(M, \mathbf{g})$, for any $q, n \in \mathbb{N}^+$, $\zeta \in \left(0, \frac{\operatorname{injrad}(M, \mathbf{g})}{2n (3n)^q}\right)$, and a finite set of geodesic balls $\{B(p_i, \zeta)\}^q_{i = 1}$ (not necessarily disjoint), there exist disjoint geodesic balls $\{B(p'_i, \zeta')\}^{q'}_{i = 1}$ such that:
        \begin{enumerate}[label=\normalfont(\arabic*)]
            \item $q' \leq q$, $\zeta' = (3n)^{q - q'}\zeta$ and $\{p'_i\}^{q'}_{i = 1} \subset\{p_i\}^{q}_{i = 1}$.
            \item $\displaystyle \bigcup^{q}_{i = 1} B(p_i, \zeta) \subset \bigcup^{q'}_{i = 1}B(p'_i, \zeta')$.
            \item The closed balls $\{\overline{B(p'_i, n\zeta')}\}^{q'}_{i = 1}$ are disjoint from each other.
        \end{enumerate}
    \end{lem} 
    \begin{proof}
        If the closed balls $\{\overline{B(p_i, n\zeta)}\}^q_{i = 1}$ are disjoint from each other, then the lemma holds trivially with $q' = q$, $p'_i = p_i$ and $\zeta' = \zeta$. 
        Otherwise, there exist two indices $k \neq l \in \{1, 2, \dots, q\}$ such that $\overline{B(p_k, n\zeta)} \cap \overline{B(p_l, n\zeta)} \neq \emptyset$. Then, we have $B(p_k, \zeta) \cup B(p_l, \zeta) \subset B(p_k, 3n \zeta)\,.$
        Set $\zeta_1 := 3n \zeta$ and $\{p_{1, i}\}^{q - 1}_{i = 1} := \{p_i\}^{q}_{i = 1} \setminus \{p_k\}$. Note that we have
        \[
            \bigcup^{q}_{i = 1} B(p_i, \zeta) \subset \bigcup^{q - 1}_{i = 1} B(p_{1, i}, \zeta_1)\,.
        \] 

        Inductively, we define $\zeta_{l + 1} = 3n \zeta_l$ provided that $\left\{\overline{B(p_{l,i}, n \zeta_l)}\right\}^{q - l}_{i = 1}$ are not pairwise disjoint. We also define $\{p_{l + 1, i}\}^{q - l - 1}_{i = 1}$ as a subset of $\{p_{l, i}\}^{q - l}_{i = 1}$ by removing the center of a ball that intersects another as above. Since the number of balls strictly decreases in the induction, this process stops at defining $\zeta_l$ for some $l \in \{1, \dots, q - 1\}$. Hence, we can take 
        $\zeta' = \zeta_l$, $q' = q - l$,  $\{p'_i\}^{q'}_{i = 1} = \{p_{l, i}\}^{q - l}_{i = 1}\,,$
        and complete the proof. 
    \end{proof} 

    Given two punctate surface $S_1, S_2 \in \GS(M)$ with punctate sets $P_1$ and $P_2$, we say that $S_1$ and $S_2$ \emph{intersect transversally} if the following holds: If $S_1 \cap P_2 = S_2 \cap P_1 = \emptyset$, then $S_1 \setminus P_1$ and $S_2 \setminus P_2$   intersect transversely
    We denote their intersection by $S_1 \pitchfork S_2$.
    We now show how to approximate a punctate surface by a smooth surface.

    \begin{lem}\label{lem:approx_PS}
        In an orientable closed Riemannian $3$-manifold $(M, \mathbf{g})$, suppose that $\Sigma$ is either an empty set or an embedded smooth orientable closed surface. For any $q\in \mathbb{N}^+$ and $\varepsilon \in \left(0, \frac{\operatorname{injrad}(M, \mathbf{g})}{2}\right)$, there exists a constant $\delta = \delta(M, \mathbf{g}, \Sigma, q, \varepsilon) \in (0, \varepsilon)$ with the following property. 
        For any punctate surface $S \in \mathcal{S}(M)$ with a punctate set $P$ such that  $\# P \leq q\,$ and $  \mathbf{F}(S, \Sigma) < \delta$,
        there exist an integer $\tilde q \leq q$, a real number $r \in (0, \varepsilon)$ and geodesic balls $\{B(p_i, r)\}^{\tilde q}_{i = 1}$ in $(M, \mathbf{g})$ such that:
        \begin{enumerate}[label=\normalfont{(\arabic*)}]
            \item\label{item:approx_PS_disjoint} The closed balls $\{\overline{B(p_i, r)}\}^{\tilde q}_{i = 1}$ are pairwise disjoint geodesic balls.
            \item\label{item:approx_PS_transverse} For each $i \in \{1, 2, \dots, \tilde q\}$, $\partial B(p_i, r) \pitchfork \Sigma$ is a union of finitely many loops.
            \item\label{item:approx_PS_genus} $S \setminus \bigcup^{\tilde q}_{i = 1} B(p_i, r)$ is an embedded smooth surface with boundary and genus at most $\mathfrak{g}(S)$.
            \item\label{item:approx_PS_approx} There exists an embedded smooth orientable closed surface $\Gamma$ such that
                \begin{enumerate}[label=\normalfont{(\alph*)}]
                    \item $\Gamma \setminus \bigcup^{\tilde q}_{i = 1} B(p_i, r) = S \setminus \bigcup^{\tilde q}_{i = 1} B(p_i, r)$.
                    \item $\bF(\Gamma, \Sigma) < \varepsilon$.
                    \item $\Gamma \cap \bigcup^{\tilde q}_{i = 1} \overline{B(p_i, r)}$ is a union of disjoint closed disks $\{\overline{D_j}\}^n_{j = 1}$ satisfying 
                    $\sum^n_{j = 1} \mathcal{H}^1(\partial D_j) < \varepsilon$ and $\sum^n_{j = 1} \operatorname{diam}(\overline{D_j}) < \varepsilon\,.$
                \end{enumerate}
        \end{enumerate}
    \end{lem} 
    \begin{proof}
        Since $\Sigma$ is either an empty set or a smooth closed surface, there exists a positive constant $C_\Sigma > 1$ depending only on $M, \mathbf{g}$ and $\Sigma$, such that for any $p \in M$ and $r \in (0, \operatorname{injrad}(M, \bg) / 2)$,
        \[
            \mathcal{H}^2(\overline{B(p, r)} \cap \Sigma) \leq C_\Sigma r^2\,, \quad \mathcal{H}^3(\overline{B(p, r)}) \leq C_\Sigma r^3\,.
        \]

        By the Jordan-Schoenflies theorem and Lemma~\ref{lem:isop_ineq}, there exists $\nu_M > 0$, depending only on $(M, \bg)$, such that for any $p \in M$ and $r \in (0, \operatorname{injrad}(M, \bg) / 2)$, every simple loop $c$ in $\partial B(p, r)$ bounds a disk $D \subset \partial B(p, r)$ with 
        \begin{equation}\label{eqn:approx_PS_isop_ineq}
            \mathcal{H}^2(D) \leq \nu_M \mathcal{H}^1(c)^2\,.
        \end{equation} 

        Let $\eta_1 = \delta(M, \bg, \Sigma, \varepsilon / 2) / 2$ be defined as in Corollary~\ref{cor:Pitts_bF_flat_quant}. We choose $\eta_2 \in (0, \operatorname{injrad}(M, \bg) / 4)$ such that
        \[
            (2\nu_M + 1) \eta^2_2 < \eta_1\,, \quad qC_\Sigma (\eta_2)^3 < \eta_1\,, \quad\text{and } 2\eta_2 < \varepsilon \,.
        \]
        Let $C_q > 1$ be a constant depending only on $q$ and determined later. And we set $\eta_3 = \frac{\eta_2}{8q C_q C_\Sigma}$. 
        
        By the definition of $\mathbf{F}$-metric, there exists $\delta \in (0, \varepsilon / 2)$ such that for any $S \in \GS(M)$ with $\mathbf{F}(S, \Sigma) \leq \delta$ and any $p \in M$,
        \[
            \mathcal{H}^2(\overline{B(p, \eta_3)} \cap S) \leq 2\mathcal{H}^2(\overline{B(p, 2 \eta_3)} \cap \Sigma) \leq 2C_\Sigma (2\eta_3)^2 = 8 C_\Sigma (\eta_3)^2\,.
        \] 

        In the sequel, we fix a punctate surface $S$ with a punctate set $P$ satisfying $ \mathbf{F}(S, \Sigma) \leq \delta$ and $\#P \leq q$. 
        By Sard's theorem and the coarea formula, we can choose a constant $C_q > 0$, only depending on $q$, and a radius $\tilde r \in \left(\eta_3 / 3^q, 2\eta_3 / 3^q\right)$ such that, for any $p \in P$ and any $l \in \{1, 3, 3^2, \dots, 3^q\}$, $\partial B(p, l\tilde r) \pitchfork S$ is a finite union of disjoint simple loops and
        \[
            \mathcal{H}^1(\partial B(p, l\tilde r) \cap S) \leq C_q \mathcal{H}^2(\overline{B(p, 2\eta_3)} \cap \Sigma) / \eta_3 < 8C_qC_\Sigma \eta_3 = \eta_2 / q\,.
        \] 

        By Lemma~\ref{lem:disjoint_balls} with $n = 1$, there exist $P' \subset P$ and a radius $r \in \{\tilde r, 3\tilde r, 3^2 \tilde{r}, \dots, 3^q \tilde{r}\} \subset (0, \eta_3)$ such that $\breve\Gamma = S \setminus \bigcup_{p' \in P'} B(p', r)$ is a smooth surface with boundary satisfying
        \[
            \sum_{p' \in P'} \mathcal{H}^1(\partial B(p', r) \cap S) < (\# P') \eta_2/q \leq \eta_2\,,
        \]
        $P \subset \bigcup_{p' \in P'} B(p', r)$ and $\{\overline{B(p', r)}\}_{p' \in P'}$ are disjoint from each other. Clearly, we also have $\mathfrak{g}(\breve\Gamma) \leq \mathfrak{g}(S)$. These confirms~\ref{item:approx_PS_disjoint},~\ref{item:approx_PS_transverse} and~\ref{item:approx_PS_genus}. %

        Then by~\eqref{eqn:approx_PS_isop_ineq}, for each $p' \in P'$, every simple loop $c$ in $\partial B(p', r) \cap S$ bounds a disk $D \subset \partial B(p', r)$ with 
        $\mathcal{H}^2(D) \leq \nu_M \left(\mathcal{H}^1(c)\right)^2\,.$
        In particular, for each $p' \in P'$, near $B(p', r)$, we can remove $B(p', r) \cap S$ and attach finitely many disks $\{D_{p', j}\}^{m_{p'}}_{j = 1}$ near $\partial B(p', r)$, yielding a smooth closed surface $\Gamma$ of genus $\leq \mathfrak{g}(S)$, confirming~\ref{item:approx_PS_approx}(a). 
        
        Furthermore, we have
        \begin{align*}
            \sum_{p' \in P'} \sum^{m_{p'}}_{j = 1} \mathcal{H}^2(D_{p', j}) &\leq 2\sum_{p' \in P'} \sum^{m_{p'}}_{j = 1} \nu_M(\mathcal{H}^1(\partial D_{p', j}))^2 \leq 2\nu_M [\sum_{p' \in P'} \sum^{m_{p'}}_{j = 1} \mathcal{H}^1(\partial D_{p', j}) ]^2 \leq 2\nu_M (\eta_2)^2\,.
        \end{align*} 

        Therefore, we have
        \begin{align*}
            |\mathcal{H}^2(\Gamma) - \mathcal{H}^2(S)| &\leq \sum_{p' \in P'} \sum^{m_{p'}}_{j = 1} \mathcal{H}^2(D_{p', j}) + \sum_{p' \in P'} \mathcal{H}^2(\overline{B(p', r)} \cap S)\\
                & < 2\nu_M (\eta_2)^2 + 8qC_\Sigma(\eta_3)^2  < \left(2\nu_M + 1\right) (\eta_2)^2 < \eta_1\,,
        \end{align*}
        and
        \[
            \mathcal{F}([\Gamma], [S]) \leq \sum_{p' \in P'} \mathcal{H}^3(B(p', r)) \leq q C_\Sigma r^3 < qC_\Sigma (\eta_2)^3 < \eta_1\,.
        \]
        By the choice of $\eta_1$ and Corollary~\ref{cor:Pitts_bF_flat_quant}, we have
        $\mathbf{F}(\Gamma, S) < \varepsilon /2\,,$
        and thus, 
       $\mathbf{F}(\Gamma, \Sigma) < \varepsilon\,,$
confirming~\ref{item:approx_PS_approx}(b). 
        
        Finally, we have 
        \[
            \sum_{p' \in P'} \sum^{m_{p'}}_{j = 1} \mathcal{H}^1(\partial D_{p', j}) = \sum_{p' \in P'} \mathcal{H}^1(\partial B(p', r) \cap S) < \eta_2 < \varepsilon\,,
        \]
        and 
        \[
            \sum_{p' \in P'} \sum^{m_{p'}}_{j = 1} \operatorname{diam}_M(\overline{D_{p', j}}) \leq 2\sum_{p' \in P'} \sum^{m_{p'}}_{j = 1} \mathcal{H}^1(\partial D_{p', j})  < 2\eta_2 < \varepsilon\,,
        \]
        confirming~\ref{item:approx_PS_approx}(c). 
    \end{proof} 

    Here is a generalization of Proposition~\ref{prop:short_loops_I} to punctate surfaces, as an application of the previous approximation lemma. 

    \begin{prop}[Existence of small balls I]\label{lem:small_balls_I}
        In an orientable closed Riemannian $3$-manifold $(M, \mathbf{g})$, for any integers $q, N\in \mathbb{N}^+$ and any real number $\varepsilon \in \left(0, \operatorname{injrad}(M, \mathbf{g})/(2N)\right)$, there exists a constant $\delta = \delta(M, \mathbf{g}, q, N, \varepsilon) > 0$ with the following property. 
        For any punctate surface $S \in \mathcal{S}(M)$ with a punctate set $P$ such that $\mathfrak{g}(S) + \# P \leq q$ and $\mathcal{H}^2(S) < \delta$,
        there exist an integer $\tilde q \leq q$, a real number $r \in (\frac{\varepsilon}{2(3N)^q}, \varepsilon)$ and geodesic balls $\{B(p_i, r)\}^{\tilde q}_{i = 1}$ in $(M, \mathbf{g})$ such that:
        \begin{enumerate}[label=\normalfont(\arabic*)]
            \item\label{item:small_balls_I_disjoint} The closed balls $\{\overline{B(p_i, Nr)}\}^{\tilde q}_{i = 1}$ are pairwise disjoint geodesic balls.
            \item\label{item:small_balls_I_transverse} For each $i \in \{1, 2, \dots, \tilde q\}$ and $l \in \{1, 2, \dots, N\}$, $\partial B\left(p_i, lr\right) \pitchfork S$ is a union of finitely many disjoint loops.
            \item\label{item:small_balls_I_genus} $S \setminus \bigcup^{\tilde q}_{i = 1} B\left(p_i, r\right)$ is a smooth genus $0$ surface with boundary.
        \end{enumerate} 
    \end{prop}
    \begin{proof}
        We set $
            \eta := \frac{\varepsilon}{2(3N)^q}\,.$
 Let $S$ be a punctate surface with a punctate set $P$ and $\# P \leq q$. By Sard's theorem and the coarea formula, there exist a constant $C = C(q)$ and a radius $r \in \left(\frac{\eta}{2 \cdot 3^q}, \frac{\eta}{3^q}\right)$ such that for any $p \in P$ and $l \in \{1, 3, 3^2, \dots, 3^q\}$, $\partial B(p, lr) \pitchfork S$ is a finite union of loops and
        \[
            \sum_{p \in P} \mathcal{H}^1(\partial B(p, lr) \cap S) \leq C {\mathcal{H}^2(S)}/{\eta}\,.
        \] 
        
        By Lemma~\ref{lem:disjoint_balls}, there exist $P' \subset P$ and $r' \in \{r, 3r, 3r^2, \dots, 3^q r\} \subset (\frac{\eta}{2 \cdot 3^q}, \eta)$ such that $\Gamma = S \setminus \bigcup_{p' \in P'} B(p', r')$ is a smooth surface with boundary satisfying
        \begin{equation}\label{eqn:est_by_coarea_form}
            \sum_{p' \in P'} \mathcal{H}^1(\partial B(p', r') \cap S) \leq C {\mathcal{H}^2(S)}/{\eta}\,, \quad P \subset \bigcup_{p' \in P'} B(p', r')\,,
        \end{equation}
        and $g := \mathfrak{g}(\Gamma) \leq \mathfrak{g}(S)$. 

        In the sequel, we regard $\Gamma$ as an \emph{intrinsic} surface, rather than a surface embedded in $(M, \mathbf{g})$. We also let $\tilde \delta = \delta(q, \eta)$ be the constant as determined in Proposition~\ref{prop:short_loops_I}. By~\eqref{eqn:est_by_coarea_form}, taking $\delta \in (0, \tilde \delta / 2)$ sufficiently small, if we further assume that $\mathcal{H}^2(S) < \delta$, then we can attach disks $\{D_i\}^n_{i=1}$ with arbitrarily small area to the boundary of $S \setminus \bigcup_{p' \in P'} B(p', r')$ to obtain a smooth genus $g$ closed surface $\Sigma$, and
        \begin{equation}\label{eqn:smoothing_est}
            \sum^n_{i = 1} \mathcal{H}^1(\partial D_i) \leq \tilde \delta, \quad \mathcal{H}^2(\Gamma) \leq 2 \mathcal{H}^2(S) < 2 \delta < \tilde \delta\,.
        \end{equation} 
        Moreover, it follows from Proposition~\ref{prop:short_loops_I} with~\eqref{eqn:smoothing_est}, there exists $g$ disjoint simple loops $\{c_j\}^g_{j = 1}$ in $\Sigma \setminus \bigcup^n_{i = 1} \overline{D_i} \subset S$ such that
        \begin{enumerate}[label = \normalfont(\roman*)]
            \item For each $j \in \{1, \dots, g\}$, each $c_j$ satisfies $\mathcal{H}^1(c_j) \leq \eta$.
            \item $\{[c_j]\}^g_{j = 1}$ is linearly independent in $H_1(\Sigma)$.
        \end{enumerate} 

        In particular, each $c_j$ is contained in some ball $B(p_j, \eta)$. Applying Sard's theorem again, there exists $\eta' \in (\eta, 2\eta)$ such that for any integer $l \in \{1, 2, 3, \dots, N(3N)^q\}$, each $B(p_j, l \eta')$ with $j \in \{1, 2, \dots, g\}$ intersects $S$ transversally, and so does each $B(p', l\eta')$ with $p' \in P'$. 
        
        In the sequel, we denote $\{B(p, \eta')\}_{p \in P} \cup \{B(p_j, \eta')\}^g_{j = 1}$ by $\{B(p_j, \eta')\}^{q'}_{j = 1}$, where $q' = g + \# P' \leq q$. Applying Lemma~\ref{lem:disjoint_balls} to $\{B(p_j, \eta')\}^{q'}_{j = 1}$ with $n = N$, we obtain $\tilde q \leq q'$, $\{p'_j\}^{\tilde q}_{j = 1} \subset \{p_j\}^{q'}_{j = 1}$, and a radius $\tilde r = (3N)^{q' - \tilde q}\eta' \in (\frac{\varepsilon}{2(3N)^q}, \varepsilon)$, such that the closed balls $\{\overline{B(p'_j, N \tilde r)}\}^{\tilde q}_{j = 1}$ are disjoint from each other. This confirms~\ref{item:small_balls_I_disjoint}. Finally, the transversality result ~\ref{item:small_balls_I_transverse} follows from the choice of $\eta'$, and the genus result ~\ref{item:small_balls_I_genus} follows from the choice of the balls. 
    \end{proof} 

    Similarly, we can extend Proposition~\ref{prop:short_loops_II} to punctate surfaces. 

    \begin{prop}[Existence of small balls II] \label{lem:small_balls_II}
        In a $3$-dimensional orientable closed Riemannian manifold $(M, \mathbf{g})$, for any $q, N \in \N^+$, $g_0 \in \N$, $\mathscr{S}_{g_0}$ a compact subset of embedded smooth orientable genus $g_0$ closed surfaces into $(M, \mathbf{g})$ (endowed with smooth topology), $\varepsilon \in (0, \operatorname{injrad}(M, \mathbf{g}) / (2N))$, there exists $\delta = \delta(M, \mathbf{g}, \mathscr{S}_{g_0}, q, N, \varepsilon)> 0$ with the following property. 
        For any punctate surface $S \in \GS(M)$ with punctate set $P$ that satisfies
      $\mathfrak{g}(\Sigma) - g_0 + \# P \leq q$ and any smooth surface $\Sigma \in \mathscr{S}_{\fg_0}$ with $\bF(S, \Sigma)\leq \delta$, there exist an integer $\tilde q \leq q$, a real number $r \in \left(\frac{\varepsilon}{2(3N)^q}, \varepsilon\right)$ and geodesic balls $\{B(p_i, r)\}^{\tilde q}_{i = 1}$ in $(M, \mathbf{g})$ such that:
        \begin{enumerate}[label=\normalfont{(\arabic*)}]
            \item\label{item:small_balls_II_disjoint} The closed balls $\{\overline{B(p_i, Nr)}\}^{\tilde q}_{i = 1}$ are pairwise disjoint geodesic balls.
            \item\label{item:small_balls_II_transverse} For each $i \in \{1, 2, \dots, \tilde q\}$ and $l \in \{1, 2, \dots, N\}$, $\partial B(p_i, l r) \pitchfork S$ is a union of finitely many loops.
            \item\label{item:small_balls_II_genus} $S \setminus \bigcup^{\tilde q}_{i = 1} B(p_i, r)$ is a smooth genus $g_0$ surface with boundary.
        \end{enumerate}
    \end{prop} 
    \begin{proof} $\,$
         {\textbf{Step 1.}} We prove the proposition when $\mathscr{S}_{g_0} = \{\Sigma\}$ is a singleton set. 
        Let $\tilde \delta_1(\eta) = \delta(M, \mathbf{g}, \Sigma, q, \eta) \in (0, \eta)$ be as determined in Proposition~\ref{prop:short_loops_II}.
        We set $\eta_1 := \frac{\varepsilon}{2(3N)^q}$, and $
            \eta_2 := \tilde \delta_1(\eta_1) / 2\,.$
        Let $\tilde \delta_2 = \delta(M, \bg, \Sigma, q, \eta_2)$ be as in Lemma~\ref{lem:approx_PS}. If $\bF(S, \Sigma) < \delta$, then there exist a finite set $P' \subset P$, a radius $r' \in (0, \eta_2)$ and an embedded smooth orientable closed surface $\Gamma$ and a finite set of closed disks $\{\overline{D}_j\}^n_{j = 1}$ in $\Gamma$ such that:
        \begin{enumerate}
            \item $\Gamma \setminus \bigcup^n_{j = 1} \overline{D_j} = S \setminus \bigcup_{p' \in P'} \overline{B(p', r')}$ and thus, $g_1 := \mathfrak{g}(\Gamma) \leq \mathfrak{g}(S)$.
            \item $\mathbf{F}(\Gamma, S) \leq 2\eta_2$.
            \item $\sum^n_{j = 1} \operatorname{diam}_M (\overline{D_j}) < \eta_2$.
            \item $\sum^n_{j = 1} \mathcal{H}^1(\partial\overline{D_j}) < \eta_2$.
        \end{enumerate}

        It follows from Proposition~\ref{prop:short_loops_II} that there exist $g = g_1 - \mathfrak{g}(\Sigma)$ disjoint simple loops $\{c_j\}^g_{k = 1}$ in $\Gamma' \setminus \bigcup^n_{j = 1} \overline{D_j} \subset S$, which are homologically independent in $H_1(\Gamma')$. Moreover, each $c_k$ satisfies $\mathcal{H}^1(c_k) \leq \eta_1\,,$
        so there exists $\tilde p_k \in M$ and $c_k \subset B(\tilde p_k, \eta_1)$. 

        We set $P'':= P' \cup \{\tilde p_k\}^g_{k = 1}$. By Sard's Lemma, there exists $\tilde r \in (\eta_1, 2\eta_1) = \left(\frac{\varepsilon}{2(3N)^q}, \frac{\varepsilon}{(3N)^q}\right)$ such that for every $p \in P''$, $l \in \{1, 2, \dots, N(3N)^q\}$,  $
            \partial B(p, l\tilde r) \pitchfork \Sigma$ is a union of finitely many loops. Then applying Lemma~\ref{lem:disjoint_balls} again to $\# P''$ balls
$\{B(p, \tilde r)\}_{p \in P''}$
        with $n = N$, we obtain $\tilde q \leq \# P'' \leq q$, $r \in (\frac{\varepsilon}{2(3N)^q}, \varepsilon)$ and disjoint geodesic balls $\{B(p_i, r)\}^{\tilde q}_{i = 1}$ satisfying ~\ref{item:small_balls_II_disjoint},~\ref{item:small_balls_II_transverse},~\ref{item:small_balls_II_genus} of the proposition. In particular, we have defined $\delta(M, \bg, \varepsilon, q, N, \{\Sigma\}) := \tilde \delta_2$ for any embedded smooth orientable closed genus $g_0$ surface $\Sigma$. 

        \medskip
        \paragraph*{\textbf{Step 2}} Let $\mathscr{S}_{g_0}$ be a compact set with respect to the smooth topology. 
        Since the smooth topology is finer than the topology induced by the $\mathbf{F}$-distance defined in Definition~\ref{def:bF_metric_GS}, $\mathscr{S}_{g_0} \subset \GS$ is a compact set with respect to the $\mathbf{F}$-distance. For each $\Sigma \in \mathscr{S}_{g_0}$, we define 
        \[
            B_\Sigma := \left\{\Gamma \in \mathscr{S}_{g_0} : \mathbf{F}(\Sigma, \Gamma) < \frac{\delta(M, \mathbf{g}, \{\Sigma\}, q, N \varepsilon)}{4} \right\}\,,
        \]
        where $\delta(M, \mathbf{g}, \{\Sigma\}, q, N, \varepsilon)$ is already defined in Step 1. Then the open balls $\{B_\Sigma\}_{\Sigma \in \mathscr{S}_{g_0}}$
        is an open covering of $\mathscr{S}_{g_0}$. By compactness, it has a finite subcovering 
            $\{B_{\Sigma_l}\}^{m}_{l = 1}.$
        Therefore, we can define 
        \[
            \delta(M, \bg, \mathscr{S}_{g_0}, q, N, \varepsilon) = \min_{l \in \{1, 2, \dots, m\}} \frac{\delta(M, \bg, \{\Sigma_l\}, q, N, \varepsilon)}{4}\,, 
        \]
        and then ~\ref{item:small_balls_II_disjoint},~\ref{item:small_balls_II_transverse},~\ref{item:small_balls_II_genus} of the proposition hold.  
    \end{proof} 

\section{Pinch-off processes and genus reduction} \label{sect:interpolation} 
    Let $g_0 \leq g_1$ be two natural numbers. Intuitively, for a Simon--Smith family of genus $\leq g_1$ sufficiently close to a smooth genus $g_0$ closed surface $\Sigma$, we show that one can pinch all the small necks and shrink all small connected components in a continuous way to obtain a Simon--Smith family of genus $g_0$. Note that by Lemma~\ref{lem:genusLarger}, each surface in the Simon--Smith family has genus at least $g_0$. 
With ingredients in this section, we will be able to prove  Theorem \ref{thm:pinchOffMinMax},  and Proposition \ref{prop:zeroWidth} and \ref{prop:pinchOffg0g1}.

    \begin{prop}[Pinch-off proposition I]\label{Prop_Technical Deformation}
        Let $(M, \bg)$ be an orientable closed Riemannian $3$-manifold, $X$ be a cubical subcomplex of $I(m, k)$, $g_0 \in \mathbb{N}$, $\mathscr{S}_{g_0}$ be a compact subset of embeddings of a genus $g_0$ surface into $M$ (endowed with smooth topology), $q \in \N$ and $\varepsilon \in (0,1)$. Then there exists $\delta = \delta(M, \bg, \mathscr{S}_{g_0}, m, q, \varepsilon)\in (0, \varepsilon)$ with the following property. 
        If $\Phi: X\to \GS(M)$ is a Simon--Smith family of genus $\leq g_1$ with 
        $N(P_\Phi) + g_1 - g_0 \leq q\,,$ and $\Sigma: X\to \mathscr{S}_{g_0}$ is a map such that for every $x \in X$,  $
            \bF(\Phi(x), \Sigma(x)) \leq \delta\,,$
        then there exists a Simon--Smith family $H: [0, 1] \times X \to \GS(M)$ of genus $\leq g_1$ such that for any $x \in X$:
        \begin{enumerate}[label=\normalfont(\arabic*)]
            \item $H(0, x) = \Phi(x)$, $\mathfrak{g}(H(1, x)) = g_0$.
            \item\label{item:Prop_Technical Deformation_       _area_genus_control} For any $0\leq t \leq t'\leq 1$,
                \[
                    \mathcal{H}^2(\Phi(x)) - \varepsilon \leq \cH^2(H(t, x)) \leq \cH^2(\Phi(x)) + \varepsilon\,,\quad 
                    \cF([H(t, x)], [\Phi(x)]) \leq \varepsilon\,,
                \]
                and $\mathfrak{g}(H(t', x)) \leq \mathfrak{g}(H(t, x)) \,.
                $
            \item The Simon--Smith family $\{H(t, x)\}_{t \in [0, 1]}$ is a pinch-off process.
        \end{enumerate} 
    \end{prop} 

    Moreover, if $\Phi: X \to \GS(M)$ consists of punctate surfaces of small area, we can perform pinch-off processes to obtain a Simon--Smith family of genus $0$. This can be viewed as a special case of the previous proposition by taking $g_0 = 0$ and $\mathscr{S}_0 = \{\emptyset\}$.

    \begin{prop}[Pinch-off proposition II]\label{prop:Technical Deformation_II}
        Let $(M, \bg)$ be an orientable closed Riemannian $3$-manifold, $X$ be a cubical subcomplex of $I(m, k)$, $q \in \N$ and $\varepsilon \in (0,1)$. Then there exists $\delta = \delta(M, \bg, q, \varepsilon)\in (0, \varepsilon)$ with the following property.
        If $\Phi: X\to \GS(M)$ is a Simon--Smith family of genus $\leq g_1$ with  $N(P_\Phi) + g_1 \leq q\,,$ and 
        $\sup_{x \in X} \mathcal{H}^2(\Phi(x)) \leq \delta\,,$
        then there exists a Simon--Smith family $H: [0, 1] \times X \to \GS(M)$ of genus $\leq g_1$ such that for any $x \in X$:
        \begin{enumerate}[label=\normalfont(\arabic*)]
            \item $H(0, x) = \Phi(x)$, $\mathfrak{g}(H(1, x)) = 0$.
            \item\label{item:area_control} For any $0\leq t \leq t'\leq 1$,
                $\cH^2(H(t, x)) \leq \cH^2(\Phi(x)) + \varepsilon$ and  $\mathfrak{g}(H(t', x)) \leq \mathfrak{g}(H(t, x)) \,.$
            \item The Simon--Smith family $\{H(t, x)\}_{t \in [0, 1]}$ is a pinch-off process.
        \end{enumerate} 
    \end{prop} 

\subsection{Local deformations}
    Recall the following local pinch-off processes, Pinching and Shrinking, introduced in \cite{chuLi2024fiveTori}. 
    
    \begin{lem}[Local Deformation A: Pinching] \label{Lem_Loc Deform A}
        In an orientable closed Riemannian $3$-manifold $(M, \bg)$, there exists a constant $C_\mathbf{g}$ with the following property. 
        Let $\Lambda>0$, $r \in \left(0, \injrad(M, \mathbf{g})/4 \right)$ and $p\in M$. Suppose that $\Phi: X \to \GS(M)$ is a Simon--Smith family and $x_0 \in X$ so that 
        \[
            \cH^2(\Phi(x_0) \cap A(p; r, 2r)) \leq \Lambda r^2\,.
        \]
        Then for any $\eta > 0$, there exist $s\in (r, 2r)$, $\zeta\in (0, \min\{\Lambda r, s-r, 2r-s\}/5)$, an integer $K\geq 0$, a neighborhood $O_{x_0}$ of $x_0$ in $X$, and a Simon--Smith family 
            $H: [0, 1] \times O_{x_0}\to \GS(M) \,,$
        with the following properties.
        \begin{enumerate}[label=\normalfont(\arabic*)]
            \item For any $y \in O_{x_0}$, $\Sigma_y := \Phi(y) \cap A(p; s-5\zeta, s+5\zeta)$ is a smooth surface, varying smoothly in $y$. 
            \item For any $s'\in [s-4\zeta, s+4\zeta]$ and $y \in O_{x_0}$, $\Phi(y)$ intersects $\partial B(p, s')$ transversally, and 
            \[
                \Phi(y)\cap \partial B(p, s') =: \bigsqcup_{i=1}^K \gamma_i(s', y)\,, \quad \sum_{i=1}^K \cH^1(\gamma_i(s', y)) \leq 2\Lambda r\,,
            \]
            where each $\gamma_i(s',y)$ is a simple loop. 
        \end{enumerate} 
        
        If $K = 0$, then $H(t, y) \equiv \Phi(y)$ holds true for any $y \in O_{x_0}$ and $t\in [0, 1]$; and if $K\geq 1$, then: 
        \begin{enumerate}[label=\normalfont(\arabic*)]
            \setcounter{enumi}{2}
            \item For any $y \in O_{x_0}$, $H(0, y) = \Phi(y)$. 
            \item There exist $0 = t_0<t_1<t_2<\dots<t_K<t_{K+1} = 1$ such that for any $0\leq i\leq K$, $t\in (t_i, t_{i+1})$ and $y \in O_{x_0}$, 
            \[
                \Sigma_{t, y} := H(t, y) \cap A(p; s-5\zeta, s+5\zeta) 
            \] 
            is a smooth surface with boundary, which is diffeomorphic to $\Sigma_y \setminus \bigsqcup_{1\leq j\leq i} \gamma_j(s, y)$ attached with $2i$ disks. Moreover, $H(1, y)\cap A(p; s-5\zeta, s+5\zeta)$ is also a smooth surface and
            \[
                H(1, y)\cap A(p; s-\zeta, s+\zeta) = \emptyset \,.
            \] 
            \item\label{item:surgery_process} For any $i \in \{1, 2, \dots, K\}$ and $y \in O_{x_0}$, near $t_i$, $H(t, y)$ is a surgery process via pinching the cylinder $\bigcup_{s' \in (s - 2\zeta, s + 2\zeta)} \gamma_i(s', y)$. 
            \item\label{item:Lem_Loc Deform A_estimates} For any $y \in O_{x_0}$ and $t \in [0, 1]$, we have
            \begin{align*}
                H(t, y)\setminus A(p; s-3\zeta, s+3\zeta) &= \Phi(y)\setminus A(p; s-3\zeta, s+3\zeta) \,,\\
                \cH^2(\Phi(y) \cap A(p; s - 3\zeta, s + 3\zeta)) - \eta &\leq \cH^2(H(t, y) \cap A(p; s - 3\zeta, s + 3\zeta))\\ &\leq \cH^2(\Phi(y) \cap A(p; s - 3\zeta, s + 3\zeta)) + C_\bg\Lambda^2 r^2 \,,
            \end{align*}
            and $\mathcal{F}([\Phi(y))], [H(t, y)]) \leq \eta\,.$
        \end{enumerate}
    \end{lem} 
    \begin{rmk}
        For every $y \in O_{x_0}$, by~\ref{item:surgery_process}, $\{H(t, y)\}_{t \in [0, 1]}$ is a pinch-off process.
    \end{rmk} 

    \begin{proof}[Sketch of proof]
        The only difference between  \cite[Lemma~8.3]{chuLi2024fiveTori} and this lemma is the inequalities in ~\ref{item:Lem_Loc Deform A_estimates}, where we additionally require a lower bound on $\cH^2(H(t,y))$ and that $[\Phi(y)]$ and $[H(t,y)]$ are close in the flat norm. The same proof applies, and one observes that by choosing $\zeta$ sufficiently small, both inequalities are automatically satisfied, since $H(t, y)$ is obtained by pinching cylinders of height at most $\zeta$. 
    \end{proof}
    
    \begin{lem}[Local Deformation B: Shrinking, {\cite[Lemma~8.4]{chuLi2024fiveTori}}] \label{Lem_Loc Deform B}
        In an orientable closed Riemannian $3$-manifold $(M, \bg)$, suppose that $p\in M$ and $0< r^- < r^+ < \injrad(M, \mathbf{g})/4$. Then there exists a smooth one-parameter family of maps $\{\cR_t: M\to M\}_{t\in [0,1]}$ such that: 
        \begin{enumerate}[label=\normalfont(\arabic*)]
            \item $\cR_t = \id$ for any $t\in [0, 1/4]$; $\cR_t(B(p, r^-)) = \{p\}$ for any $t\in [1/2, 1]$.
            \item $\cR_t$ is a diffeomorphism for any  $t\in [0, 1/2)$.
            \item $\cR_t|_{M \setminus B(p, r^+)} = \id$ and $\cR_t(B(p, r^\pm))\subset B(p, r^\pm)$ for any $t\in [0, 1]$.
            \item $\cR_t|_{B(p, r^-)}$ is $1$-Lipschitz for any $t\in [0, 1]$.
        \end{enumerate}
    \end{lem} 
    
    \begin{lem}[Local Deformation C: Successive shrinkings, {\cite[Lemma~8.5]{chuLi2024fiveTori}}] \label{Lem_Loc Deform Composition} 
        In an orientable closed Riemannian $3$-manifold $(M, \bg)$, for $N \in \N^+$, suppose that $\{A(p_j; r_j^-, r_j^+)\}_{j=1}^N$ is a collection of pairwise disjoint annuli in $M$, where $0 < r_j^- < r_j^+<\injrad(M, \mathbf{g})/4$. For any $j \in \{1, 2, \dots, N\}$, let $\cR^j_t: M\to M$ be a map from Lemma~\ref{Lem_Loc Deform B} with $p_j, r_j^\pm$ in place of $p, r^\pm$.
        For simplicity, we denote $
            A_j := A(p_j; r^-_j, r^+_j)$, and $B^\pm_j:= B(p_j, r^\pm_j)\,.$ 
        Then the following hold: 
        \begin{enumerate}[label=\normalfont(\arabic*)]
            \item There exists a permutation $\sigma\in \mathfrak{S}_N$ such that for any $1\leq i<j\leq N$, either $B_{\sigma(i)}^+\cap B_{\sigma(j)}^+ = \emptyset$, or $B_{\sigma(i)}^+\subset B_{\sigma(j)}^-$. In this case, we call such a $\sigma$ {\em admissible}. 
            \item For any admissible $\sigma$, the map 
            \[
               \cR^{\sigma(N)}_t\circ\cR^{\sigma(N-1)}_t\circ \cdots \circ \cR^{\sigma(2)}_t\circ\cR^{\sigma(1)}_t: M \to M
            \]
            is $1$-Lipschitz on every connected component of $M\setminus \bigcup_{j=1}^N A_j$; and it is independent of the choice of admissible $\sigma$. 
        \end{enumerate} 
    \end{lem} 

\subsection{Combinatorial arguments}
    
    Recall the following combinatorial arguments in \cite{chuLi2024fiveTori}, inspired by~\cite[Lemma~4.8,~Proposition~4.9]{Pit81} and~\cite[Lemma~5.7]{De_Lellis_Jusuf_2018_min_max}. 
    
    \begin{lem}[Combinatorial Argument I, {\cite[\S 8.3]{chuLi2024fiveTori}}]\label{lem:combinatorial_I}
        For any positive integers $m \in \N^+$ and $q \in \N^+$, there exists $N = N(m, q) \in \N^+$ with the following property.

        Given any $k \in \N^+$ and a cubical complex $I(m, k)$, suppose that $\{\cF_\sigma\}_{\sigma\in I(m,k)}$ is a family of collections of open sets in a closed manifold $M$ assigned to each cell $\sigma$ of $I(m, k)$, where each collection $\cF_\sigma$ has the form
        \[
            \cF_\sigma = (\cO_{\sigma, 1}, \cO_{\sigma,2}, \dots, \cO_{\sigma, q})
        \]
        and each $\cO_{\sigma, i} = \{U^1_{\sigma, i}, U^2_{\sigma, i}, \dots, U^N_{\sigma, i}\}$ satisfies
        \begin{equation}\label{eqn:dist_diam}
            \dist(U^r_{\sigma, i}, U^s_{\sigma, i}) \geq 2 \min\{\diam(U^r_{\sigma, i}), \diam(U^s_{\sigma, i})\}\,,
        \end{equation}
        for all $\sigma \in I(m, k)$, $i\in \{1, 2, \dots, q\}$ and $r, s \in \{1, 2, \dots, N\}$ with $r \neq s$.
        Then we can extract a family of open sets 
        \[
            \{(U_{\sigma, 1}, U_{\sigma,2}, \dots, U_{\sigma, q})\}_{\sigma \in I(m, k)}
        \]
        where $U_{\sigma, i} \in \cO_{\sigma, i}$ such that $\dist(U_{\sigma, i} \cap U_{\tau, j}) > 0$ whenever $\sigma, \tau \in I(m, k)$, $(\sigma, i) \neq (\tau, j)$ and $\sigma, \tau$ are faces of a common cell $\gamma$ of $I(m, k)$.

        Moreover, this is also true if the number of collections in each $\mathcal{F}_\sigma$ is no greater than $q$.
    \end{lem}

    \begin{lem}[Combinatorial argument II, {\cite[\S 8.3]{chuLi2024fiveTori}}]\label{lem:combinatorial_II}
        In a closed Riemannian manifold $(M, \bg)$, given any integers $q, N \in \N^+$ and any $R_0 \in (0, \injrad(M, \mathbf{g}) / 4)$, if $P$ is a finite set of at most $q$ points, then there exists a radius $R \in (5^{-2Nq^2}R_0, 5^{-2N} R_0)$ such that for every $p \in P$, we have 
            $A(p; R, 5^{2N} R) \cap P = \emptyset\,.$
    \end{lem}

\subsection{Proof of Proposition~\ref{Prop_Technical Deformation} and Proposition~\ref{prop:Technical Deformation_II}}

The proof will proceed in three steps.

\medskip
    
\paragraph*{\bf Step 1. Set up}
    Let $C_\mathbf{g}$ be as in Lemma~\ref{Lem_Loc Deform A} and $N := N(m, q)$ as in Lemma~\ref{lem:combinatorial_I}. 
    
    In Proposition~\ref{Prop_Technical Deformation}, since $\mathscr{S}_{g_0}$ is a compact set of embedded smooth closed surfaces, without loss of generality, we assume that 
    $\varepsilon < \min\left\{\injrad(M, \bg), 1\right\}\,,$
    such that for any $r \in (0, \varepsilon / 1000)$, we have $\sup_{p \in M}\cH^3(B(p, 2r))$
    and 
    \begin{equation}\label{eqn:area_ball_ratio_0}
        \sup_{\Sigma \in \mathscr{S}_{g_0},\ p \in M} \left\{\cH^2\left(\Sigma\cap B(p, 2r)\right)\right\} \leq 50 r^2\,.
    \end{equation} 
    We set 
    \[
        \varepsilon_1:= \frac{\varepsilon}{10000Nq\cdot3^m \max(C_\bg, 1)}\quad\textrm{ and }\quad
        \varepsilon_2 := \frac{\varepsilon_1}{2(3N)^{q}q^2 \cdot 5^{2N} }\,.
    \] 
        
    By~\eqref{eqn:area_ball_ratio_0}, we choose $\delta_0 = \delta_0(\mathscr{S}_{g_0}) > 0$ such that for any $S \in \GS(M)$, if $\bF(S, \Sigma) < \delta_0$ for some $\Sigma \in \mathscr{S}_{g_0}$, then for any $r \in (\varepsilon_2, \varepsilon_1)$, 
    \begin{equation}\label{eqn:area_ball_ratio}
        \sup_{p \in M} \left\{\cH^2\left(S\cap B(p, 2r)\right)\right\} \leq 100 r^2\,,
    \end{equation} 
    and in particular, 
    \begin{equation}\label{eqn:area_ratio}
        \sup_{p \in M} \left\{\cH^2\left(S\cap A(p; r, 2r)\right)\right\} \leq 100 r^2\,.
    \end{equation}
    Note that these inequalities also hold for Proposition~\ref{prop:Technical Deformation_II}. 
    
    In the following, we choose $\delta_1 := \delta(M, \bg, q, N, 2(3N)^q \varepsilon_2)$ as defined in Proposition~\ref{lem:small_balls_I} and $\delta_2 := \delta(M, \bg, \mathscr{S}_{g_0}, q, N, 2(3N)^q \varepsilon_2)$ in Proposition~\ref{lem:small_balls_II}. Set
    $\delta := \min\{\delta_0, \delta_1, \delta_2\}\,.$
    
    For a Simon--Smith family $\Phi$ as in Proposition~\ref{Prop_Technical Deformation} or Proposition~\ref{prop:Technical Deformation_II}, by Proposition~\ref{lem:small_balls_I} or Proposition~\ref{lem:small_balls_II} with $N = 1$ therein, for any $x \in X$, there exist $r_x \in (\varepsilon_2, 2(3N)^q\varepsilon_2)$ and a finite set of balls $\{B(p_{x, i}, r_x)\}^{q_x}_{i = 1}$ such that
    \begin{enumerate}[label = (\roman*)]
        \item $q_x \leq q$;
        \item $\Phi(x) \setminus \bigcup^{q_x}_{i = 1} B(p_{x, i}, r_x)$ is a smooth surface of genus $g_0$ with boundary.
    \end{enumerate}
    Note that when $\Phi$ is a family as in Proposition~\ref{prop:Technical Deformation_II}, $g_0 = 0$. 
    
    We write  $\tilde P(x) := \{p_{x, i}\}^{q_x}_{i = 1}$
    and for every $l \in \{1, 2, \dots, N\}$, $
        r_{x, l} := 5^{2(l - 1)} r_x\,.$
    By \eqref{eqn:area_ball_ratio} and \eqref{eqn:area_ratio}, for every $1 \leq l \leq N$, $x \in X$ and $p \in \tilde P(x)$, since $r_{x, l} \in (\varepsilon_2, \varepsilon_1)$, we have
    \begin{align}
        \cH^2 \left(\Phi(x)\cap B(p, 2r_{x,l})\right) \leq 100 r_{x,l}^2 \,,
    \end{align}
    and
    \begin{align}
        \cH^2 \left(\Phi(x)\cap A(p; r_{x,l}, 2r_{x,l})\right) \leq 100 r_{x,l}^2 \,. \label{Equ_Uniform area bound in small annuli}
    \end{align}

    Thus, applying Lemma~\ref{Lem_Loc Deform A} with $\eta = \varepsilon / 2$ to $\Phi(x)$ in $A(p; r_{x,l}, 2r_{x,l})$ and $\Lambda = 100$, we obtain the following data:
    \begin{enumerate}
        \item $s_{x,l,p}\in (r_{x,l}, 2r_{x,l})$,
        \item $\zeta_{x,l,p}\in (0, \min\{100 r_{x,l}, s_{x,l}-r_{x,l}, 2r_{x,l}-s_{x,l}\}/5)$,
        \item a neighborhood $O_{x,l,p}\subset X$ of $x$,
        \item and a Simon--Smith family $
            H_{x,l,p}: [0, 1] \times O_{x,l,p}\to \GS(M)\,.$
    \end{enumerate}
    In addition, for each $x \in X$, we fix a neighborhood of $x$
    \[
        O_x \subset \bigcap_{l \in \{1, 2, \dots, N\},\ p \in \tilde P(x)} O_{x, l, p}\,,
    \]
    such that $y \mapsto \Phi(y) \setminus \bigcup_{p \in \tilde P(x)} B(p, r_x)$ is continuous in the smooth topology for $y \in O_x$. 
    
\medskip
\paragraph*{\bf Step 2. Refinement} 
    
    Since $X$ is a compact cubical subcomplex of $I(m, k)$, its open cover $\{O_x\}_{x \in X}$ has a finite subcover $\{O_{x_j}\}$. 
    
    Then we refine $X$ and assign to each cell $\sigma$ of $X$ $x_\sigma \in \{x_j\}$ such that: Every cell $\tau$ of $X$ is a subset of $O_{x_\sigma}$ provided that $\sigma$ and $\tau$ are faces of some cell $\gamma$ of $X$. Note that for each $\sigma$, the number of such $\tau$ is no more than $5^m$. 
    By Lemma~\ref{lem:combinatorial_I} with 
    \[
        \cF_\sigma = (\{A(p; r_{x_\sigma, i}, 2r_{x_\sigma, i})\}^N_{i = 1})_{p \in \tilde P(x_\sigma)}\,,
    \]
    we obtain, associated with each $\sigma$, a collection of annuli denoted by
    \[
        A_{\sigma, p} := A(p, r_{\sigma, p}, 2r_{\sigma, p})_{p \in \tilde P(x_\sigma)}
    \] such that whenever $\sigma$ and $\tau$ are faces of some cell $\gamma$ of $X$ and for any $p \in \tilde P(x_\sigma)$ and $p' \in \tilde P(x_\tau)$, we have 
        $A_{\sigma, p} \cap A_{\tau, p'} = \emptyset\,,$
    unless $(\sigma, p) = (\tau, p')$. 

\medskip

\paragraph*{\bf Step 3. Construction of $H$}
    
    For each cell $\sigma$ and $p \in \tilde P(x_\sigma)$, for convenience, we denote the data associated with each $A_{\sigma, p}$, constructed at the end of Step 1, as follows:
    \begin{enumerate}
        \item $s_{\sigma, p} \in (r_{\sigma, p}, 2r_{\sigma, p})$, 
        \item $\zeta_{\sigma, p} \in (0, \min\{s_{\sigma, p}-r_{\sigma, p}, 2r_{\sigma, p}-s_{\sigma, p}\}/5)$, 
        \item a neighborhood $O_{\sigma} := O_{x_\sigma}$,
        \item and a Simon--Smith family 
            $H_{\sigma, p}: [0, 1] \times O_{\sigma} \to \GS(M)\,.$
    \end{enumerate}
    We also denote
    \begin{align*}
        r_{\sigma, p}^\pm:= s_{\sigma, p}\pm \zeta_{\sigma, p}\,, \quad
        B_{\sigma, p}^\pm := B_{r_{\sigma, p}^\pm}(p)\,, \quad
        \hat{A}_{\sigma, p} := \overline{B_{\sigma, p}^+\setminus B_{\sigma, p}^-}\subset A_{\sigma, p} \,.
    \end{align*} 
    
    After the refinement in Step 2, suppose that $X$ is a cubical subcomplex of the unit cube $I(m, k')$. For each cell $\sigma$ of $X$ and for each $x \in X$, we define
    $\bd_\sigma(x) := \min\left\{{2\|x- c_\sigma\|_{\ell_\infty}}/{3^{-k'}} , 1\right\}\,,$
    to be the normalized $l^\infty$ distant function to $\sigma$, where $c_\sigma$ is the center of $\sigma$.
        
    For each $x \in X$, we define $H(t, x)$, $0 \leq t \leq 1/2$, using pinching surgeries as described in Lemma~\ref{Lem_Loc Deform A}: For every cell $\sigma$ of $X$ and every $p \in \tilde P(x_\sigma)$,
    \begin{align*}
        H(t, x)\cap A_{\sigma, p} = H_{\sigma, p}\left(\min\{8t(1-\bd_\sigma(x)), 1\}, x\right) \cap A_{\sigma, p}
    \end{align*}
    and $H(t, x)=\Phi(x)$ outside $\bigcup_{\sigma \in X, p \in \tilde P(x_\sigma)} A_{\sigma, p}$. 
    
    The map has the following properties for every $x \in X$:
    \begin{itemize}
        \item[(i)] The set 
            $Z_x:= \{\sigma: \bd_\sigma(x)<1\}$
        satisfies that every cell $\sigma$ in $Z_x$ is a face of the cell $\gamma$, which is the smallest cell containing $x$. By Step 2, the corresponding annuli $\{A_{\sigma, p}\}_{\sigma \in Z_x, p \in \tilde P(x_\sigma)}$ are pairwise disjoint, and the number is no greater than $3^m \cdot q$. In particular, the map $H(t, x)$ is well-defined.
        \item[(ii)] By Lemma~\ref{Lem_Loc Deform A} (4) and (5), $t \mapsto \mathfrak{g}(H(t, x))$ is non-increasing and $t \mapsto H(t, x)$ is a pinch-off process.
        \item[(iii)] By (\ref{Equ_Uniform area bound in small annuli}) and Lemma~\ref{Lem_Loc Deform A} (6), for every $t\in [0, 1/2]$, 
        \begin{align*}
            \cH^2(H(t, x)) &\leq \cH^2(\Phi(x)) + (3^m \cdot q) \cdot C_\bg 100^2 \varepsilon^2_1\leq \cH^2(\Phi(x)) + \varepsilon / 2\,,
        \end{align*}
        \[
            \cH^2(H(t, x)) \geq \cH^2(\Phi(x)) - \varepsilon / 2\,,
        \]
        and $\cF([H(t,x)], [\Phi(x)]) \leq \varepsilon / 2\,$;
        in fact, for every $x \in X$, $\sigma \in Z_x$ and $p \in \tilde P(x)$,
        \begin{equation}\label{eqn:est_ratio_ball_iii}
            \cH^2(H(t, x) \cap B^+_{\sigma, p}) \leq 100(r^+_{\sigma, p})^2 + \varepsilon/(100 q \cdot 3^m) < \varepsilon/(10 q \cdot 3^m)\,.
        \end{equation}
        \item[(iv)] If $\bd_{\sigma}(x)\leq 3/4$, then for each $p \in \tilde P(x_\sigma)$, in $A_{\sigma, p}$, $H(1/2, x) = H_{\sigma, p}(1, x)$, and thus by Lemma~\ref{Lem_Loc Deform A} (4), we know that 
        $ H(1/2, x)\cap \hat{A}_{\sigma, p} = \emptyset \,.$
    \end{itemize}

    Finally, for each $t \in [1/2, 1]$ and each $x \in X$, we define $H(t, x)$ using the shrinking process as described in Lemma~\ref{Lem_Loc Deform B}:
    For each cell $\sigma$ of $X$ and $p \in \tilde P(x_\sigma)$, we let $\cR_{(\sigma, p), t}: S^3\to S^3$ be the one-parameter family of shrinking deformation from Lemma~\ref{Lem_Loc Deform B} with $r^- = r^-_{\sigma, p}$ and $r^+ = r^+_{\sigma, p}$.
    Then for each $x \in X$, we label $\{(\sigma, p) \mid \sigma \in Z_x,\ p \in P'(x_\sigma)\}$ as
    $\{(\sigma_i, p_i)\}_{1 \leq i \leq N'}$
    such that for every $1 \leq i < j \leq N'$, either $B_{\sigma_i, p_i}^+\cap B_{\sigma_j, p_j}^+ = \emptyset$, or $B_{\sigma_i, p_i}^+\subset B_{\sigma_j, p_j}^-$, where $N' \leq 3^m \cdot q$. The existence of such labeling follows from Lemma~\ref{Lem_Loc Deform Composition} (1). Then, we denote for simplicity, 
    \[
        \hat{\cR}^{(i)}_{t, x} := \cR_{(\sigma_i, p_i), (2t-1)(1-\bd_{\sigma_i}(x))} \,.
    \]
    
    Note that 
    \begin{itemize}
        \item By Lemma~\ref{Lem_Loc Deform B} (1), if $\bd_{\sigma_i}(x)\geq 3/4$ or $t\in [1/2, 5/8]$, then $\hat{\cR}^{i}_{t, x}=\id$; when $\bd_{\sigma_i}(x)< 3/4$, by the construction above, $H(1/2, x)\cap \hat{A}_{\sigma_i, p_i} = \emptyset$.
        \item if $\bd_{\sigma_i}(x)\leq 1/2$, then $\hat{\cR}^{(i)}_{1, x}(B_{\sigma_i, p_i}^-) \subset \{p_i\}$.
    \end{itemize}
    
    Now for $t\in [1/2, 1]$ and $x \in X$, we define
    \begin{align*}
        H(t, x):= \hat{\cR}^{(N')}_{t, x}\circ\cdots\circ \hat{\cR}^{(2)}_{t, x}\circ \hat{\cR}^{(1)}_{t, x}\left(H(1/2, x) \right) \,.
    \end{align*}
    Intuitively, $\{H(t, x)\}_{t \in [1/2, 1]}$ is obtained by shrinking some connected components of $H(1/2, x)$ to points. Therefore, for every $x \in X$, $t\mapsto \mathfrak{g}(H(t, x))$ is non-increasing, and $\{H(t, x)\}_{t \in [1/2, 1]}$ is a pinch-off process. This confirms statement (3) of both propositions.
    
    By Lemma~\ref{Lem_Loc Deform B} (3), (4) and Lemma~\ref{Lem_Loc Deform Composition} (2), we see that $\cH^2(H(t, x))$ is non-increasing for $t \in [1/2, 1]$. In particular, for every $t\in [1/2, 1]$, 
    \[
        \cH^2(H(t, x)) \leq \cH^2(\Phi(x)) + \varepsilon \,.
    \]
    Also, by the estimate~\eqref{eqn:est_ratio_ball_iii}, one also have
    \[
        \cH^2(H(t, x)) \geq \cH^2(H(1/2, x)) - \varepsilon / (2 \cdot 3^q \cdot m) \cdot N' \geq \cH^2(\Phi(x)) - \varepsilon\,,
    \]
    and
    \begin{align*}
        \cF([H(t,x)], [\Phi(x)])  \leq \cF([H(t,x)], [H(1/2, x)]) + \cF([H(1/2, x)], [\Phi(x)]) \leq \varepsilon / 2 + N' (\varepsilon_1)^2 \leq \varepsilon\,,
    \end{align*}
    due to the essence of shrinking $N'$ balls.
    This confirms statement (2) of both propositions.
    
    Furthermore, by the construction above, $\fg(H(1, x)) = g_0$. This verifies statement (1) of both propositions and completes the proof.

\subsection{Proof of Theorem \ref{thm:pinchOffMinMax} and Proposition~\ref{prop:zeroWidth}}\label{sect:proof_prop:zeroWidth}
    Equipped with the ingredients above, the proof of Theorem \ref{thm:pinchOffMinMax} proceeds similarly to   \cite[\S 8.5]{chuLi2024fiveTori}: The details are included in Appendix \ref{sect:proof_pinchOffMinMax}.  
    As for the proof of Proposition~\ref{prop:zeroWidth},
    we suppose that $X=\operatorname{dmn}(\Phi)$ is a cubical subcomplex of $I(m, k)$. Let $q = N(P_\Phi) + g$, and $\delta = \delta(\varepsilon / 2, q, m)$ as in Proposition~\ref{prop:Technical Deformation_II}.
    Since $\bL(\Lambda(\Phi)) = 0$, there exists $\Phi' \in \Lambda(\Phi)$ with
    $\sup_{x \in X} \mathcal{H}^2(\Phi'(x)) < \delta\,.$
    Then, the proposition follows from applying Proposition~\ref{prop:Technical Deformation_II} on $\Phi'$.

\subsection{Proof of Proposition~\ref{prop:pinchOffg0g1}}

Following the proof of Proposition~\ref{Prop_Technical Deformation} and Proposition~\ref{prop:Technical Deformation_II}, we can similarly prove Proposition~\ref{prop:pinchOffg0g1} in three steps. This argument is much simpler, as it does not involve area control, and the family is ``continuous in the Simon--Smith sense" rather than merely with respect to the $\mathbf{F}$-metric.

\medskip
\paragraph*{\textbf{Step 1. Set up}} 
    Let $q = N(P_\Phi)$ and $m$ be the ambient dimension of $X$, and $N := N(m, q)$ as in Lemma~\ref{lem:combinatorial_I}. Fix 
   $\varepsilon_1 := {\min\{\injrad(M, \bg), 1\}}/{3^m}$ and $ \varepsilon_2 := {\varepsilon_1}/({3^q 5^{2qN} 5^{2N}})$.
    
    For every $x \in K$, let $P_\Phi(x)$ denote a punctate set associated with $\Phi(x)$. By Sard's theorem and Lemma~\ref{lem:disjoint_balls} with $n = 5^{2N}$, there exist $q_x \leq q$, $\tilde P(x) := \{p_{x, i}\}^{q_x}_{i = 1} \subset P_\Phi(x)$ and $r_x \in (\varepsilon_2, 3^q 5^{2qN} \varepsilon_2)$ such that:
    \begin{enumerate}
        \item $\{\overline{B(p_{x, i}, 5^{2N} r_x)}\}$ are disjoint from each other;
        \item For every $k \in \{1, 2, \cdots, 5^{2N}\}$, $\Phi(x) \setminus \bigcup^{q_x}_{i = 1}B(p_{x, i}, k r_x)$ is an embedded smooth surface with boundary and genus at most $g_0$.
    \end{enumerate}
    For every $l \in \{1, 2, \cdots, N\}$, we write
    $ r_{x, l} := 5^{2(l - 1)} r_{x}\,.$

    For every $x \in K$, $p \in \tilde P(x)$ and $l \in \{1, 2, \cdots, N\}$, applying Lemma~\ref{Lem_Loc Deform A} with $\eta = 1$ to $\Phi(x)$ in $A(p; r_{x,l}, 2r_{x,l})$ with suitable $\Lambda$, we obtain the following data:
    \begin{enumerate}
        \item $s_{x,l,p}\in (r_{x,l}, 2r_{x,l})$,
        \item $\zeta_{x,l,p}\in (0, \min\{100 r_{x,l}, s_{x,l}-r_{x,l}, 2r_{x,l}-s_{x,l}\}/5)$,
        \item a neighborhood $O_{x,l,p}\subset X$ of $x$,
        \item and a Simon--Smith family $H_{x,l,p}: [0, 1] \times O_{x,l,p}\to \GS(M)\,.$
    \end{enumerate}
    In addition, for each $x \in K$, we fix a neighborhood of $x$
    \[
        O_x \subset \bigcap_{l \in \{1, 2, \dots, N\},\ p \in \tilde P(x)} O_{x, l, p}\,,
    \]
    such that $y \mapsto \Phi(y) \setminus \bigcup_{p \in \tilde P(x)} B(p, r_x)$ is continuous in the smooth topology for $y \in O_x$; in particular, each $\Phi(y) \setminus \bigcup_{p \in \tilde P(x)} B(p, r_x)$ has genus $\leq g_0$.

\medskip
\paragraph*{\textbf{Step 2. Refinement}} 

    Since $K$ is compact and $\{O_x\}_{x \in K}$ is an open cover, there exists a finite subcover  $\{O_{x_j}\}_{j \in J}\,.$
    Up to refinement of $X$, there exists a cubical subcomplex $Z$ of $X$ such that
    \[
        K \subset \operatorname{int}(Z) \subset Z \subset \bigcup_{j \in J} O_{x_j}\,.
    \]
    For simplicity, we denote $\Phi|_Z$ by $\Psi$, $O_{x_j} \cap Z$ by $U_{x_j}$ for every $j \in J$.
    
    Then we refine $Z$ and assign to each cell $\sigma$ of $Z$ an $x_\sigma \in \{x_j\}_{j \in J}$ such that: Every cell $\tau$ of $Z$ is a subset of $U_{x_\sigma}$ provided that $\sigma$ and $\tau$ are faces of a common cell $\gamma$ of $X$.
    By Lemma~\ref{lem:combinatorial_I} with
    \[
        \mathcal{F}_\sigma = \left(\{A(p; r_{x_\sigma, l}, 2 r_{x_\sigma, l}\}^N_{l = 1} \right)_{p \in \tilde P(x_\sigma)}\,,
    \]
    we obtain, associated with each cell $\sigma$ in $Z$, a collection of annuli denoted by
    \[
        \left(A_{\sigma, p} := A(p; r_{\sigma, p}, 2r_{\sigma, p})\right)_{p \in \tilde P(x_{\sigma})}
    \]
    such that whenever $\sigma$ and $\tau$ are faces of a common cell of $Z$ and for any $p \in \tilde P(x_\sigma)$ and $p' \in \tilde P(x_\tau)$, we have $A_{\sigma, p} \cap A_{\tau, p'} = \emptyset,$
    unless $(\sigma, p) = (\tau, p')$.

\medskip

\paragraph*{\textbf{Step 3. Construction of $H$}}

    Following the same construction as in Step 3 of the proof of Proposition~\ref{Prop_Technical Deformation} and Proposition~\ref{prop:Technical Deformation_II}, without keeping track of the area estimates, one obtain  $H: [0, 1] \times Z \to \GS(M)$
    such that $\Psi' := H(1)$ is a Simon--Smith family of genus $\leq g_0$, and $H$ is a deformation via pinch-off processes from $\Psi$ to $\Psi'$.
    
\appendix
\section{Some geometric measure theory results}\label{sect:GMT}
    In an $N$-dimensional closed Riemannian manifold $(M, \bg)$, let $G_n(M)$ denote its (unoriented) Grassmanian  $n$-plane bundle. Given two $n$-dimensional varifolds $V, W \in \cV(M)$, their $\mathbf{F}$-distance   is defined by
    \[
       \bF(V, W) := \sup\{V(f) - W(f) : f \in C(G_n(M)), |f| \leq 1, \operatorname{Lip}(f) \leq 1\}\,;
    \]
    for a Borel set $B \subset M$, we also define
       $\bF_B(V, W) := \bF(V \llcorner G_n(B), W\llcorner G_n(B))\,,$
    where $G_n(B)$ denotes the restriction of $G_n(M)$ to the base $B$.

    Let us first recall a result by Pitts relating the $\bF$-norm and the $\mathcal{F}$-norm on the space of modulo $2$ integral currents.
    
    \begin{lem}[{\cite[68]{Pit81}.}]\label{lem:Pitts_bF_flat}
        In a closed $N$-dimensional Riemannian manifold $(M, \bg)$, given a sequence of modulo $2$ integral currents $\{S_i\}^\infty_{i = 1}$ in $\mathbf{I}_n(M;\Z_2)$ and $S \in \mathbf{I}_n(M; \Z_2)$, we have 
            $\lim_{i \to \infty} \bF(S_i, S) = 0\,,$
        if and only if
        \[
            \lim_{i \to \infty} \mathcal{F}(S_i, S) = 0\quad\textrm{ and } \quad \lim_{i \to \infty} \mathbf{M}(S_i) = \mathbf{M}(S)\,.
        \]
    \end{lem}

    A direct consequence of this lemma is the following quantitative corollary.

    \begin{cor}\label{cor:Pitts_bF_flat_quant}
        In a closed $N$-dimensional Riemannian manifold $(M, \bg)$, for any $S \in \GS(M)$ and $\varepsilon > 0$, there exists $\delta = \delta(M, \bg, S, \varepsilon) > 0$ such that for any $S' \in \GS(M)$,
        \begin{itemize}
            \item if  $\bF(S', S) < \delta\,,$
        then 
            $\mathcal{F}([S'], [S]) < \varepsilon$ and  $|\mathcal{H}^2(S') - \mathcal{H}^2(S)| < \varepsilon$,
            \item if 
        $\mathcal{F}([S'], [S]) < \delta$ and $|\mathcal{H}^2(S') - \mathcal{H}^2(S)| < \delta\,$
        then  $ \bF(S', S) < \varepsilon\,.$
        \end{itemize}
    \end{cor}
    \begin{proof}
        We argue by contradiction.
        Suppose there exists a sequence $\{S_i\}^\infty_{i = 1}$ in $\GS(M)$ such that $\bF(S_i, S) < \frac{1}{i}$, but $\mathcal{F}([S_i], [S]) \geq \varepsilon$. Then we have  $\lim_{i \to \infty} \bF(S_i, S) = 0$ and $\liminf_{i \to \infty} \mathcal{F}([S_i], [S]) > \varepsilon$,
        contradicting Lemma~\ref{lem:Pitts_bF_flat}. 
        The other direction can be proved similarly.
    \end{proof}

    \begin{lem}\label{lem:boldF_restriction}
        In an $N$-dimensional closed Riemannian manifold $(M, \bg)$, for any $V\in \mathcal{V}_n(M)$, $\varepsilon > 0$ and any family of Borel sets $\{B_\alpha\}_{\alpha \in I}$ such that 
        \[
            \lim_{r \to 0} \sup_{\alpha \in I} V(\overline{N_r(\partial B_\alpha)}) = 0\,,
        \]
        where $\overline{N_r(\partial B_\alpha)} = \{x \in M : \operatorname{dist}_M(x, \partial B_\alpha) \leq r\}$, there exists a constant $\delta = \delta(M, \bg, V, \varepsilon, \{B_\alpha\}_{\alpha \in I})$ satisfying the following property:
        For any $W \in \mathcal{V}_n(M)$, if $\mathbf{F}(V, W) < \delta$, then $\sup_{\alpha \in I} \mathbf{F}_{B_\alpha}(V, W) \leq \varepsilon\,.$
    \end{lem}
    \begin{proof}
        Since $\lim_{r \to 0} \sup_{\alpha \in I} V(\overline{N_r(\partial B_\alpha)}) = 0$, we can pick $r_0 \in (0, 1)$ such that $\sup_{\alpha \in I} V(\overline{N_{r_0}(\partial B_\alpha)}) < \varepsilon / 10$. 
        We set $\delta :=  r_0 \varepsilon/1000\,.$
        
        Firstly, we claim that 
        \[
            \sup_{\alpha \in I} W(\overline{N_{r_0 / 2}(\partial B_\alpha)}) < \varepsilon / 5\,.
        \]
        Suppose by contradiction that there exists $\alpha_0 \in I$ such that $W(\overline{N_{r_0 / 2}(\partial B_{\alpha_0})}) \geq \varepsilon / 5$. We can define a cut-off function $h: M \to [0, 1]$ with $h \equiv 1$ in $\overline{N_{r_0 / 2}(\partial B_{\alpha_0})}$, $g \equiv 0$ outside $\overline{N_{r_0}(\partial B_{\alpha_0})}$ and $\operatorname{Lip}(g) \leq \frac{3}{r_0}$. Then we have
        \[
            \frac{\varepsilon}{10} > V(\overline{N_{r_0}(\partial B_\alpha)}) \geq V(g) \geq W(g) - \mathbf{F}(V, W) \operatorname{Lip}(g) \geq \frac{\varepsilon}{5} - \frac{3\delta}{r_0} \geq \frac{\varepsilon}{10}\,,
        \]
        a contradiction.

        Then fix an arbitrary $\alpha \in I$. For any Lipschitz function $f: M \to \mathbb{R}$ with $\sup |f| \leq 1$ and $\operatorname{Lip}(f) \leq 1$, define 
        \[
            \tilde f(x) = \begin{cases}
                f(x), & x \in B_\alpha\\
                0, & \operatorname{dist}_M(x, \partial B_\alpha) \geq r_0\,.
            \end{cases}
        \]
        We can extend $\tilde f$ to a function defined on $M$ with Lipschitz constant $\frac{1}{r_0}$ and $L^\infty$-norm $1$, still denoted by $\tilde f$. Then we have
        \begin{align*}
            &\quad W\llcorner G_n(B_\alpha)(f) - V\llcorner G_n(B_\alpha)(f)\\
            &= W(\tilde f) - V(\tilde f)   - W\llcorner G_n(N_{r_0}(\partial B_\alpha) \setminus B_\alpha)(\tilde f) + V\llcorner G_n(N_{r_0}(\partial B_\alpha) \setminus B_\alpha)(\tilde f)\\
            &\leq |W(\tilde f) - V(\tilde f)| + \|W\|(N_{r_0}(\partial B_\alpha) \setminus B_\alpha) + \|V\|(N_{r_0}(\partial B_\alpha) \setminus B_\alpha)\\
            &\leq \mathbf{F}(W, V) \operatorname{Lip}(\tilde f) + \|W\|(N_{r_0}(\partial B_\alpha) \setminus B_\alpha) + \|V\|(N_{r_0}(\partial B_\alpha) \setminus B_\alpha)\\
            &\leq \mathbf{F}(W, V)/{r_0} + \varepsilon / 5 < \varepsilon\,.
        \end{align*}
        This implies that $\mathbf{F}_{B_\alpha}(V, W) \leq \varepsilon$.
        Since $\alpha \in I$ is arbitrary, we conclude that $\sup_{\alpha \in I} \mathbf{F}_{B_\alpha}(V, W) \leq \varepsilon$.
    \end{proof}

    \begin{lem}[Isoperimetric inequality, {\cite[6.2]{Federer_Fleming_60_currents}}]\label{lem:isop_ineq}
        In a closed $N$-dimensional Riemannian manifold $(M, \bg)$, there exists $\delta = \delta(M, \bg) > 0$ and $\nu > 0$ such that, if $T \in \mathcal{Z}_n(M; \Z_2)$ and $\mathcal{F}(T) < \delta$, there exists $S \in \mathbf{I}_{n + 1}(M; \Z_2)$ with $\partial S = T$ and $ \mathbf{M}(S) \leq \nu \mathbf{M}(T)^{\frac{n+1}{n}}\,.$
    \end{lem}

\section{A genus inequality}\label{sect:genusIneq}
    \begin{lem}\label{lem:genus_bound_of_close_in_nbhd}
        Let $\Sigma$ be a smooth closed oriented surface, $S\subset \Sigma\times [0, 1]$ be an embedded, smooth, oriented  surface homologous to $\Sigma\times \{0\}$ in $H_2(\Sigma\times [0, 1]; \Z_2)$. Then  $\fg(S) \geq \fg(\Sigma)\,.$
    \end{lem}
    \begin{proof}
        Without loss of generality, we assume $\Sigma$ and $S$ are connected and $S\cap (\Sigma\times \{0\}) = \emptyset$. 
        Consider the projection map onto the first factor $\varpi: \Sigma\times [0,1]\to \Sigma$. Since $S$ is homologous to $\Sigma\times\{0\}$, we claim that the mod $2$ degree of $\pi:=\varpi|_S$ is nonzero. 
        To see this, let $x\in \Sigma$ be a regular value of $\pi$, $D\subset \Sigma\times [0,1]$ be the domain such that $\partial[D] = [S] - [\Sigma\times \{0\}]$ in $\bI_2(\Sigma\times[0,1], \Z_2)$. Then $D\cap \varpi^{-1}(x)$ is a finite union of line segments. Note, exactly one endpoint in $D\cap \varpi^{-1}(x)$ lies in $\Sigma\times \{0\}$, and so there is an odd number of remaining endpoints, which together form the set $\pi^{-1}(x)$.

        Now we shall show that $\pi^*: H^1(\Sigma; \Z_2)\to H^1(S; \Z_2)$ is injective, which immediately implies the genus inequality. Let $\alpha\in \ker\pi^*\cap H^1(\Sigma; \Z_2)$ be arbitrary, for every $\beta\in H^1(\Sigma; \Z_2)$, \[
         \langle \alpha\cup\beta, [\Sigma] \rangle = \langle \alpha\cup\beta, \pi_*[S]\rangle = \langle \pi^*\alpha\cup \pi^*\beta, [S]\rangle = 0\,. 
        \]
        By the nondegeneracy of cup product, it follows that $\alpha = 0$.
    \end{proof}
 
    \begin{proof}[Proof of Lemma \ref{lem:genusLarger}]
        Let $\eta$ be defined as in Corollary~\ref{cor:restr_PS_to_nbhd}, and $\varepsilon = \delta(M, \bg)$ as in Lemma~\ref{lem:isop_ineq}. We also let $r = \delta(M, \bg, \Sigma, \varepsilon)$ as defined in Corollary~\ref{cor:restr_PS_to_nbhd}. 
        
        By Corollary~\ref{cor:restr_PS_to_nbhd}, there exists a punctate surface $S'$ such that:
        \begin{enumerate}[label=(\roman*)]
            \item\label{item:genusLarger_in_nbhd} $S' \subset N_{\eta}(\Sigma)$.
            \item\label{item:genusLarger_F_close} $\mathcal{F}(S' - \Sigma) \leq \mathbf{F}(S', \Sigma) < \varepsilon$.
            \item\label{item:genusLarger_genus_bound} $\mathfrak{g}(S') \leq \mathfrak{g}(S)$.
        \end{enumerate}

        By~\ref{item:genusLarger_in_nbhd}, ~\ref{item:genusLarger_F_close} and Lemma~\ref{lem:isop_ineq}, $S'$ and $\Sigma$ are homologous in $H_2(N_\eta(\Sigma), \Z_2)$.
Since $N_\eta(\Sigma) = \exp^\perp(\Sigma, (-\eta, \eta))$ is diffeomorphic to $\Sigma \times (0, 1)$, by ~\ref{item:genusLarger_genus_bound} and Lemma~\ref{lem:genus_bound_of_close_in_nbhd}, we can conclude that $\mathfrak{g}(S) \geq \mathfrak{g}(S') \geq \mathfrak{g}(\Sigma)\,. $
    \end{proof}

\section{Punctate surfaces with arbitrarily long loops}\label{sect:PS_long_loops}

    In this section, we show that for all $L \in \mathbb{R}^+$, there exists a punctate surface $S$ of genus $1$ in $\mathbb{R}^3$ such that its area is $1$ but any nontrivial loop in $S$ has length at least $L$.

    We start with a flat torus whose meridian has length $L$ and longitude $1/L$. Then we remove $9/10$ of an orbit of slope $2L^2$ to obtain a ``slit torus'', denoted by $T_1$,  as illustrated in Figure~\ref{fig:Torus_Slit}. It is easy to see that every nontrivial loop in the slit torus has length more than $L$ while the slit torus still has area $1$. However, note that this is not a punctate surface in $\mathbb{R}^3$, as it is not embedded in $\mathbb{R}^3$ and the singular set is a slit instead of a finite set of points.

    \begin{figure}
        \centering
        \includegraphics[width=0.4\textwidth]{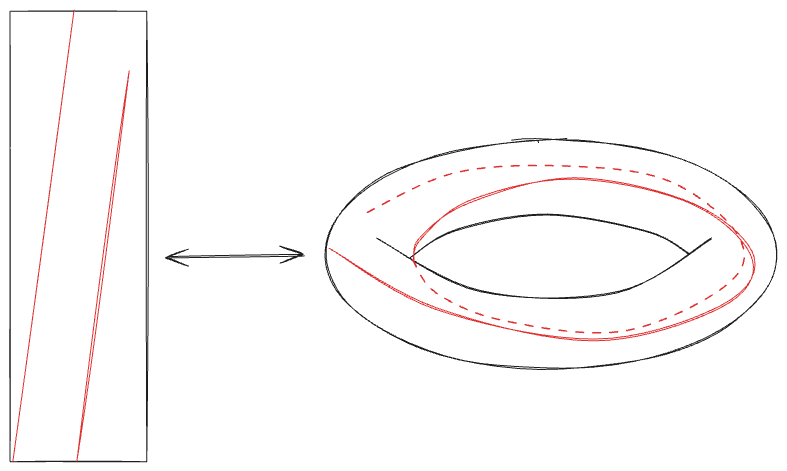}
        \caption{A slit torus}
        \label{fig:Torus_Slit}
    \end{figure} 

    Next, in $\mathbb{R}^3$, we can choose a sufficiently small embedded torus with one point removed, denoted by $T_2$, and a smooth embedding $\Phi: T_1 \to \mathbb{R}^3$ which maps the slit in $T_1$ to the point in $T_2$. Moreover, we require that $\Phi$ is a \emph{strictly short} map in the sense that its Lipschitz constant is less than $1$. Finally, one can follow Nash-Kuiper's spiral-strain construction \cite{Kuiper_1955_C1_isometric} to deform $\Phi$ and, without taking the limit, to obtain a smooth \emph{almost isometry} $\Phi': T_1 \to \mathbb{R}^3$ such that $T'_2 := \Phi'(T_1)$ is still a torus with a point removed. One can also verify that $T'_2$ is a punctate surface that we need. The details are left to the readers. 

\section{Proof of Theorem~\ref{thm:pinchOffMinMax}}\label{sect:proof_pinchOffMinMax}
    In this section, we use Proposition~\ref{Prop_Technical Deformation} and the following local min-max theorem to finish the proof of Theorem~\ref{thm:pinchOffMinMax}. The proof closely follows that in \cite[\S~8.5]{chuLi2024fiveTori}, with a slight improvement: we no longer modify the parameter space. For completeness, we include the details below.

    Let ${\mathbb{B}}^k$ be the open unit $k$-ball, and $\overline{\mathbb{B}}^k$ the closed unit $k$-ball.

    \begin{thm}[Local min-max theorem {\cite[Theorem~6.1]{MN21}}]\label{thm:local-min-max}
        Let $\Sigma$ be a closed, smooth, embedded non-degenerate minimal surface with Morse index $k$ and multiplicity one, in a closed $3$-dimensional manifold $(M, g)$. For every $\beta > 0$, there exists $\varepsilon_0 > 0$ and a smooth family $\{F_v\}_{v \in \overline{\mathbb{B}}^k} \subset \operatorname{Diff}^\infty(M)$ such that
        \begin{enumerate}[label=\normalfont(\arabic*)]
            \item $F_0 = \operatorname{Id}$, $F_v = F^{-1}_v$ for all $v \in \overline{\mathbb{B}}^k$.
            \item $\|F_v - \operatorname{Id}\|_{C^1} < \beta$ for all $v \in \overline{\mathbb{B}}^k$.
            \item The function 
            $A^\Sigma: \overline{\mathbb{B}}^k \to [0, \infty]$ given by $v \mapsto \cH^2((F_v)_\# \Sigma)$
            is strictly concave.
            \item For every $T \in \cZ_2(M; \Z_2)$ with $\mathcal{F}(T, [\Sigma]) < \varepsilon_0$, we have
            \[
                \max_{v \in \overline{\mathbb{B}}^k} \bM((F_v)_\# T) \geq \cH^2(\Sigma)
            \]
            with equality only if $[\Sigma] = (F_v)_\# T$ for some $v \in \overline{\mathbb{B}}^k$.
        \end{enumerate}
    \end{thm}
    
    We proceed with the proof in three steps. 

    \medskip
    
    \paragraph*{\bf Step 1. Set up}
    
    Since $\cW_{L, \leq g}(M, \mathbf{g})$ consists of a varifold associated with a non-degenerate multiplicity-one minimal surface $\Gamma$, we define
     $\delta_0 := \delta(M, \mathbf{g}, \Gamma, r/10), $
    where the function $\delta(\cdot, \cdot, \cdot, \cdot)$ is as defined in Corollary~\ref{cor:Pitts_bF_flat_quant}.
    Let $\varepsilon_0 > 0$ and $\{F_v\}_{v \in \overline{\mathbb{B}}^k}$ be the smooth family associated with $\Gamma$ given by Theorem~\ref{thm:local-min-max} such that:
    \begin{itemize}
        \item[(i)] For every $S_1, S_2 \in \GS(M)$ and $v \in \overline{\mathbb{B}}^k$, 
        $$ 
            \bF(F_v(S_1), F_v(S_2)) \leq 2\bF(S_1, S_2)\quad \textrm{ and }\quad  
            \mathbf{F}(F_v(S_1), S_1) \leq \frac{\delta_0}{10 L} \cH^2(S_1) \,.$$
    \end{itemize} 
    We can choose smaller $\varepsilon_0 > 0$ such that:
    \begin{itemize}
        \item[(ii)] For every $S \in \GS(M)$ with $\mathbf{F}(S, \Gamma) < \varepsilon_0$, the function 
            $A^T: \overline{\mathbb{B}}^k \to [0, \infty]$ given by $v \mapsto \cH^2(F_v(S))$ 
        is strictly concave with a unique maximum in $\mathbb{B}^k_{1/2}$.
        \item[(iii)] We have
        \[
            \varepsilon_1 := \min_{v \in \partial \mathbb{B}^k} (\mathcal{H}^2(\Gamma) - \mathcal{H}^2(F_v(\Gamma))) > 0\,.
        \]
    \end{itemize}
    Let $\mathscr{S} := \{F_v(\Gamma)\}_{v \in \overline{\mathbb{B}}^k}$, which is a compact subset of embeddings into $(M, g)$. We set
    \[
        \delta_1 := \delta(M, \bg, \mathscr{S}, m + 1, N_P(\Phi) + g, \min(\delta_0, \varepsilon_0, \varepsilon_1) / 10)
    \] as defined in Proposition~\ref{Prop_Technical Deformation}.

    \medskip

    \paragraph*{\bf Step 2. Initial Simon--Smith family}
    Let $r_0 := \delta_1/10$.
    It follows from Theorem~\ref{thm:currentsCloseInBoldF} with $r = r_0$ that there exist $\eta \in (0, \delta_0/100)$ and $\Phi_1 \in \Lambda(\Phi)$ such that $\mathcal{H}^2(\Phi_1(x)) \geq L - 4\eta$ implies  $ [\Phi_1(x)] \in \bB^\bF_{r_0}([\Gamma])\,.$
    In particular, there exists a Simon--Smith family  $ H_1: [0, 1] \times X \to \GS(M)$ such that $H_1(0, \cdot) = \Phi$ and $H_1(1, \cdot) = \Phi_1$. In fact, for every $x \in X$, $H_1(\cdot, x)$ is a pinch-off process, because it is induced by a one-parameter group of diffeomorphisms.
    
    We refine the cubical subcomplex $X$ so that, for each cell $\sigma$ of $X$, exactly one of the following conditions holds:
    \begin{itemize}
        \item There exists a point $x_0 \in \sigma$ such that $\mathcal{H}^2(\Phi_1(x_0)) \geq L - 2\eta$; Moreover, in this case, for every cell $\sigma'$ with $\sigma \cap \sigma' \neq \emptyset$ and $x \in \sigma'$, $\mathcal{H}^2(\Phi_1(x)) \geq L - 4\eta$.
        \item For every $x \in \sigma$, $\mathcal{H}^2(\Phi_1(x)) < L - 2\eta$.
    \end{itemize}
    We denote the side length of each cell in $X$ by $\ell$.
    Let $Y$ be the smallest cubical subcomplex of $X$ containing all cells that satisfy the first condition, and let $\tilde Y$ be the cubical subcomplex of $X$ whose underlying set is $\{x \in X : d(x, Y) \leq \ell\}\,.$
    Let $Z := \overline{\tilde Y \setminus Y}$, which is compact. Clearly, for every $z \in Z$, 
    \[
        L - 4\eta \leq \cH^2(\Phi_1(z)) \leq L - 2\eta\,.
    \]

    \medskip

    \paragraph*{\bf Step 3. Deformation}
    
    For every $z \in Z$, let $A^z: \overline{\mathbb{B}}^k \to [0, \infty)$ be the function given  by $v\mapsto \mathcal{H}^2(F_v(\Phi_1(z))).$
    By (ii) of Step 1, $A^z$ is strictly concave and has a unique maximum $m(z) \in \mathbb{B}^k_{1/2}$. As shown in \cite[\S 2]{chuLi2024fiveTori}, the Simon--Smith family $\Phi_1$ is continuous in the $\bF$-metric, so the function $m: Z \to \mathbb{B}^k_{1/2} $
    is also continuous.

    It follows from Theorem~\ref{thm:local-min-max} (4) and $\cH^2(\Phi_1(z)) < L$, that $m(z) \neq 0$ for every $z \in Z$. Hence, there exists $\alpha > 0$ such that  $\alpha \leq |m(z)| < 1/2\,. $
    Consider the one-parameter flow $\{\phi^z(\cdot, t)\}_{t \geq 0} \subset \operatorname{Diff}(\overline{\mathbb{B}}^k)$ generated by
        $v \mapsto -(1 - |v|^2) \nabla A^z(v)\,.$
    For every $v \in \overline{\mathbb{B}}^k \setminus m(y)$, $t \mapsto A^z(\phi^z(v, t))$ is decreasing, and the limit $\lim_{t \to \infty} \varphi^z(v, t) \in \partial \mathbb{B}^k$. In particular, by (iii) of Step 1, we have
    \[
        A^z(0) - \lim_{t \to \infty} A^z(\phi^z(v, t)) \geq \varepsilon_1\,.
    \]
    By the compactness of $Z$, we can choose $T_0 > 0$ such that for every $z \in Z$, $t \mapsto A^z(\phi^z(v, t))$ is decreasing along $[0, T_0]$ and
    \[
        A^z(0) - A^z(\phi^z(v, T_0)) \geq \varepsilon_1 / 2\,.
    \]

    Now, let us define a bump function $\xi: [0, \infty) \to [0, 1]$ by
    \[
        \xi(s) := \begin{cases}
            3s/\ell & s \in [0, \ell/3)\\
            1 & s \in [\ell/3, 2\ell/3)\\
            3 - 3s/\ell & s \in [2\ell/3, \ell)\\
            0 & s \in [\ell, \infty)\,.
        \end{cases}
    \] Then we define a new Simon--Smith family $\Phi_2: X \to \GS(M)$ by
    \[
        \Phi_2(x) := F_{\phi^x(0, \xi(d(x, Y))\cdot T_0)}(\Phi_1(x))\,. 
    \]
    We also define a deformation $H_2: [0, 1] \times X \to \GS(M)$ from $\Phi_1$ to $\Phi_2$ by
    \begin{align*}
        H_2(t, x) &= F_{\phi^x(0, t\cdot\xi(d(x, Y))\cdot T_0)}(\Phi_1(x))\,.
    \end{align*}
    Note that when $x \notin Z$, we abuse the notation by setting $F_{\phi^x(0, 0)} = \operatorname{Id}$.
    
    Similarly, for every $x \in X$, $H_2(\cdot, x)$ is also induced by an isotopy, and thus, a pinch-off process. In particular, we have
    \begin{equation}\label{eqn:NP_estimates}
        N_P(\Phi) = N_P(\Phi_1) = N_P(\Phi_2)\,.
    \end{equation}
    Furthermore, $\Phi_2$ has the following properties: For every $x \in X$, 
    \begin{itemize}
        \item[(i)] if $d(x, Y) > 0$, then  $\cH^2(\Phi_2(x)) \leq L - 2\eta\,;$
        \item[(ii)] if $d(x, Y) \leq 2\ell/3$, by (i) of Step 1, 
            \[
                \bF(\Phi_2(x), \mathscr{S}) \leq 2 \bF(\Phi_1(x), \Gamma) < 2 r_0 < \delta_1\,,
            \]
            \begin{align*}
                \bF(\Phi_2(x), \Gamma) & \leq \bF(\Phi_2(x), \mathscr{S}) + \bF(\mathscr{S}, \Gamma) < \delta_1 + \delta_0/10  \leq \delta_0/5,
            \end{align*}
            and thus, $\cF([\Phi_2(x)], [\Gamma]) \leq \delta_0/10\,;$
        \item[(iii)] if $d(x, Y) \in [\ell/3, 2\ell/3]$, then  $\cH^2(\Phi_2(x)) \leq L - 2\eta - \varepsilon_1 / 2\,.$
    \end{itemize}
    \medskip

    \paragraph*{\bf Step 4. Genus reduction}

    Let $Y' := \{x \in X : d(x, Y) \leq 2\ell/3\}$. By Step 3 (ii), we can apply Proposition~\ref{Prop_Technical Deformation} to $\Phi_2|_{Y'}$ and obtain a deformation 
    $H_{3, Y'}: [0, 1] \times Y' \to \GS(M)\,.$
    Note that by the choice of $\delta$, we can assume $\varepsilon = \min(\delta_0, \varepsilon_0, \varepsilon_1)/10$ in Proposition~\ref{Prop_Technical Deformation}.

    Consequently, we can define a Simon--Smith family $H_3: [0, 1] \times X \to \GS(M)$ as
    \begin{align*}
        H_3(t, x) = \begin{cases}
            H_{3, Y'}(t, x) & d(x, Y) \leq \ell/3\\
            H_{3, Y'}\left((2-3d(x, Y)/\ell) t, x\right) & x \in Y'\\
            \Phi_2(x), & x \notin Y'\,.
        \end{cases}
    \end{align*}
    We denote $H_3(1, \cdot): X \to \GS(M)$ by $\Phi_3$.
    
    By Proposition~\ref{Prop_Technical Deformation} (1), for every $x \in X$ with $d(x, Y) \leq \ell/3$, we have 
  $\fg(\Phi_3(x)) = \fg(H_3(1, x)) = \mathfrak{g}(\Gamma)\,,$
    and
    \[
        \cH^2(\Phi_3(x)) = \cH^2(H_{3,Y'}(1, x)) \leq \cH^2(\Phi_2(x)) + \delta_0/10\,.
    \]
    By Proposition~\ref{Prop_Technical Deformation} (2), Step 3 (ii) and (iii), for every $x \in X$ with $d(x, Y) \in [\ell/3, 2\ell/3]$, we have
    \[
        \cH^2(\Phi_3(x)) = \cH^2\left((2-3d(x, Y)/\ell) t, x\right) \leq \cH^2(\Phi_2(x)) + \varepsilon_1 / 10 \leq L - 2\eta\,.
    \]
    By Step 3 (i), for every $x \in X$ with $d(x, Y) \geq 2\ell/3$,
    \[
        \cH^2(\Phi_3(x)) = \cH^2(\Phi_2(x)) \leq L - 2\eta\,.
    \]
    Therefore, for every $x \in X$ with $\cH^2(\Phi_3(x)) \geq L - \eta$, we have $d(x, Y) \leq \ell/3$, which implies:
    \begin{itemize}
        \item[(i)] $\mathfrak{g}(\Phi_3(x)) = \mathfrak{g}(\Gamma)$;
        \item[(ii)] $|\cH^2(\Phi_3(x)) - L| \leq \delta_0$;
        \item[(iii)] 
            $\cF([\Phi_3(x)], [\Gamma]) \leq \cF([\Phi_3(x)], [\Phi_2(x)]) + \cF([\Phi_2(x)], [\Gamma])
            < \delta_0\,.$
    \end{itemize}
    In particular, by (ii), (iii) and Corollary~\ref{cor:Pitts_bF_flat_quant}, we can conclude that  $\bF(\Phi_3(x), \Gamma) < r\,.$

    Moreover, by Proposition~\ref{Prop_Technical Deformation} (3), for every $x \in X$, $\{H_3(t, x)\}_{t \in [0, 1]}$ is also a pinch-off process.
    Finally, we can set $\Phi' = \Phi_3$, which satisfies (1) of the theorem. The deformation map $H: [0, 1] \times X \to \GS(M)\,,$
    is simply the concatenation of $H_1$, $H_2$ and $H_3$ defined above.
    Clearly, it is a Simon--Smith family of genus $\leq g$ satisfying (2)(a) and (2)(b) of the theorem.
This completes the proof.

\printbibliography

\end{document}